\documentclass{article}
\usepackage[a4paper, margin=3cm, bottom=4cm]{geometry}
\usepackage[utf8]{inputenc}
\usepackage[T1]{fontenc}
\usepackage{lmodern}

\usepackage[
    backend=biber,
    style=alphabetic, maxalphanames=5, 
    maxbibnames=100, maxsortnames=100, 
    sorting=nyt,
    natbib=true,
    url=false,doi=false,isbn=false,eprint=false
]{biblatex}
\usepackage[pdftex]{graphicx} 
\usepackage{subcaption}
\graphicspath{ {./images/} }
\usepackage[linesnumbered,ruled,vlined]{algorithm2e}
\DontPrintSemicolon
\usepackage{enumitem}
\usepackage[normalem]{ulem}
\usepackage[toc,page]{appendix}

\usepackage[bookmarks, bookmarksnumbered,
			colorlinks=true,
            linkcolor = blue,
            urlcolor  = blue,
            citecolor = teal,
            hypertexnames=false  
]{hyperref}

\newcommand\fnsurl[1]{{\footnotesize\url{#1}}}

\usepackage{myquicksetup}
\numberwithin{definition}{section}
\numberwithin{theorem}{section}
\numberwithin{corollary}{section}
\numberwithin{proposition}{section}
\numberwithin{lemma}{section}
\numberwithin{claim}{section}
\numberwithin{fact}{section}
\numberwithin{remark}{section}
\numberwithin{example}{section}
\numberwithin{equation}{section}
\mathtoolsset{showonlyrefs} 

\usepackage{mylivemacros}

\usepackage{sinkhorn_small} 

\newif\ifextended  
\extendedtrue

\title{Effective dynamics of the Sinkhorn algorithm in the regime of low entropy regularization}
\date{\today}

\author{Guillaume Wang\footnote{Courant Institute School, New York
University~ \texttt{guillaume.wang@nyu.edu}}}

\hypersetup{
    pdftitle={Effective dynamics of the Sinkhorn algorithm in the regime of low entropy regularization},
    pdfauthor={Guillaume Wang},
    pdfkeywords={Sinkhorn algorithm, optimal transport, matrix scaling, limiting dynamics}
}

\begin{document}
\maketitle


\begin{abstract}
    The Sinkhorn algorithm is the de facto standard method for numerically solving entropy-regularized optimal transport problems over finite sets.
    In this work, we investigate a phenomenon arising when Sinkhorn is applied with a small regularization parameter $\tau$: the evolution of the dual variables (the logarithm of the scaling factors) is approximately piecewise-linear, while the primal variables (the approximate transport plans) exhibit a saddle-to-saddle type behavior.
    We prove that as $\tau \to 0$, the Sinkhorn iterates indeed converge to a continuous-time curve consistent with these observations, when time is rescaled as $t = \tau k$, and we characterize the limiting ``cold Sinkhorn'' dynamics explicitly.
    In particular, we show that it acts as a dual optimization dynamics for the unregularized problem with properties analogous to the simplex algorithm.
    Notably, this dynamics converges in finite time to an unregularized solution,
    implying a novel guarantee for the Sinkhorn algorithm itself: it achieves
    $\widetilde{O}(\tau)$ dual suboptimality in $k = O(\tau^{-1})$ iterations, instead of $k = O(\tau^{-2})$ as existing analyses would suggest.
\end{abstract}

\section{Introduction} \label{sec:intro}

Consider the discrete optimal transport (OT) problem with target marginals $\mu \in \Delta_m$ and $\nu \in \Delta_n$---where $\Delta_m$ denotes the probability simplex in dimension $m$---and with transport cost matrix $C \in \RR^{m \times n}$:
\begin{equation} \label{eq:intro:OT}
    \min_{\pi \in \Delta_{m \times n}} \sum_{ij} C_{ij} \pi_{ij}
    ~~~~\text{subject to}~~~~
    \begin{cases}
        \forall i \leq m,~ \sum_j \pi_{ij} = \mu_i \\
        \forall j \leq n,~ \sum_i \pi_{ij} = \nu_j.
    \end{cases}
\end{equation}
Optimization problems of this form arise throughout machine learning and data science \cite{peyre2019computational,chewi2025statistical}, computer vision \cite{rubner2000earth}, computational biology \cite{schiebinger2019optimal}, and economics \cite{galichon2016optimal}.
Closely related is the problem of entropic optimal transport (EOT) with an entropy regularization parameter, or temperature, $\tau>0$:
\begin{equation} \label{eq:intro:EOT}
    \min_{\pi \in \Delta_{m \times n}} \sum_{ij} C_{ij} \pi_{ij} + \tau \Hdiv{\pi}{\mu \otimes \nu}
    ~~~~\text{subject to}~~~~
    \begin{cases}
        \forall i \leq m,~ \sum_j \pi_{ij} = \mu_i \\
        \forall j \leq n,~ \sum_i \pi_{ij} = \nu_j.
    \end{cases}
\end{equation}
Here $\Hdiv{\pi}{\mu \otimes \nu} = \sum_{ij} \pi_{ij} \log \frac{\pi_{ij}}{\mu_i \nu_j} \geq 0$ denotes the relative entropy between discrete probability distributions.
EOT can be used as a proxy for unregularized OT for computational purposes \cite{cuturi2013sinkhorn}, but it is also an important problem in its own right due to its connection to the Schr\"odinger bridge problem \cite{leonard2013survey}, with applications in generative modeling \cite{debortoli2021diffusion,pooladian2025plug} and trajectory inference \cite{lavenant2021towards,chizat2022trajectory}.

Over the past decades, there has been much research activity around designing and analyzing optimization algorithms for these problems.
Algorithms for OT \eqref{eq:intro:OT} can be roughly divided into two categories: those using linear programming or combinatorial approaches, and those leveraging methods for EOT with a well-chosen $\tau$.
Indeed for any prescribed tolerance $\eps$, an $\eps$-minimizer for \eqref{eq:intro:OT} can be obtained by exactly solving \eqref{eq:intro:EOT} with
$
    \tau 
    = \frac{\eps}{\max_{\pi} \Hdiv{\pi}{\mu \otimes \nu})} 
    = \frac{\eps}{-\log (\mu_{\min} \nu_{\min})}
$
where $\mu_{\min} = \min_i \mu_i$ and $\nu_{\min} = \min_j \nu_j$.%
\footnote{
    One can also use an approximate solution of \eqref{eq:intro:EOT}, naturally, or even an approximate solution of its dual thanks to the rounding procedure of \cite{altschuler2017near}; see Related works for details.
    The fact that $\max_{\pi \in \Delta_{m \times n}} \Hdiv{\pi}{\mu \otimes \nu} = -\log (\mu_{\min} \nu_{\min})$ can be seen by noting that by convexity of relative entropy, the $\max$ must be attained at an extremal point of $\Delta_{m \times n}$, i.e., at some $\pi = \ind_{(i_0,j_0)}$.
    As also shown in \cite{altschuler2017near}, in fact it is sufficient to use $\tau = \frac{\eps}{\log(mn)}$, because changing $\Hdiv{\pi}{\mu \otimes \nu}$ to $\Hdiv{\pi}{\frac{\bmone_m}{m} \otimes \frac{\bmone_n}{n}}$ in \eqref{eq:intro:EOT} does not change the problem's optimal solution.
}
While the theoretical state-of-the-art computational complexity is at present only attained by algorithms from the first category \cite{vandenbrand2020bipartite}, 
methods from the second category are equally prevalent in practice \cite{peyre2019computational}.

For the EOT problem \eqref{eq:intro:EOT}, one method stands out: the Sinkhorn algorithm, recalled in \autoref{sec:prelim}.
On the methodology side, this simple yet remarkably efficient algorithm is currently the de facto standard for this problem, as witnessed by its use in all mainstream computational optimal transport libraries \cite{feydy2019interpolating,flamary2021pot,cuturi2022optimal}.
It also proves highly adaptable, with numerous variants tailored to large-scale computations \cite{altschuler2019massively,scetbon2020linear,scetbon2021low}, streaming data \cite{mensch2020online,wang2023compressed}, more than two marginals \cite{benamou2015iterative,carlier2022linear}, or unbalanced optimal transport settings \cite{chizat2016scaling,sejourne2022faster}.
On the theoretical side, much effort has been devoted to finely understanding this algorithm's convergence properties, both in the general-cost discrete setting considered here \cite{knight2008sinkhorn,altschuler2017near,dvurechensky2018computational,leger2021gradient,aubin2022mirror,chizat2024annealed,ghosal2025convergence} and in structured continuous settings (EOT between continuous probability measures $\mu, \nu$ with a smooth transport cost function) \cite{conforti2023quantitative,chizat2026sharper,chiarini2024semiconcavity}.
Since the Sinkhorn algorithm is our main object of study in this work, we defer a brief discussion of alternative algorithms for EOT computation to Related works below.

In this work, we study the behavior of the Sinkhorn algorithm in the regime of small entropy regularization $\tau$.
As a motivating example, we display in \autoref{fig:intro:illustrative} the results of a simple numerical experiment where $m=5$, $n=6$, and the Sinkhorn algorithm is run with $\tau=0.005$.
Strikingly, we observe that%
\footnote{Python code to reproduce the experiments is publicly available online at \url{https://github.com/guillaumew16/cold-sinkhorn}.}
\begin{itemize}
    \item (\autoref{fig:intro:illustrative_a}) 
    The suboptimality of the Sinkhorn iterates evolves approximately piecewise linearly when measured by the dual EOT objective $\Psi(f^k, g^k)$.
    It also evolves in an approximately piecewise constant manner when measured by the distance between the primal variable's marginals and the target marginals $\mu, \nu$, consistent with \cite[Lemma~2]{altschuler2017near}.
    \item (\autoref{fig:intro:illustrative_b})
    The dual Sinkhorn iterates $(f^k, g^k)$ themselves evolve approximately piecewise linearly, i.e., their one-iteration increments are approximately piecewise constant.
    Moreover, at any iteration $k$ outside of the ``phase transitions'', the $m$ components of the increment $(f^k_i-f^{k-2}_i)_{i \leq m}$ collapse onto a number of values smaller than $m$, and likewise for the increment of $g^k$.
    Furthermore, the set of values taken by $(f^k_i-f^{k-2}_i)_{i \leq m}$ then approximately coincides with that taken by $(-g^k_j+g^{k-2}_j)_{j \leq n}$.
    \item Experimenting with even smaller values of $\tau$ reveals that, when time is rescaled as $t = k \cdot \tau$, the iterates $(f^{\floor{t/\tau}}, g^{\floor{t/\tau}})$ appear to converge to a truly piecewise linear curve in $t$.
    That is, the ``jumps'' in \autoref{fig:intro:illustrative_b} become sharper as $\tau$ decreases.
    \item Experimenting with other cost matrices $C$, the marginals $\mu, \nu$ being kept fixed, reveals that the set of values that can be taken by the one-iteration increments $f^k_i-f^{k-2}_i$, $-(g^k_j-g^{k-2}_j)$ is finite and independent of $C$.
    \item The corresponding primal variables $\pi^k_{ij} = e^{[-C_{ij}+f^k_i+g^k_j]/\tau} \mu_i \nu_j$, which are known to converge to an optimal solution of \eqref{eq:intro:EOT}, evolve in an approximately piecewise constant manner.
    They are not represented in \autoref{fig:intro:illustrative}, but we refer to \autoref{fig:CS:CS} for a plot of their evolution in a rescaled log-domain, on a smaller example.
\end{itemize}
Two other similar experiments are also presented in \autoref{sec:apx_addtl_exps}, with $m=n=50$ and $m=n=400$ respectively.

These phenomena have not been reported previously in the literature, to our knowledge, despite their generality across all of our numerical experiments. This may be due to the fact that the piecewise-linear phases tend to be shorter, and can become indiscernible compared to the algorithm's total runtime until approximate convergence, when $m$ and $n$ are large.

\begin{figure}[t]
    \centering
    \begin{subfigure}[t]{0.349 \textwidth}
        \centering
        \includegraphics[width=0.992\linewidth]{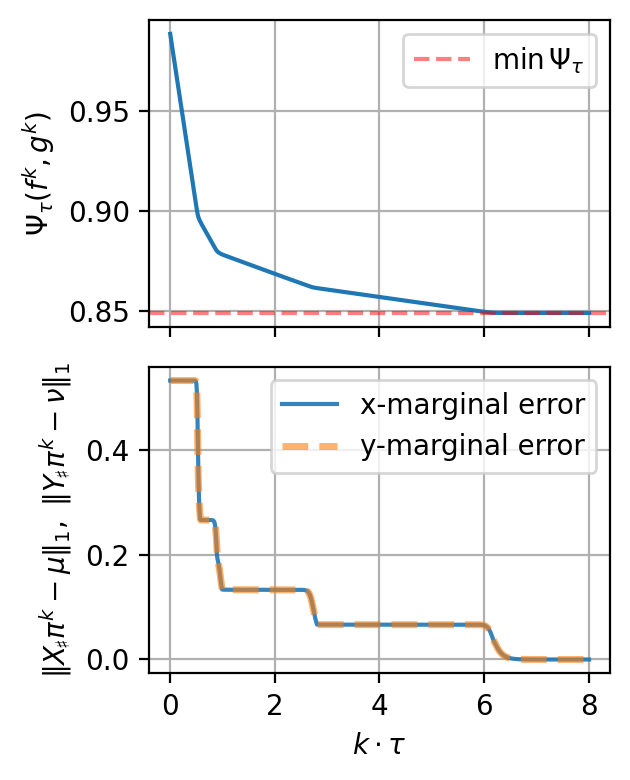} 
        \caption{Suboptimality in dual objective~$\Psi_\tau$ \eqref{eq:prelim:dualobj} and 
        in $\ell_1$-norm error of the marginals.}
        \label{fig:intro:illustrative_a}
    \end{subfigure}
    \hfill
    \begin{subfigure}[t]{0.622\textwidth}
        \centering
        \includegraphics[width=1.01\linewidth]{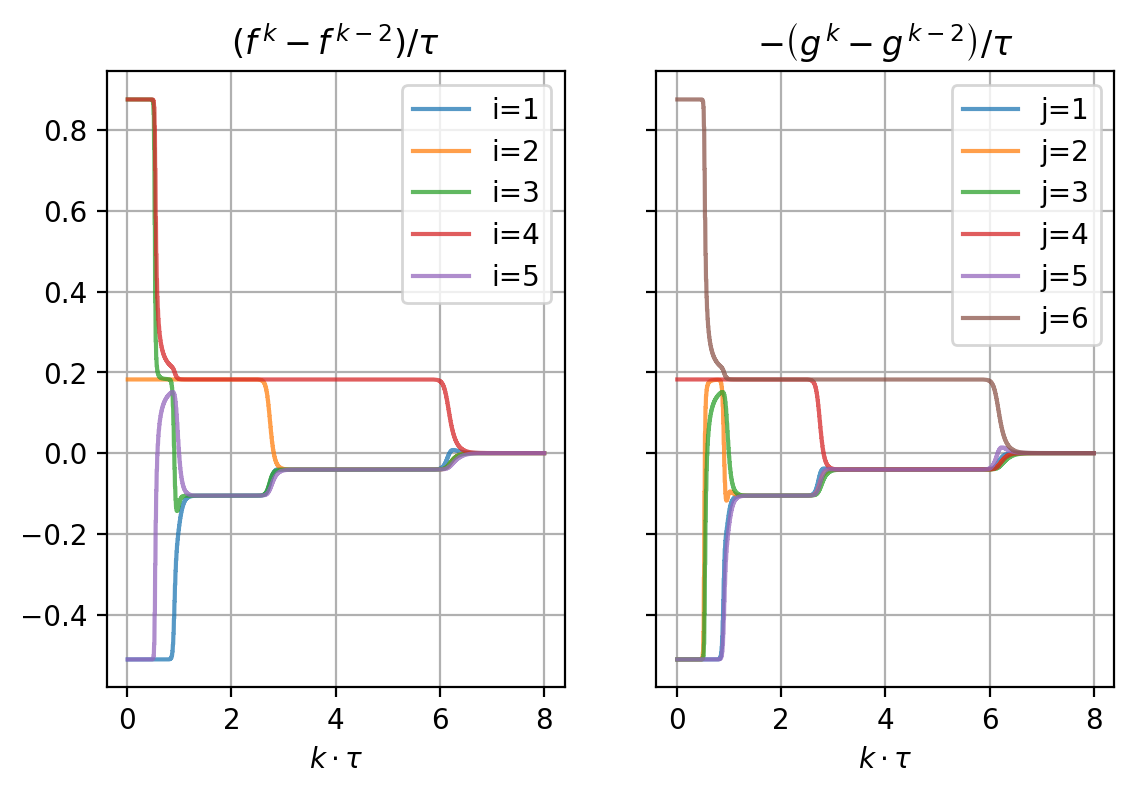}
        \caption{Evolution of the Sinkhorn iterates $f^k, g^k$ represented by their one-iteration increments, rescaled by $\tau$. Note that the increments of $g^k$ are represented with a negative sign.}
        \label{fig:intro:illustrative_b}
    \end{subfigure}
    \caption{Behavior of the Sinkhorn algorithm on \eqref{eq:intro:EOT} for $m=5, n=6$, uniform target marginals, and a generic cost matrix (drawn randomly with i.i.d.\ Gaussian entries), with $\tau = 0.005$.
    The x-axes represent iteration count scaled by $\tau$. If a smaller $\tau$ is used, the figures look almost identical, only with sharper transitions between the piecewise-linear phases.}
    \label{fig:intro:illustrative}
\end{figure}

The purpose of this work is to explain these phenomena and investigate their implications. Specifically, \textbf{our contributions are as follows.}
\begin{itemize}
    \item We prove that as $\tau \to 0$, the time-rescaled Sinkhorn iterates $(f^{\floor{2 t/\tau}}, g^{\floor{2 t/\tau}})$ converge to a continuous curve $(f(t), g(t))$ which we describe explicitly (\autoref{def:CS:formaldef}, \autoref{thm:cv:cv}).
    In particular, all of the phenomena observed above are explained.
    This limiting curve acts as a continuous-time dual optimization dynamics for the unregularized OT problem, termed the \emph{cold Sinkhorn} dynamics, with properties analogous to the simplex algorithm (\autoref{lm:CS:ftgt_partialFF}--\autoref{lm:CS:t_L=O(1)}).
    \item We show that the cold Sinkhorn dynamics converges in finite time, after a finite number of piecewise linear phases, to a solution of the dual OT problem (\autoref{lm:CS:finalphase_optimal}).
    This implies a novel guarantee for the Sinkhorn algorithm itself: it achieves $\tO(\tau)$ dual suboptimality and $\tO(\sqrt{\tau})$ $\ell_1$-norm marginal error in $O(1/\tau)$ iterations (\autoref{thm:cv:improved_guarantee}, \autoref{coroll:cv:improved_guarantee_Ek}), where $\tO(\cdot)$ hides constants dependent on $\mu, \nu$, and $C$ as well as logarithmic factors.
    This generalizes \cite[Corollary~1.3]{berman2020sinkhorn} to the general-cost, finite-domain case and shows that geometric structure is actually not needed for such an estimate to hold.
    \item 
    We also describe the corresponding evolution of the primal variables explicitly. Namely, we show that as $\tau \to 0$, $(\pi^{2\floor{t/\tau}}, \pi^{2\floor{t/\tau}+1})$ converges to a piecewise-constant curve in $t$, and we characterize the ``saddle points'' as the minimizers of a certain sequence of EOT problems (\autoref{thm:pik:even_odd}).
    \item As an ingredient for the proof of our main results, we generalize the best-known convergence guarantees for the Sinkhorn algorithm to the non-scalable case, which may be of independent interest (\autoref{prop:prelim:qtve_BV24}, \autoref{prop:prelim:qtve_BV24_exp}).
\end{itemize}

\subsection{Related work}\label{subsec:intro:relwork}

\paragraph{Convergence analyses for the Sinkhorn algorithm.}
The convergence behavior of the Sinkhorn algorithm has been the topic of much research interest since its introduction for the purpose of matrix scaling by \cite{sinkhorn1967concerning},
and even more so since its advantageous properties for computational optimal transport were put forward by \cite{cuturi2013sinkhorn}.
For the standard finite-domain EOT problem as presented in \eqref{eq:intro:EOT}, local and global exponential convergence bounds were established by \cite{sinkhorn1967diagonal,franklin1989scaling,soules1991rate,knight2008sinkhorn,qu2025sinkhorn}; the rates in the last reference are sharp. 
This line of work reveals that the exponential rate (one minus the contraction coefficient) scales as $e^{-\Theta(1/\tau)}$ in the worst case \cite[Remark~4.15]{peyre2019computational}.
In parallel, polynomial convergence bounds were established by \cite{kalantari2008complexity,altschuler2017near,dvurechensky2018computational,lin2022efficiency,ghosal2025convergence}, leading to bounds that scale more gracefully when $\tau$ is small, in $O(1/\sqrt{\tau k})$ or $O(1/(\tau k))$.
For EOT in ``structured'' continuous-domain cases, the exponential rate's scaling can be much more favorable: $\Theta(\tau)$ instead of $e^{-\Theta(1/\tau)}$, as shown by~\cite{chizat2026sharper}.

Besides computational optimal transport, analyzing the Sinkhorn algorithm is also of interest for matrix scaling.
Indeed, solving the dual EOT problem (\eqref{eq:prelim:dualobj} below) is equivalent to finding diagonal matrices $D, D'$ such that $D A D'$ has row-sums $\mu_1, ..., \mu_m$ and column-sums $\nu_1, ..., \nu_n$, where $A = \left( e^{-C_{ij}/\tau} \right)_{ij}$ \cite{idel2016review}.
In this context, it is desirable to also treat the case where $A$ can have zero entries, corresponding to $C$ having infinite entries---which is rarely of interest in computational optimal transport, but will turn out to be relevant for our investigation in this paper.
Our analysis will rely on results in this direction by \cite{baradat2024convergence,qu2025sinkhorn}, as well as \cite{kalantari2008complexity,allen2017much} via~\cite{wang2026almost}.

\paragraph{Saddle-to-saddle behavior in gradient-based optimization.}
While most convergence analyses of the Sinkhorn algorithm are relatively ad hoc, let us highlight a line of work based on interpreting Sinkhorn as an instance of mirror descent, allowing for a connection to the broader optimization literature \cite{leger2021gradient,aubin2022mirror,karimi2024sinkhorn}.
Two variants of this interpretation exist: in the first, the variables of mirror descent are the primal variables' marginals $(\sum_j \pi^k_{ij})_i \in \Delta_m$ \cite{leger2021gradient}; in the second, they are the primal variables $\pi^k$ themselves \cite{aubin2022mirror}.
In terms of both of these variants, the phenomenon investigated in our work corresponds to a \emph{saddle-to-saddle behavior} of the algorithm, as explained in \autoref{sec:pik}.
Saddle-to-saddle behaviors in first-order optimization dynamics were studied in special cases by \cite{jacot2021saddle,berthier2023incremental,pesme2023saddle,berthier2025diagonal}.
Very recently,
a general analysis for mirror flow on quadratic objectives was performed by \cite{berthier2026incremental}, with essentially identical phenomenology as ours,
indicating that their results may extend beyond quadratics.

\paragraph{Other algorithms for EOT computation.}
As mentioned in the introduction, many variants of the Sinkhorn algorithm have been proposed. From the point of view of convergence guarantees, some are more amenable to analysis than others, notably the Greenkhorn algorithm 
\cite{altschuler2017near,lin2022efficiency,abid2018stochastic,kostic2022batch}
and the overrelaxed and damped Sinkhorn algorithms \cite{thibault2021overrelaxed,lehmann2022note,vaskevicius2023computational}.

Besides the Sinkhorn algorithm and its variants, another remarkably simple and effective approach to solving the EOT problem \eqref{eq:intro:EOT} is to apply gradient-based algorithms to its semi-dual formulation \cite{cuturi2018semidual}, i.e.,
to $\min_{g \in \RR^n} \tPsi_\tau(g)$ where 
$\tPsi_\tau(g) 
= \min \Psi_\tau(\cdot, g) 
= \tau \sum_i \mu_i \log\big( \sum_j e^{[-C_{ij}+g_j]/\tau} \nu_j \big) - \nu^\top g$.
This idea goes back at least to \cite{kosowsky1994invisible}, who propose using (various time-discretizations of) gradient flow for $\tPsi_\tau$, under the name ``invisible hand algorithm''.
This approach has recently received renewed attention in the context of semi-discrete optimal transport, where the marginal $\mu$ is actually a probability density function over a continuous set, because the semi-dual objective can then be put in the form of an expectation over $i \sim \mu$ and is thus amenable to stochastic gradient descent \cite{genevay2016stochastic,mousavi2025flow,genans2026decreasing}.
This approach, being closer to the standard framework of gradient methods for optimization compared to Sinkhorn, has also recently inspired momentum-based algorithms with provably fast polynomial rates of convergence \cite{xie2022accelerated,luo2023improved}.

We note that experimentally, Greenkhorn,%
\footnote{However we also note that there does not seem to be a numerically stable way to implement Greenkhorn for small $\tau$ while preserving the $O(m \vee n)$ per-iteration cost.}
damped Sinkhorn, and gradient flow on the semi-dual all exhibit the same piecewise-linear evolution phenomenon for small $\tau$ as vanilla Sinkhorn.
Explaining these observations theoretically 
is left for future work.

\paragraph{Reducing unregularized OT computation to dual EOT computation.}
The authors of \cite{altschuler2017near} proposed an approach to unregularized OT computation which proved quite influential (see the aforecited \cite{dvurechensky2018computational,lin2022efficiency,xie2022accelerated,luo2023improved,li2025fast}, among others),
and which we now briefly review.
As explained in the introduction, approximately solving the OT problem \eqref{eq:intro:OT} can be reduced to approximately solving the EOT problem \eqref{eq:intro:EOT} with a small enough choice of $\tau$.
On the other hand, the Sinkhorn algorithm yields approximate solutions not for \eqref{eq:intro:EOT} but only for the \emph{dual} EOT problem, \eqref{eq:prelim:dualobj} below.
Fortunately, this is still sufficient thanks to the rounding algorithm of \cite[Alg.~2]{altschuler2017near}, which incurs negligible computational overhead and comes with the following guarantee:
if dual EOT with $\tau = \frac{\eps}{\log(m n)}$ is solved up to a $\ell_1$-norm marginal error of $\frac{\eps}{\norm{C}_\infty}$, then the rounding algorithm yields a feasible and $O(\eps)$-optimal solution for primal unregularized OT.

The approach to OT computation proposed by \cite{altschuler2017near} is thus to approximately solve dual EOT with a small $\tau$, and to apply their rounding algorithm as a post-processing step.
When the Sinkhorn algorithm is used for the first step, our results suggest that this effectively amounts to running the cold Sinkhorn dynamics for the unregularized OT problem.

\paragraph{Effective dynamics of Sinkhorn for squared-Euclidean transport costs.}
Consider the case where $\mu, \nu$ arise from the space-discretization of some smooth and compactly-supported probability density functions on $\RR^d$ onto some fixed grid $\{x_1, ..., x_n\}$, and where the transport costs are given by $C_{ij} = \norm{x_i-x_j}^2$.
Call $h$ the edge length of the grid.
Then it was shown by \cite[Theorem~1.2]{berman2020sinkhorn} that as $\tau, h \to 0$ jointly with $h \lesssim \tau^{1/2+\eps}$ for any $\eps>0$, the Sinkhorn iterates $f^{\floor{2t/\tau}}, g^{\floor{2t/\tau}}$ converge to smooth functions $f(t, x), g(t, x)$ over $\RR_+ \times \RR^d$.
The proved convergence is pointwise in time and uniform in space (recall the marginals are compactly supported), and the limiting function $f(t,x)$ is characterized as the solution of a PDE called the parabolic Monge-Amp\`ere equation (and symmetrically for $g(t,x)$).
Combined with the fact that the parabolic Monge-Amp\`ere equation converges in time to an optimal Kantorovich potential $f^*(x)$---i.e., $(f^*, g^*)$ is an optimal solution of the continuous-space dual OT problem for some $g^*$---this implies that
$\max_i \abs{f^k_i - f^*(x_i)} \lesssim \tau \log(1/\tau)$ 
after $k \gtrsim \tau^{-1} \log(1/\tau)$ iterations \cite[Corollary~1.3]{berman2020sinkhorn}.

These results were refined by \cite{deb2023wasserstein} in the fully continuous-domain setting, i.e., assuming $\mu, \nu$ themselves are probability densities, corresponding to $h=0$.
Namely, it was shown there that the corresponding primal variables' marginals also converge to smooth curves, that~is,
$\int_y e^{\frac1\tau\left[ -\norm{x-y}^2+f^{2\floor{t/\tau}}(x)+g^{2\floor{t/\tau}}(y) \right]} \mu(\d x) \nu(\d y)$
converges to an absolutely continuous probability measure $\mu_t(\d x)$ for all $t$ (and symmetrically for the second marginal).
This is equivalent to convergence of the rescaled increments of $f^{\floor{2t/\tau}}(x)$ to the time-derivatives of $f(t,x)$, as can be seen from~\eqref{eq:prelim:rel_vk_Xpik} below.
Additionally, this reference proposes an interpretation of the parabolic Monge-Amp\`ere equation as a mirror flow for the $\mu_t$ in probability space, by leveraging the mirror descent interpretation of Sinkhorn from \cite{leger2021gradient}.

On the one hand, our results can be viewed as analogs of those of \cite{berman2020sinkhorn,deb2023wasserstein} for the finite-domain case without geometric structure---though in our case, interpreting the cold Sinkhorn dynamics as a mirror flow seems to be impossible (\autoref{rk:pik:noclosedform}).
On the other hand, our results apply in the fixed-grid setting described above whenever $\tau \leq \tau_0$, for some constant $\tau_0 = \tau_0(\mu, \nu, h)$ that could be determined from our analysis, to be contrasted with the regime $h \lesssim \tau^{1/2+\eps}$ studied by \cite{berman2020sinkhorn}.
So it remains to determine quantitatively the regime $\tau_0(\mu, \nu, h)$ where our results apply,
and to study the limiting behavior of the Sinkhorn algorithm in the intermediary regime
$h^{-1/2-\eps} \gg \tau > \tau_0(\mu, \nu, h)$;
we leave these questions open for future research.


\vspace{1em}
The paper is organized as follows.
In \autoref{sec:prelim}, we present preliminary definitions and facts around the Sinkhorn algorithm---some of which are new, presented in \autoref{subsec:prelim:auxlemmas}, \autoref{subsec:prelim:qtve_BV24}.
In \autoref{sec:CS}, we introduce and analyze the convergence of the cold Sinkhorn dynamics.
In \autoref{sec:cv}, we prove that the Sinkhorn iterates converge to the cold Sinkhorn dynamics at a rate $\tO(\tau)$, and deduce a new convergence guarantee for the Sinkhorn algorithm.
In \autoref{sec:pik}, we spell out the corresponding limiting dynamics for the primal variables $\pi^k$.
We conclude in \autoref{sec:ccl} with perspectives and possible directions for future work.

\section{Preliminaries} \label{sec:prelim}

\subsection{Background on (entropic) optimal transport and Sinkhorn algorithm} \label{subsec:prelim:bg}

All of the notions reviewed in this section are standard in computational optimal transport \cite{peyre2019computational}.

\paragraph{The dual OT problem.}
The dual of the unregularized OT problem \eqref{eq:intro:OT}, viewed as a linear program, is
\begin{equation} \label{eq:prelim:dualOT}
    \max_{f \in \RR^m,\, g \in \RR^n}
    f^\top \mu + g^\top \nu
    ~~~~\text{subject to}~~~~
    \forall i, j,~
    f_i + g_j \leq C_{ij}.
\end{equation}
For ease of presentation, we define
the minimization objective
and the feasible set
\begin{equation} \label{eq:prelim:def_FF}
    \Psi_0(f, g) = -f^\top \mu - g^\top \nu,
    \qquad\qquad
    \FF = \left\{ (f, g) \in \RR^m \times \RR^n ;~~ \forall i,j,~ f_i + g_j \leq C_{ij} \right\},
\end{equation}
so that the dual OT problem rewrites
$\min_{f,g} \Psi_0(f, g)$ subject to $(f,g) \in \FF$.
We also introduce the operators, called $C$-transform resp.\ $\ol C$-transform in the optimal transport literature,
\begin{equation}
    \forall g \in \RR^n,~
    f_0[g]_i = \min_j \, C_{ij} - g_j
    \qquad \text{and} \qquad
    \forall f \in \RR^m,~
    g_0[f]_j = \min_i \, C_{ij} - f_i.
\end{equation}
We will repeatedly use the following characterization in the sequel.

\begin{lemma} \label{lm:prelim:partialFF}
    Define the Pareto frontier of $\FF$ as the subset $\partial \FF$ such that
    \begin{equation}
        \forall (f,g) \in \FF,~~~~
        (f,g) \in \partial \FF \iff
        \left\{ (f',g') \in \FF \text{ s.t. } \forall i, f'_i \geq f_i, \forall j, g'_j \geq g_j \right\} = \{ (f,g) \}.
    \end{equation}
    For any $f \in \RR^m, g \in \RR^n$, the following conditions are equivalent: 
    \begin{itemize}[itemsep=1pt]
        \item $(f, g) \in \partial \FF$.
        \item $f = f_0[g]$ and $g = g_0[f]$.
        \item $f_i + g_j \leq C_{ij}$ for all $(i, j)$, and the bipartite graph $(\{1\dots m\} \sqcup \{1 \dots n\}, \EEE)$ with edge set $\EEE = \left\{ (i,j);~ f_i+g_j = C_{ij} \right\}$ has no isolated vertex.
    \end{itemize}
\end{lemma}

\paragraph{The dual EOT problem.}
The dual of the EOT problem \eqref{eq:intro:EOT}, viewed as a convex optimization problem, is
\begin{equation}
    \max_{f \in \RR^m,\, g \in \RR^n}
    -\Psi_\tau(f,g)
    ~~\equiv~~
    \min_{f \in \RR^m,\, g \in \RR^n} \Psi_\tau(f,g)
\end{equation}
where, for ease of presentation, we define
\begin{equation} \label{eq:prelim:dualobj}
    \Psi_\tau(f, g) =
    \tau \bigg( \sum_{ij} e^{\left[ -C_{ij} + f_i + g_j \right]/\tau} \mu_i \nu_j - 1 \bigg)
    - f^\top \mu - g^\top \nu.
\end{equation}
The KKT stationarity condition linking primal variables $\pi \in \Delta_{m \times n}$ and dual variables $(f, g) \in \RR^m \times \RR^n$ is:
$\forall i,j,~ 
f_i + g_j = \tau \log \frac{\pi_{ij}}{\mu_i \nu_j} + C_{ij} + c$
for some normalizing constant $c \in \RR$,
or equivalently,
$\pi = \pi_\tau[f,g]$ where we define
\begin{equation}
    \pi_\tau[f,g]_{ij}
    = \frac{1}{Z_\tau(f,g)} e^{\left[ -C_{ij} + f_i + g_j \right]/\tau} \mu_i \nu_j
    \quad \text{where} \quad
    Z_\tau(f,g) = \sum_{i'j'} e^{\left[ -C_{i'j'} + f_{i'} + g_{j'} \right]/\tau} \mu_{i'} \nu_{j'}.
\end{equation}

\paragraph{The Sinkhorn algorithm.}
For a given initial pair $(f^0, g^0)$, typically $(0,0)$, the Sinkhorn iterates $(f^k, g^k)_{k \geq 0}$ are defined by the update rule
\begin{align}
    \text{for $k \geq 0$ even,}~~~
    &f^{k+1} = f_\tau[g^k]
    \quad \text{and} \quad
    g^{k+1} = g^k \\
    \text{for $k$ odd,}~~~~
    &f^{k+1} = f^k
    \quad\quad~\, \text{and} \quad
    g^{k+1} = g_\tau[f^k]
\end{align}
where
$f_\tau[g] = \argmin \Psi_\tau(\cdot, g)$
and
$g_\tau[f] = \argmin \Psi_\tau(f, \cdot)$.
More explicitly,
\begin{equation} \label{eq:prelim:def_f[g]_g[f]}
    \forall g \in \RR^n, \,
    f_\tau[g]_i = -\tau \log \sum_j e^{\left[ -C_{ij} + g_j \right]/\tau} \nu_j
    ~~~\quad \text{and} \quad~~~
    \forall f \in \RR^m, \,
    g_\tau[f]_j = -\tau \log \sum_i e^{\left[ -C_{ij} + f_i \right]/\tau} \mu_i.
\end{equation}
\pagebreak
These operators are sometimes called the soft-$C$-transform resp.\ soft-$\ol C$-transform in the literature.
We also set $\pi^k = \pi_\tau[f^k, g^k]$ for all $k \geq 0$.
Note that by explicit computations,
\begin{equation}
    \forall g,~ Z_\tau(f_\tau[g], g) = 1
    \qquad \text{and} \qquad
    \forall f,~ Z_\tau(f, g_\tau[f]) = 1,
\end{equation}
so that $Z_\tau(f^k, g^k) = 1$ and $\pi^k_{ij} = e^{\left[ -C_{ij} + f^k_i + g^k_j \right]/\tau} \mu_i \nu_j$ for all $k \geq 1$ (but not for $k=0$ in general).

Note that the iterates $(f^k, g^k)$ for $k \geq 1$ are determined solely by $g^0$, since the initial value $f^0$ gets ``overwritten'' at the very first iteration.
So in this description of the algorithm, $f^0$ doesn't play any role, i.e., its value could be chosen arbitrarily without affecting any of the subsequent iterates. 
Nonetheless, introducing notation for $f^0$ and $\pi^0$ is useful for consistency with the following alternative description of the algorithm (where, correspondingly, $w^1$ doesn't play any role).

\paragraph{Reformulation of the algorithm in terms of the increments.}
Let us introduce notation for the one-iteration increments rescaled by $\tau$:
\begin{equation} \label{eq:prelim:def_vk_wk}
    \forall k \geq 2,~
    v^k = (f^k - f^{k-2})/\tau
    \qquad \text{and} \qquad
    w^k = (g^k - g^{k-2})/\tau.
\end{equation}
Also set $v^1 = (f^1 - f^0)/\tau$ and let $w^1$ be any arbitrary vector; for convenience, take $w^1=0$.
Further denote the ``logits'' of $\pi^k$ w.r.t.\ $\mu \otimes \nu$ by
\begin{equation} \label{eq:prelim:def_Uk}
    \forall k \geq 0,~
    U^k_{ij} = (f^k_i + g^k_j - C_{ij})/\tau.
\end{equation}
Then the Sinkhorn algorithm can be fully re-expressed in terms of the variables $v^k, w^k, U^k$ instead of $(f^k, g^k)$, as they follow the self-contained update rule
\begin{align} \label{eq:prelim:upd_vwU}
    \text{for $k \geq 0$ even,}~~~
    &v^{k+1} = v[U^k]
    ~\quad \text{and} \quad~
    w^{k+1} = w^k
    ~\quad\quad~ \text{and} \quad~
    U^{k+1}_{ij} = U^k_{ij} + v^{k+1}_i \\
    \text{for $k$ odd,}~~~~
    &v^{k+1} = v^k
    ~\quad\quad~ \text{and} \quad~
    w^{k+1} = w[U^k]
    ~\quad \text{and} \quad~
    U^{k+1}_{ij} = U^k_{ij} + w^{k+1}_j
\end{align}
where
\begin{equation} \label{eq:prelim:def_v[U]_w[U]}
    \forall U \in \RR^{m \times n},~~
    v[U]_i = -\log \sum_j \, e^{U_{ij}} \nu_j
    \qquad\text{and}\qquad
    w[U]_j = -\log \sum_i \, e^{U_{ij}} \mu_i.
\end{equation}
In this formulation---which is the standard one in the matrix scaling literature, up to a component-wise logarithm---the cost matrix $C$ and the temperature $\tau$ only come into play via the initialization $U_{ij}^0 = (f^0_i+g^0_j-C_{ij})/\tau$.

Also note that the $v^k, w^k$ are related to the primal variables $\pi^k$ via
\begin{equation} \label{eq:prelim:rel_vk_Xpik}
    \text{for $k \geq 1$ even,}~~~
    v^{k+1}_i = v[U^k]_i
    = -\log \bigg( \sum_j e^{U^k_{ij}} \mu_i \nu_j \,/\, \mu_i \bigg)
    = \log \left( \mu_i \,/\, {\textstyle \sum_j} \pi^k_{ij} \right)
\end{equation}
and likewise for the $w^k$.
The case $k=0$ is slightly different, as $Z_\tau(f^0, g^0) \neq 1$ in general: we have
$v^1 = \log \left( \mu_i \,/\, {\textstyle \sum_j} \pi^0_{ij} \right) - \log Z_\tau(f^0, g^0)$.

\subsection{The Sinkhorn algorithm in the case with infinite costs} \label{subsec:prelim:infinite_costs}

Our work relies on previous results by \citet{baradat2024convergence} on the behavior of the Sinkhorn algorithm when applied to cost matrices $C$ with entries in $\RR \cup \{+\infty\}$.
In this section, we briefly review their relevant results restated in our notations.

Fix $\mu \in \Delta_m, \nu \in \Delta_n$, with $\mu_{\min}, \nu_{\min}>0$ without loss of generality, and $C \in (\RR \cup \{\infty\})^{m \times n}$.
The Sinkhorn algorithm for target marginals $\mu, \nu$ and transport cost matrix $C$ is defined by the same update rules as in the previous section, provided that one uses the convention $\exp(-\infty) = 0$.
Equivalently, all summations in the previous section should be restricted to neighboring vertices in the bipartite graph 
$G = (\{1\dots m\} \sqcup \{1\dots n\}, \EEE)$ where $\EEE = \left\{ (i,j) ;~ C_{ij}<\infty \right\}$.
For example, in the definitions of $f_\tau[g]_i$ and $v[U]_i$, $\sum_j$ should be replaced by $\sum_{j: (i,j) \in \EEE}$.
If there is a vertex with no neighbors in $G$, the Sinkhorn algorithm is undefined.

For ease of notation, let us introduce the shorthands
\begin{equation}
    \forall \pi \in \Delta_\EEE,~~
    (X_\sharp \pi)_i = \sum_{j: (i,j) \in \EEE} \pi_{ij}
    \qquad \text{and} \qquad
    (Y_\sharp \pi)_j = \sum_{i: (i,j) \in \EEE} \pi_{ij}.
\end{equation}
Naturally, $X_\sharp \pi \in \Delta_m$ and $Y_\sharp \pi \in \Delta_n$.
We will use the same shorthands $X_\sharp, Y_\sharp$ regardless of the set $\EEE \subset \{1 \dots m \} \times \{1 \dots n\}$ under consideration.

With these notations, a fundamental result of \cite{baradat2024convergence} is as follows.

\begin{theorem}[{\cite[Theorem~3.2, Proposition~4.5]{baradat2024convergence}}] \label{thm:prelim:baradat_ventre}
    Let $\mu \in \Delta_m$, $\nu \in \Delta_n$, $C \in (\RR \cup \{\infty\})^{m \times n}$, and denote $\EEE = \left\{ (i,j) ;~ C_{ij}<\infty \right\}$.
    Suppose $\mu_{\min}, \nu_{\min}>0$ and that the bipartite graph $(\{1\dots m\} \sqcup \{1\dots n\}, \EEE)$ has no isolated vertex.
    Consider $v^k, w^k$ the rescaled one-iteration increments and $U^k \in (\RR \cup \{-\infty\})^{m \times n}$ the ``logits'' of the Sinkhorn algorithm defined as in \eqref{eq:prelim:def_vk_wk}, \eqref{eq:prelim:def_Uk}.
    Then for any initialization $U^0_{ij} = (f^0_i + g^0_j - C_{ij})/\tau$ of the algorithm:
    \begin{itemize}
        \item Let $v^*_i = \log (\mu_i/\mu^*_i)$,
        $w^*_j = \log (\nu_j/\nu^*_j)$ where
        \begin{align}
            \mu^* &= \argmin_{\ol\mu}~ \Hdiv{\ol\mu}{\mu}
            ~~~~\text{subject to}~~~~
            \exists Q \in \Delta_\EEE;~
            X_\sharp Q = \ol\mu ~~\text{and}~~ Y_\sharp Q = \nu \\
            \nu^* &= \argmin_{\ol\nu}~ \Hdiv{\ol\nu}{\nu}
            ~~~~\text{subject to}~~~~
            \exists P \in \Delta_\EEE;~
            X_\sharp P = \mu ~~\text{and}~~ Y_\sharp P = \ol\nu.
        \end{align}
        Then $(v^k, w^k) \to (v^*, w^*)$ as $k \to \infty$.
        Moreover, these vectors satisfy
        \begin{equation}
            \forall (i, j) \in \EEE,~
            v^*_i + w^*_j \leq 0.
        \end{equation}
        \item $(U^{2k})_{k \geq 0}$ and $(U^{2k+1})_{k \geq 0}$ converge in $(\RR \cup \{-\infty\})^{m \times n}$ to two possibly distinct matrices $U^\infty_{\mathrm{even}}$, $U^\infty_{\mathrm{odd}}$ respectively.
        These limit matrices depend on $U^0$ and are characterized explicitly by \cite[Eq.~(3.2)]{baradat2024convergence}.%
        \footnote{Convergence of matrices in $(\RR \cup \{-\infty\})^{m \times n}$ can be understood as convergence of the component-wise exponentials. 
        The $P^*, Q^*, R$ appearing in \cite[Eq.~(3.2)]{baradat2024convergence} correspond in our notations to $(e^{U_{ij}} \mu_i \nu_j)_{ij}$ for $U=U^\infty_{\mathrm{odd}}, U^\infty_{\mathrm{even}}, U^0$ respectively.
        See also the first item of \autoref{lm:pik:def_Q*P*} for a restatement of this result.}
    \end{itemize}
\end{theorem}

Note that the $\mu^*, \nu^*$ and $v^*, w^*$ defined in the first item of \autoref{thm:prelim:baradat_ventre} depend only on $\mu, \nu$, and $\EEE$. They do not depend on the specific coefficients of~$C$, nor on $\tau$, nor on the algorithm's initialization.
For ease of future reference, we denote them by
\begin{align}
    \mu^*(\EEE) 
    &\coloneqq \mu^*
    &
    \nu^*(\EEE) 
    &\coloneqq \nu^* \\
    v^*(\EEE) 
    &\coloneqq v^* 
    = \left( \log (\mu_i/\mu_i^*) \right)_i
    &
    w^*(\EEE) 
    &\coloneqq w^*
    = \left( \log (\nu_j/\nu_j^*) \right)_j,
\end{align}
leaving their dependency on $\mu$ and $\nu$ implicit.

Let us furthermore record the following facts, also from \cite{baradat2024convergence}.

\begin{theorem} \label{thm:prelim:baradat_ventre_sets}
    In the setting of the previous theorem, denote
    \begin{equation}
        \ol\SSS = \left\{ (i,j) \in \EEE;~ v^*_i + w^*_j = 0 \right\}
    \end{equation}
    and
    \begin{equation}
        \SSS_{\mathrm{even}}(U^0) = \left\{ (i,j);~ (U^\infty_{\mathrm{even}})_{ij} > -\infty \right\},
        \qquad
        \SSS_{\mathrm{odd}}(U^0) = \left\{ (i,j);~ (U^\infty_{\mathrm{odd}})_{ij} > -\infty \right\}.
    \end{equation}
    Then the sets $\SSS_{\mathrm{even}}(U^0)$ and $\SSS_{\mathrm{odd}}(U^0)$ are equal, and moreover they are in fact independent of $U^0$.
    That is, $\SSS \coloneqq \SSS_{\mathrm{even}}(U^0) = \SSS_{\mathrm{odd}}(U^0)$ depends only on $\mu, \nu$, and $\EEE$.%
    \footnote{Specifically, it follows from \cite[Eq.~(3.2)]{baradat2024convergence} that ${(i_0,j_0) \in \SSS \iff \exists Q \in \Delta_{\EEE}; X_\sharp Q = \mu^*, Y_\sharp Q = \nu, \text{ and } Q_{i_0 j_0}>0}$.}
    
    Moreover, $\SSS \subset \ol\SSS \subset \EEE$.
    (Both inclusions can be strict simultaneously in general.)
    
    Furthermore, the bipartite graph on $\{1 \dots m \} \sqcup \{1 \dots n\}$ with edge set $\SSS$ has no isolated vertex. A fortiori, the same is true for $\ol\SSS$.
\end{theorem}


\subsection{Three auxiliary lemmas} \label{subsec:prelim:auxlemmas}

We now prove three new lemmas related to the Sinkhorn algorithm in the case with infinite costs.
The first one clarifies the relation between the sets $\SSS, \ol\SSS, \EEE$ appearing in \autoref{thm:prelim:baradat_ventre_sets}.

\begin{lemma} \label{lm:prelim:equal_mu*}
    Let $\mu, \nu, C, \EEE$ be as in \autoref{thm:prelim:baradat_ventre} and let
    $\SSS, \ol\SSS$ be as in \autoref{thm:prelim:baradat_ventre_sets}.
    Then
    \begin{equation}
        \mu^*(\SSS) = \mu^*(\ol \SSS) = \mu^*(\EEE)
        \qquad \text{and} \qquad
        \nu^*(\SSS) = \nu^*(\ol \SSS) = \nu^*(\EEE).
    \end{equation}
\end{lemma}

\begin{proof}
    Recall from \autoref{subsec:prelim:bg} that the Sinkhorn algorithm can be expressed purely in terms of the increments $v^k, w^k$ and the logits $U^k$ via the update rule \eqref{eq:prelim:upd_vwU}.
    Further recall from the introduction of \autoref{subsec:prelim:infinite_costs} that this formulation of the algorithm is valid both for $C \in \RR^{m \times n}$ and for $C \in (\RR \cup \{\infty\})^{m \times n}$, provided that one sets $\exp(-\infty) = 0$.

    By \autoref{thm:prelim:baradat_ventre}, $v^*(\EEE)$ and $w^*(\EEE)$ are characterized as the limits of $v^k, w^k$ when the algorithm is initialized at any $U^0$ such that $\left\{ (i,j);~ U^0_{ij} > -\infty \right\} = \EEE$.
    Now,
    \begin{itemize}
        \item Consider the iterates $(v^k, w^k, U^{2k}, U^{2k+1})_k$ of the algorithm initialized at $U_0 = (f^0_i + g^0_j - C_{ij})/\tau$ for some $f^0 \in \RR^m, g^0 \in \RR^n$. Then $(v^k, w^k, U^{2k}, U^{2k+1}) \to (v^*(\EEE), w^*(\EEE), U^\infty_{\mathrm{even}}, U^\infty_{\mathrm{odd}})$.
        \item Consider the iterates $(\tv^k, \tw^k, \tU^{2k}, \tU^{2k+1})_k$ of the algorithm initialized at $\tU^0 = U^\infty_{\mathrm{even}}$. Then by definition of $U^\infty_{\mathrm{even}}$, the iterates stay constant: $(\tv^k, \tw^k, \tU^{2k}, \tU^{2k+1}) = (v^*(\EEE), w^*(\EEE), U^\infty_{\mathrm{even}}, U^\infty_{\mathrm{odd}})$ for all $k$. On the other hand, since $\left\{ (i,j);~ \tU^0_{ij} > -\infty \right\} = \left\{ (i,j);~ (U^\infty_{\mathrm{even}})_{ij} > -\infty \right\} = \SSS$, we have $\lim_{k \to \infty} (\tv^k, \tw^k) = (v^*(\SSS), w^*(\SSS))$. Thus, $v^*(\EEE) = v^*(\SSS)$ and $w^*(\EEE) = w^*(\SSS)$.
    \end{itemize}
    This shows that $\mu^*(\SSS) = \mu^*(\EEE)$ and $\nu^*(\SSS) = \nu^*(\EEE)$.

    By definition of
    $\mu^*(\EEE) = \argmin_{\ol\mu} \Hdiv{\ol\mu}{\mu}$
    subject to
    $\exists Q \in \Delta_\EEE;~ X_\sharp Q = \ol\mu$ and $Y_\sharp Q = \nu$,
    since $\SSS \subset \ol\SSS \subset \EEE$, then
    \begin{equation}
        \Hdiv{\mu^*(\SSS)}{\mu}
        \geq \Hdiv{\mu^*(\ol\SSS)}{\mu}
        \geq \Hdiv{\mu^*(\EEE)}{\mu}.
    \end{equation}
    Now since $\mu^*(\SSS) = \mu^*(\EEE)$, the above inequalities are actually equalities.
    In particular, $\mu^*(\ol\SSS)$ is (feasible and) optimal for the optimization problem defining $\mu^*(\EEE)$, and so $\mu^*(\ol\SSS) = \mu^*(\EEE)$.
    A similar reasoning shows that $\nu^*(\ol\SSS) = \nu^*(\EEE)$.
\end{proof}

Our second result shows that the rescaled one-iteration increments of the Sinkhorn algorithm are bounded uniformly after
the first two updates.
Importantly, the bound depends only on $\mu_{\min}$ and $\nu_{\min}$, and is independent of $C$ and $\tau$.
As an immediate consequence, we also get a bound on the limiting increments.
Even in the case of finite transport costs, this lemma appears to be new.

\begin{lemma} \label{lm:prelim:unifbound_vk_wk}
    Let $\mu, \nu, C, \EEE$ be as in \autoref{thm:prelim:baradat_ventre}.
    Let the rescaled one-iteration increments $v^k, w^k$ and the logits $U^k$ of the Sinkhorn algorithm be defined as in \eqref{eq:prelim:def_vk_wk}, \eqref{eq:prelim:def_Uk}.
    Then we have, uniformly over the initialization $U^0_{ij} = (f^0_i+g^0_j-C_{ij})/\tau$ of the algorithm, 
    \begin{equation}
        \forall i,
        \forall k \geq 3,~~
        \log \mu_{\min} 
        \leq v^k_i 
        \leq -\log \nu_{\min} 
        \qquad \text{and} \qquad
        \forall j,
        \forall k \geq 4,~~
        \log \nu_{\min} 
        \leq w^k_j
        \leq -\log \mu_{\min}.
    \end{equation}
    Meanwhile, for the first two iterations, denoting $\ol\delta = \max_{i'j'} U^0_{i'j'}$ and $\ul\delta = -\min_{(i',j') \in \EEE} U^0_{i'j'}$,
    \begin{equation}
        \forall i,~~
        -\ol\delta
        \leq v^2_i = v^1_i
        \leq -\log \nu_{\min} + \ul\delta
        \qquad \text{and}\qquad
        \forall j,~~
        \log \nu_{\min} 
        \leq w^3_j = w^2_j
        \leq -\log \mu_{\min} + \ul\delta + \ol\delta.
    \end{equation}
\end{lemma}

\begin{proof}
    For the lower bounds, as we remarked in \eqref{eq:prelim:rel_vk_Xpik}, we have
    \begin{equation*}
        \forall k \geq 1 ~\text{even},~
        v^{k+1}_i = \log \left( \mu_i / (X_\sharp \pi^k)_i \right)
        \geq \log \mu_i
        \geq \log \mu_{\min}
    \end{equation*}
    and symmetrically for $w^{k+1}$ for $k \geq 1$ odd
    ($k=0$ is excluded because $Z_\tau(f^0, g^0) \neq 1$ in general).
    Hence the lower bound on $v^k$ for $k \geq 3$, resp.\ on $w^k$ for $k \geq 2$.

    For the upper bounds, by definition of the update \eqref{eq:prelim:upd_vwU},
    \begin{align*}
        \forall k \geq 0 ~\text{even},~~
        \forall i,~
        & \sum_j e^{U^{k+1}_{ij}} \nu_j 
        = \sum_j e^{U^k_{ij} + v^{k+1}_i} \nu_j
        = 1 \\
        \forall k \geq 0 ~\text{odd},~~
        \forall j,~
        & \sum_i e^{U^{k+1}_{ij}} \mu_i 
        = \sum_i e^{U^k_{ij} + w^{k+1}_j} \mu_i 
        = 1
    \end{align*}
    (still with the convention $\exp(-\infty) = 0$ so that the terms with $(i,j) \not\in \EEE$ equal $0$).
    So for any $i$,
    \begin{align*}
        \forall k \geq 2 ~\text{even},~~
        -v^{k+1}_i
        = \log \sum_j e^{U^k_{ij}}\, \nu_j
        &= \log \sum_j e^{U^{k-1}_{ij} + w^k_j}\, \nu_j \\
        &\geq \log \bigg( e^{\min_{j'} w^k_{j'}} \sum_j e^{U^{k-1}_{ij}} \nu_j \bigg)
        = \min_{j'} w^k_{j'}
        \geq \log \nu_{\min}.
    \end{align*}
    Symmetrically,
    $\forall k \geq 3$ odd,
    $-w^{k+1}_j 
    \geq \min_{i'} v^k_{i'}
    \geq \log \mu_{\min}$.
    Hence the upper bound on $v^k$ for $k \geq 3$, resp.\ on $w^k$ for $k \geq 4$.

    It only remains to show the bounds for $v^2=v^1$ and the upper bound for $w^3=w^2$.
    For $v^1$,
    \begin{equation}
        -\max_{i'j'} U^0_{i'j'}
        \leq v^1_i 
        = -\log \sum_j e^{U^0_{ij}} \nu_j
        \leq -\log \left(
            \min_{(i',j') \in \EEE} e^{U^0_{i'j'}} ~
            \nu_{\min}
        \right)
        = -\log \nu_{\min}
        - \min_{(i',j') \in \EEE} U^0_{i'j'}.
    \end{equation}
    For $w^2$, similarly, since
    $\min_{(i,j) \in \EEE} U^1_{ij} 
    \geq \min_{(i,j) \in \EEE} U^0_{ij} + \min_i v^1_i$,
    \begin{equation}
        w^2_j 
        \leq -\log \mu_{\min}
        - \min_{\EEE} U^1
        \leq -\log \mu_{\min} 
        - \min_{\EEE} U^0
        + \max U^0,
    \end{equation}
    as announced.
\end{proof}

\begin{corollary} \label{coroll:prelim:bound_v*_w*}
    For any $\mu \in \Delta_m, \nu \in \Delta_n$ such that $\mu_{\min}, \nu_{\min}>0$, for any $\EEE \subset \{1 \dots m \} \times \{1 \dots n\}$ such that the bipartite graph with edge set $\EEE$ has no isolated vertex, we have
    \begin{equation}
        \forall i,~~ 
        \log \mu_{\min} \leq v^*_i(\EEE) \leq -\log \nu_{\min} 
        \qquad \text{and} \qquad
        \forall j,~~ 
        \log \nu_{\min} \leq w^*_j(\EEE) \leq -\log \mu_{\min}.
    \end{equation}
    In particular,~
    $
        \min_i \mu^*_i(\EEE),~
        \min_j \nu^*_j(\EEE)
        \geq \mu_{\min} \nu_{\min}
    $.
\end{corollary}

\begin{proof}
    Take the limit $k \to \infty$ in the inequalities of \autoref{lm:prelim:unifbound_vk_wk}.
    %
    %
    The lower bound on the $\mu^*_i(\EEE)$ follows by noting that
    $
        \forall i,\,
        \log(\mu_i / \mu^*_i(\EEE)) 
        = v^*_i(\EEE)
        \leq -\log \nu_{\min}
        \implies
        \log \mu^*_i(\EEE) 
        \geq \log (\mu_i \nu_{\min})
        \geq \log (\mu_{\min} \nu_{\min})
    $,
    and likewise for the $\nu^*_j(\EEE)$.
\end{proof}

Our third lemma controls the discrepancy between the Sinkhorn iterates for an EOT problem with high but finite costs, and those for the problem with infinite costs.

\begin{lemma} \label{lm:prelim:control_finite_infinite_costs}
    Let $(f^k, g^k)_k$ be the iterates of the Sinkhorn algorithm applied to the EOT problem~\eqref{eq:intro:EOT}.
    Let $\EEE \subset \{1 \dots m\} \times \{1 \dots n\}$ such that the bipartite graph with edge set $\EEE$ has no isolated vertex,
    and let $(\tf^k, \tg^k)_k$ be the iterates of the Sinkhorn algorithm applied to the EOT problem with the same target marginals $\mu, \nu$ but with the cost matrix
    \begin{equation*}
        \tC \in (\RR \cup \{\infty\})^{m \times n},
        \qquad
        \tC_{ij} = \begin{cases}
            C_{ij} ~~\text{if}~ (i,j) \in \EEE \\
            +\infty ~~\text{otherwise},
        \end{cases}
    \end{equation*}
    and with an initialization $(\tf^0, \tg^0)$ possibly different from $(f^0, g^0)$.
    Let $\delta$ and $(M^k)_{k \geq 0}$ be such that
    \begin{equation}
        \forall (i,j) \in \EEE,~
        -\delta \leq \left( \tf^0_i + \tg^0_j - C_{ij} \right) / \tau \leq \delta
        \quad~ \text{and} \quad~
        \forall (i,j) \not\in \EEE, \forall k,~~
        \left( f^k_i + g^k_j - C_{ij} \right) / \tau \leq -M^k.
    \end{equation}
    Further denote 
    $\tv^1 = (\tf^1-\tf^0)/\tau$,\,
    $\tw^2 = (\tg^2-\tg^0)/\tau$.
    Then $\Delta^k = \norm{f^k-\tf^k}_\infty \vee \norm{g^k-\tg^k}_\infty$ satisfies
    \begin{align}
        \forall k \geq 2,~~~
        \Delta^k 
        &\leq \Delta^0
        + \tau \left[
            e^{(\max \tv^1) - M^0}
            + e^{(\max \tw^2) - M^1}
            + (\mu_{\min} \wedge \nu_{\min})^{-1}
            \sum_{l=3}^{k-1} e^{-M^l}
        \right] \\
        &\leq \Delta^0 
        + \tau\, (\mu_{\min} \wedge \nu_{\min})^{-1} \left[
            e^{2\delta} (e^{-M^0} + e^{-M^1})
            + \sum_{l=3}^{k-1} e^{-M^l}
        \right].
    \end{align}
\end{lemma}

\begin{proof}
    Let $\tv^k, \tw^k$ denote the rescaled one-iteration increments of $(\tf^k, \tg^k)$ as defined in \eqref{eq:prelim:def_vk_wk}.
    For any $k \geq 0$ even, for any $i$, by definition of the Sinkhorn update,
    \begin{align}
        \left( f^{k+1}_i - \tf^{k+1}_i \right) / \tau
        &= -\log \frac{
            \sum_j\, e^{[g^k_j-C_{ij}]/\tau} \nu_j
        }{
            \sum_{j: (i,j) \in \EEE}\, e^{[\tg^k_j-C_{ij}]/\tau} \nu_j
        } \\
        &= -\log \frac{
            \sum_{j: (i,j) \in \EEE}
            ~
            e^{(g^k_j-\tg^k_j)/\tau} ~ 
            e^{[\tg^k_j-C_{ij}]/\tau} \nu_j
            + \sum_{j: (i,j) \not\in \EEE}\,
            e^{[g^k_j-C_{ij}]/\tau} \nu_j
        }{
            \sum_{j: (i,j) \in \EEE}\, e^{[\tg^k_j-C_{ij}]/\tau} \nu_j
        } \\
        &= -\log \left(
            e^{\lambda/\tau}
            + \frac{e^{(\tf^k_i - f^k_i)/\tau}}{\sum_{j: (i,j) \in \EEE}\, e^{[\tf^k_i+\tg^k_j-C_{ij}]/\tau} \nu_j}
            ~\sum_{j: (i,j) \not\in \EEE}
            e^{[f^k+g^k_j-C_{ij}]/\tau} \nu_j
        \right) \\
        &= -\log \left(
            e^{\lambda/\tau}
            + e^{\tv^{k+1}_i} \, 
            e^{(\tf^k_i - f^k_i)/\tau}
            \sum_{j: (i,j) \not\in \EEE}
            e^{[f^k+g^k_j-C_{ij}]/\tau} \nu_j
        \right)
    \end{align}
    for some
    \begin{equation}
        \min_{j: (i,j) \in \EEE}\,
        (g^k_j-\tg^k_j)
        \leq \lambda \leq
        \max_{j: (i,j) \in \EEE}\,
        (g^k_j-\tg^k_j),
    \end{equation}
    where in the last line we used that 
    $\tv^{k+1}_i = -\log \sum_{j: (i,j) \in \EEE} e^{[\tf^k_i + \tg^k_j - C_{ij}]/\tau} \nu_j$
    as noted in \eqref{eq:prelim:upd_vwU}.
    So
    \begin{equation}
        \left( f^{k+1}_i - \tf^{k+1}_i \right) / \tau
        \leq -\log( e^{\lambda/\tau} + 0)
        = -\lambda/\tau
        \leq \norm{\tg^k-g^k}_\infty / \tau
    \end{equation}
    and in the other direction,
    \begin{align}
        -\left( f^{k+1}_i - \tf^{k+1}_i \right) / \tau
        &\leq 
        \log \left(
            e^{\lambda /\tau}
            + e^{\tv^{k+1}_i} \,
            e^{(\tf^k_i - f^k_i)/\tau} \,
            e^{-M^k}
        \right) \\
        &\leq \log \left(
            e^{\norm{g^k-\tg^k}_\infty/\tau}
            + e^{(\max_{i'} \tv^{k+1}_{i'})}
            e^{-M^k}
            \cdot e^{\norm{\tf^k - f^k}_\infty/\tau}
        \right) \\
        &\leq \left(\norm{f^k-\tf^k}_\infty \vee \norm{g^k-\tg^k}_\infty \right)\! / \tau
        ~+~ \log\left( 1 + e^{(\max_{i'} \tv^{k+1}_{i'})} e^{-M^k} \right) \\
        &\leq \left(\norm{f^k-\tf^k}_\infty \vee \norm{g^k-\tg^k}_\infty \right)\! / \tau
        ~+~ e^{(\max_{i'} \tv^{k+1}_{i'})} e^{-M^k}.
    \end{align}
    Thus, by the analogous computation for $k$ odd, denoting $\Delta^k = \norm{f^k-\tf^k}_\infty \vee \norm{g^k-\tg^k}_\infty$,
    \begin{align}
        \forall k \geq 0 ~\text{even},~~
        \Delta^{k+1}
        &\leq \Delta^k 
        + \tau \,e^{(\max \tv^{k+1})} e^{-M^k} \\
        \forall k \geq 0 ~\text{odd},~~
        \Delta^{k+1}
        &\leq \Delta^k 
        + \tau \,e^{(\max \tw^{k+1})} e^{-M^k}.
    \end{align}
    Now by \autoref{lm:prelim:unifbound_vk_wk}, for all $k \geq 3$, 
    $\max \tv^{k+1}, \max \tw^{k+1} \leq -\log (\mu_{\min} \wedge \nu_{\min})$.
    Hence, as announced,
    \begin{equation}
        \forall k \geq 2,~
        \Delta^k
        \leq \Delta^0 
        + \tau \left[ 
            e^{(\max \tv^1) - M^0}
            + e^{(\max \tw^2) - M^1}
            + (\mu_{\min} \wedge \nu_{\min})^{-1}
            \sum_{l=3}^{k-1}  e^{-M^l}
        \right].
    \end{equation}

    To get the more explicit bound in the second inequality of the lemma,
    use that under the stated assumption on $\tf^0, \tg^0$, 
    by \autoref{lm:prelim:unifbound_vk_wk},
    $\max_i \tv^1_i \leq -\log \nu_{\min} + \delta$
    and
    $\max_j \tw^2_j \leq -\log \mu_{\min} + 2 \delta$.
\end{proof}

\subsection{Quantitative convergence bounds for Sinkhorn with infinite costs} \label{subsec:prelim:qtve_BV24}

As a final preliminary, we now state two quantitative versions of the convergence in \autoref{thm:prelim:baradat_ventre}: a polynomial convergence bound in $\tO(1/k)$ in the general case, and an exponential convergence bound under an additional scalability assumption.
The proofs are delayed to \autoref{sec:apx_pf_qtve_BV24}.

Our result for the general case is as follows.
Its proof makes crucial use of several insights from \cite{baradat2024convergence}---notably the reduction to the asymptotically scalable case implicit in their Proposition~5.3 (see \autoref{lm:apx_pf_qtve_BV24:control_E_olS}, \autoref{lm:apx_pf_qtve_BV24:prop5.3_BV24})---but also of a recent result by the author in \cite{wang2026almost} to get the $k^{-1} \log k$ convergence rate. 
Let us mention that the arguments used in the proof of \cite[Theorem~3.2]{baradat2024convergence} only allow to show a rate of $k^{-1/2}$ (see \autoref{lm:apx_pf_qtve_BV24:slow_rate}), which would be insufficient for our later purposes.

\begin{proposition} \label{prop:prelim:qtve_BV24}
    Let $\mu, \nu, C, \EEE$ be as in \autoref{thm:prelim:baradat_ventre}.
    Let $\delta \geq 0$ and consider any initialization 
    $U^0 = \left( (f^0_i + g^0_j - C_{ij})/\tau \right)_{ij} \in (\RR \cup \{-\infty\})^{m \times n}$
    of the Sinkhorn algorithm such that
    \begin{equation}
        \forall (i,j) \in \EEE,~
        -\delta \leq U^0_{ij} \leq \delta
        \qquad \text{and} \qquad
        \forall (i,j) \not\in \EEE,~
        U^0_{ij} = -\infty.
    \end{equation}
    Then the rescaled one-iteration increments $v^k, w^k$ satisfy
    \begin{equation}
        \forall k \geq K_0 (1+\delta),~~
        \norm{v^k - v^*(\EEE)}_\infty,~
        \norm{w^k - w^*(\EEE)}_\infty
        \leq B\, \frac{1 + \delta + \log k}{k}
    \end{equation}
    for some constants $K_0, B$ dependent only on $\mu, \nu$, and $\EEE$.
\end{proposition}

Before stating our second result, let us first show an equivalence between equality of the sets $\SSS, \ol\SSS$ appearing in \autoref{thm:prelim:baradat_ventre_sets} and certain exact scalability conditions.
These equivalences follow essentially immediately from the definitions, but we find it useful to record them for future reference.

\begin{lemma} \label{lm:prelim:equiv_exact_scalability}
    Let $\mu, \nu, C, \EEE$ be as in \autoref{thm:prelim:baradat_ventre} and $\SSS, \ol\SSS$ be as in \autoref{thm:prelim:baradat_ventre_sets}.
    Let $\mu^* = \mu^*(\EEE) = \mu^*(\ol\SSS) = \mu^*(\SSS)$, 
    $\nu^* = \nu^*(\EEE) = \nu^*(\ol\SSS) = \nu^*(\SSS)$
    by \autoref{lm:prelim:equal_mu*}.
    Let $\ol A \in \RR_+^{m \times n}$ be any nonnegative matrix such that $\left\{ (i,j);~ \ol A_{ij}>0 \right\} = \ol\SSS$.
    Then $\olA$ is asymptotically $(\mu^*, \nu)$-scalable 
    and asymptotically $(\mu, \nu^*)$-scalable
    in the sense of \cite[Theorem~4.2]{idel2016review}. 
    Moreover, the following statements are equivalent:
    \begin{enumerate}[noitemsep]
        \item[(i)] $\SSS = \ol\SSS$.
        \item[(ii)] The matrix $\olA$ is exactly $(\mu^*, \nu)$-scalable in the sense of \cite[Theorem~4.1]{idel2016review}.
        \item[(iii)] The matrix $\olA$ is exactly $(\mu, \nu^*)$-scalable.
    \end{enumerate}
\end{lemma}

\begin{proof}
    By definition \cite{idel2016review}, a matrix $A'$ with pattern $\left\{ (i,j);~ A'_{ij}>0 \right\} = \SSS'$ is exactly $(\mu', \nu')$-scalable if and only if
    there exists $Q \in \Delta_{\SSS'}$ such that $X_\sharp Q = \mu'$, $Y_\sharp Q = \nu'$, and $Q_{ij}>0$ for all $(i,j) \in \SSS'$.
    Asymptotic scalability is defined in the same way except the condition on the positivity of entries in $\SSS'$ is removed.
    The first part of the lemma, on asymptotic $(\mu^*, \nu)$- and $(\mu, \nu^*)$-scalability of $\ol A$, then follows directly from the definition of $\mu^*(\ol\SSS)$ and $\nu^*(\ol\SSS)$.
    
    For the second part of the lemma, note that by \cite[Eq.~(3.2)]{baradat2024convergence}, we have
    $(i_0,j_0) \in \SSS \iff \exists Q \in \Delta_{\EEE};~ X_\sharp Q = \mu^*, Y_\sharp Q = \nu$, and $Q_{i_0 j_0}>0$, and likewise with $(\mu, \nu^*)$ instead of $(\mu^*, \nu)$.
    The claimed equivalences follow immediately.
\end{proof}

Our second result is that the convergence $(v^k, w^k) \to (v^*(\EEE), w^*(\EEE))$ is actually exponential if $\SSS = \ol\SSS$.
At a high level, this is in line with the intuition that the Sinkhorn algorithm converges linearly when, and only when, the underlying problem is exactly scalable \cite{soules1991rate,achilles1993implications}.

\begin{proposition} \label{prop:prelim:qtve_BV24_exp}
    In the same setting as \autoref{prop:prelim:qtve_BV24}, additionally suppose $\SSS = \ol\SSS$, where $\SSS, \ol\SSS$ are the sets defined in \autoref{thm:prelim:baradat_ventre_sets}.
    Then
    \begin{equation}
        \forall k \geq K_0 (1+\delta),~~
        \norm{v^k - v^*(\EEE)}_\infty,~
        \norm{w^k - w^*(\EEE)}_\infty
        \leq B (1+\delta) \left( 1-e^{-R(1+\delta)} \right)^k
    \end{equation}
    for some constants $K_0, B, R$ dependent only on $\mu, \nu$, and $\EEE$.
\end{proposition}

\section{The cold Sinkhorn dynamics} \label{sec:CS}

In this section, we introduce the cold Sinkhorn dynamics, which is the limiting dynamics of the Sinkhorn algorithm in the regime of small regularization $\tau$.
We start by deriving it informally from the Sinkhorn algorithm in \autoref{subsec:CS:informal}.
Then we give its formal definition in \autoref{subsec:CS:formaldef}, and we analyze its properties in \autoref{subsec:CS:ppties}.

\subsection{Informal derivation} \label{subsec:CS:informal}

In this section, we use $o(\cdot), O(\cdot), \Omega(\cdot), \Theta(\cdot)$ to indicate scalings in the regime $\tau \to 0$.

\paragraph{The first and second iterations.}
By a well-known property of the log-sum-exp function, we have $f_\tau[g] = f_0[g] + o(1)$ for any $g$, and likewise for $g_\tau[\cdot]$.
So it is tempting to consider the sequence
$f_0[g_0[ ... [f_0[g^0]] ... ]]$
as an approximation of the Sinkhorn iterates
$f^k = f_\tau[g_\tau[ ... [f_\tau[g^0]] ... ]]$, 
and likewise for $g^k$.
However, it is classical that the former sequence is stationary after only two steps: 
indeed one can show that $(\olf, \ol g) = (f_0[g^0],~ g_0[f_0[g^0]]) \in \partial \FF$, the Pareto frontier of the dual OT problem's feasible set, and so $f_0[\ol g] = \ol f, g_0[\ol f] = \ol g$ by \autoref{lm:prelim:partialFF}
\cite[Section~3.2]{peyre2019computational}.
So this approximation is insufficient to capture the small-$\tau$ behavior of the Sinkhorn algorithm.

To derive a more precise approximation, let us consider the first few iterations of the algorithm.
At the first two steps, approximating $f_\tau[\cdot], g_\tau[\cdot]$ by $f_0[\cdot], g_0[\cdot]$ is meaningful: effectively, the first two iterations project $(f^0, g^0)$ onto some point $(f^2, g^2)$ that is $o(1)$-close to $\partial \FF$.
Thereafter, the iterates evolve according to
\begin{equation}
    \forall k \geq 3,~~
    f^k - f^{k-2} = \tau v^k
    ~~~~\text{and}~~~~
    g^k - g^{k-2} = \tau w^k
\end{equation}
by definition, and we know by \autoref{lm:prelim:unifbound_vk_wk} that $v^k, w^k = O(1)$ uniformly for all $k \geq 3$---consistent with the intuition that the algorithm is non-stationary only thanks to the presence of entropy regularization. (To be exact, $w^k = O(1)$ for $k \geq 4$, but $g^3 = g^2$ by definition anyway.)

\paragraph{The first $\Theta(1)$ iterations.}
Since $v^k, w^k = O(1)$, the matrix $\tau U^k_{ij} = f^k_i + g^k_j - C_{ij}$ stays approximately constant for $\Theta(1)$ iterations. That is, for ease of reference, 
\begin{align}
    \text{for all}~~
    2 \leq k \leq K = \Theta(1),~~~~
    \tau U^k_{ij} 
    = f^k_i + g^k_j - C_{ij}
    &= f^2_i + g^2_j - C_{ij} + O(\tau).
\end{align}
Let us describe in detail the behavior of the algorithm for $k \in \{2, ..., K\}$.
Denote 
\begin{equation}
    \ol f = f_0[g^0],
    ~~~~
    \ol g = g_0[\ol f],
    ~~~~
    \ol U_{ij} = \ol f_i + \ol g_j - C_{ij},
    ~~~~\text{and}~~~~
    \EEE = \left\{ (i,j);~ \ol U_{ij} = 0 \right\}.
\end{equation}
Then by \autoref{lm:prelim:partialFF}, $\ol U_{ij} \leq 0$ for all $(i,j)$ and the bipartite graph $( \{1 \dots m\} \sqcup \{1 \dots n\}, \EEE )$ has no isolated vertex.
Now we have $(f^2, g^2) = (\ol f, \ol g) + o(1)$, and more precisely one can show that $(f^2, g^2) = (\ol f, \ol g) + O(\tau)$ (\autoref{lm:cv:logsumexp_rate} below).
Thus, $\tau U^k_{ij} = \ol U_{ij} + O(\tau)$ for all $k \leq K$, and so
\begin{equation}
    \forall (i,j) \in \EEE,~~
    U^k_{ij} = O(1)
    ~~~~\text{and}~~~~
    \forall (i,j) \not\in \EEE,~~
    U^k_{ij} = \underbrace{\ol U_{ij}}_{<0}/\tau + O(1) = -\Theta(1/\tau)
    \xrightarrow{\tau\to0} -\infty.
\end{equation}

Consequently, for all $k \in \{2, ..., K\}$, the summation $\sum_j$ in the definition of $v[U^k]_i = -\log \sum_j e^{U^k_{ij}} \nu_j$ can be approximately replaced by $\sum_{j: (i,j) \in \EEE}$, and likewise for the $w[U^k]_j$.
In other words, \emph{the Sinkhorn algorithm behaves approximately as in the case where the cost matrix has infinite entries at the $(i,j) \not\in \EEE$.}
In particular, by \autoref{thm:prelim:baradat_ventre}, $(v^k, w^k)$ converges to some $(\ol v, \ol w)$ determined only by $\mu, \nu$, and $\EEE$. 

Let us make an Ansatz that the speed of this convergence is independent of $\tau$ and that we attain $(v^{K'}, w^{K'}) = (\ol v, \ol w) + O(\tau)$ in a number of iterations $K' < K$.
Then the Sinkhorn iterates evolve linearly with $(f^k-f^{k-2}, g^k-g^{k-2}) = (\tau \ol v, \tau \ol w) + O(\tau^2)$ over $k \in \{ K', ..., K\}$.

\paragraph{The first $\Theta(1/\tau)$ iterations (the first ``phase'').}
Reflecting on the derivation so far, notice that our reasoning is valid not just until some $K=\Theta(1)$, but also until the maximal $K$ such that 
\begin{equation}
    \text{for all}~~
    2 \leq k \leq K,~~~~
    \forall (i,j) \not\in \EEE,~ U^k_{ij} \leq -\Omega(1/\tau).
\end{equation}
Indeed, this is the condition that makes the algorithm behave approximately as in the infinite cost case, and so $(v^k, w^k) = (\ol v, \ol w) + O(\tau)$ for all $k \in \{K', ..., K\}$.
(In particular, whether or not the $U^k_{ij}$ for $(i, j) \in \EEE$ remain bounded does not matter.)
Now since 
$U^{k+1}_{ij} - U^k_{ij} = O(1)$ 
and, initially, 
$\forall (i,j) \not\in \EEE, U^2_{ij} = \ol U_{ij}/\tau + O(1) = -\Theta(1/\tau)$, 
the above condition stays satisfied for at least $\Theta(1/\tau)$ iterations.

The iterations $\{2, ..., K\}$ correspond to what we will call the first phase, and $\{2, ..., K'\}$ can be thought of as the first phase transition.

\paragraph{End of the first phase, start of the second phase.}
The first phase stops when the condition above is violated, i.e., when one of the $U^k_{i_0 j_0}$ for $(i_0, j_0) \not\in \EEE$ reaches $o(1/\tau)$ in magnitude.
Now during the first phase, since $(v^k, w^k) = (\ol v, \ol w) + O(\tau)$, the $U^k_{ij}$ evolve according to
$U^k_{ij} - U^{k-2}_{ij} = \ol v_i + \ol w_j + O(\tau)$.
So there are two cases to distinguish:
\begin{itemize}
    \item If for all $(i_0, j_0) \not\in \EEE$, $\ol v_{i_0} + \ol w_{j_0} \leq 0$, then the first phase never ends. In this rare case, the first phase coincides with the final phase, which is detailed in the next paragraph.
    \item If there exists $(i_0, j_0) \not\in \EEE$ such that $\ol v_{i_0} + \ol w_{j_0} > 0$, then there exists $K_1>K'$ such that $U^{K'}_{i_0 j_0} + (K_1-K') (\ol v_{i_0} + \ol w_{j_0}) > 0$.
    Consider the smallest such integer $K_1$.
    While it would be difficult to finely describe the behavior of the Sinkhorn algorithm at iterations $k \approx K_1$, it turns out that for small $\tau$, everything happens roughly as though $U^{K_1}$ itself satisfies $U^{K_1}_{i_0 j_0} = O(1)$.
    That is, considering $\EEE' = \left\{ (i,j);~ U^{K_1}_{ij} = O(1) \right\}$, we again have that for $k \approx K_1$,
    \begin{equation}
        \forall (i,j) \in \EEE',~~
        U^k_{ij} = O(1)
        ~~~~\text{and}~~~~
        \forall (i,j) \not\in \EEE',~~
        U^k_{ij} = -\Theta(1/\tau)
        \xrightarrow{\tau\to0} -\infty.
    \end{equation}
    Thereafter, the reasoning from the previous paragraphs again applies: the Sinkhorn iterates go through a second phase transition where $(v^k, w^k)$ converge to some $(\ol v', \ol w')$, and then there occurs a second phase where $U^k_{ij} - U^{k-2}_{ij} = \ol v'_i + \ol w'_j + O(\tau)$, for $\Theta(1/\tau)$ iterations.
\end{itemize}

\paragraph{The final phase.}
The behavior described above repeats for a certain number of phases:
at each phase $\ell$, one has $U^k_{ij} - U^{k-2}_{ij} = \ol v^\ell_i + \ol w^\ell_j + O(\tau)$, where $(\ol v^\ell, \ol w^\ell)$ is determined by a certain edge set $\EEE^\ell$ corresponding to the ``effective finite-cost pattern'' seen by the algorithm.

This repeats until there comes a phase $L$ where, for all $(i_0, j_0) \not\in \EEE^L$, $\ol v^L_{i_0} + \ol w^L_{j_0} \leq 0$.
One can show that this necessarily occurs eventually (\autoref{lm:CS:L_finite}) and that $(\ol v^L, \ol w^L) = (0, 0)$ (\autoref{lm:CS:finalphase_optimal}).
The iterates $f^k, g^k$ then evolve at a much slower rate: $(f^k-f^{k-2}, g^k-g^{k-2}) = O(\tau^2)$ instead of $\Theta(\tau)$ at the previous phases.
So at time-scales $k \asymp \tau^{-1}$, the Sinkhorn iterates effectively become stationary, and phase $L$ acts as the final phase of the dynamics.
One can also show that in the limit $\tau \to 0$, the final value taken by $(f,g)$ is an optimal solution of the dual OT problem.

\subsection{Formal definition} \label{subsec:CS:formaldef}

We can now give the formal definition of the cold Sinkhorn dynamics.

\begin{samepage}
\begin{definition} \label{def:CS:formaldef}
    For any initial pair $(\ol f^0, \ol g^0) \in \partial \FF$, the cold Sinkhorn dynamics is the continuous-time curve $(f(t), g(t))$ given by
    $f(0) = \ol f^0, g(0) = \ol g^0$, and
    \begin{equation}
        \forall \ell \leq L,~
        \forall t \in [t_\ell, t_{\ell+1}),~~
        \frac{d}{dt} f(t) = \ol v^\ell,
        ~~~~
        \frac{d}{dt} g(t) = \ol w^\ell
    \end{equation}
    for a sequence $(t_\ell, \ol v^\ell, \ol w^\ell, \EEE^\ell)_{\ell \leq L}$ defined recursively as follows.
    \begin{itemize}
        \item For $\ell=0$:
        $t_0=0$ and $\ol v^0 = v^*(\EEE^0), \ol w^0 = w^*(\EEE^0)$
        where
        $\EEE^0 = \left\{ (i,j);~ \ol f^0_i + \ol g^0_j - C_{ij} = 0 \right\}$.
        \item For any $\ell \geq 0$,
        \begin{equation}
            t_{\ell+1} = \inf \left\{ 
                t \geq t_\ell;~~
                \max_{(i,j) \not\in \EEE^\ell}~
                f_i(t_\ell) 
                + g_j(t_\ell)
                + (t-t_\ell) \left( \ol v^\ell_i + \ol w^\ell_j \right)
                - C_{ij}
                \geq 0
            \right\}.
        \end{equation}
        If $t_{\ell+1} < \infty$, then
        $\ol v^{\ell+1} = v^*(\EEE^{\ell+1}),
        \ol w^{\ell+1} = w^*(\EEE^{\ell+1})$
        where 
        \begin{equation}
            \EEE^{\ell+1} = \left\{ (i,j);~ 
                f_i(t_\ell) 
                + g_j(t_\ell)
                + (t_{\ell+1}-t_\ell) \left( \ol v^\ell_i + \ol w^\ell_j \right)
                - C_{ij}
                = 0
            \right\}.
        \end{equation}
        If $t_{\ell+1} = \infty$, then $L=\ell$ and the sequence terminates. $\ol v^{\ell+1}, \ol w^{\ell+1}$, and $\EEE^{\ell+1}$ are not defined.
    \end{itemize}
    Further set $t_{L+1} = \infty$.
    We refer to each interval $[t_\ell, t_{\ell+1})$ as a \emph{phase} of the dynamics.
\end{definition}
\end{samepage}

\begin{remark}
    Recall that $v^*_i(\EEE^{\ell}) + w^*_j(\EEE^{\ell}) \leq 0$ for all $(i,j) \in \EEE^{\ell}$, by \autoref{thm:prelim:baradat_ventre}.
    So by construction, $f_i(t) + g_j(t) - C_{ij} \leq 0$ for all $(i,j)$, for all $t$.
\end{remark}

\pagebreak

%

As a support for intuition, we display in \autoref{fig:CS:CS} the behavior of the cold Sinkhorn dynamics on a simple example with $m=2$ and $n=3$.
As \autoref{fig:CS:CS_Uij} illustrates, the matrix $\ol U_{ij}(t) = f_i(t) + g_j(t) - C_{ij}$ evolves linearly during each phase $t \in [t_\ell, t_{\ell+1})$, and a new phase begins whenever a component $\ol U_{i_0j_0}(t)$ which was previously negative reaches $0$.
The new direction $\frac{d}{dt} \ol U_{ij}(t_{\ell+1}^+) = \ol v^{\ell+1}_i + \ol w^{\ell+1}_j$ then gets computed according to 
$\ol v^{\ell+1} = v^*(\EEE^{\ell+1}),
\ol w^{\ell+1} = w^*(\EEE^{\ell+1})$,
where $\EEE^{\ell+1} = \left\{ (i,j);~ \ol U_{ij}(t_{\ell+1}) = 0 \right\}$.
This may cause some previously zero components $\ol U_{i_1j_1}(t)$ to immediately become negative, but this is not reflected in the definition of $\EEE^{\ell+1}$: these are the edges $(i_1, j_1) \in \EEE^{\ell+1} \setminus \ol \SSS^{\ell+1}$ where $\ol\SSS^{\ell+1} = \left\{(i,j) \in \EEE^{\ell+1};~ \ol v^{\ell+1}_i + \ol w^{\ell+1}_j = 0 \right\}$.

This process repeats until there comes a phase where the direction 
$\frac{d}{dt} \ol U_{ij}(t_{L}^+) = \ol v^{L}_i + \ol w^{L}_j$ is non-positive for \textit{all} components $i \leq m, j \leq n$.
In fact, it then necessarily holds $\ol v^L_i = \ol w^L_j = 0$ for all $i,j$, as shown in \autoref{lm:CS:finalphase_optimal} below.

We note that, although it does not occur for the simple example presented in \autoref{fig:CS:CS}, it is possible for a component $\ol U_{ij}(t)$ to be successively $0$, decreasing, increasing, and $0$ again along successive phases of the dynamics. 

\begin{figure}[t]
    \centering
    \begin{subfigure}[t]{\textwidth}
        \centering
        \includegraphics[width=\linewidth]{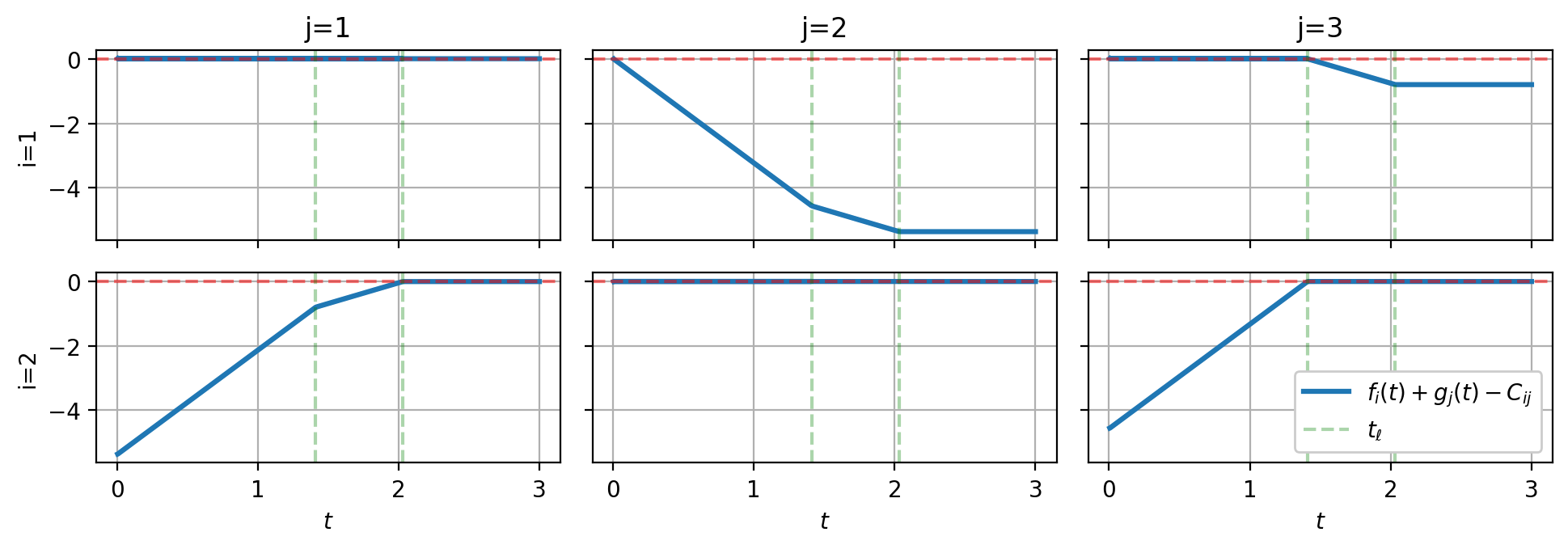} 
        \caption{Evolution of $f_i(t) + g_j(t) - C_{ij}$ along the cold Sinkhorn dynamics.
        (In reality, this figure was obtained by running the Sinkhorn algorithm with $\tau = 0.005$: the x-axis shows $\tau k/2$ and the y-axis shows $f^k_i + g^k_j - C_{ij}$.)}
        \label{fig:CS:CS_Uij}
    \end{subfigure}
    \\[1em]
    \begin{subfigure}[t]{0.32\textwidth}
        \centering
        \includegraphics[width=\linewidth]{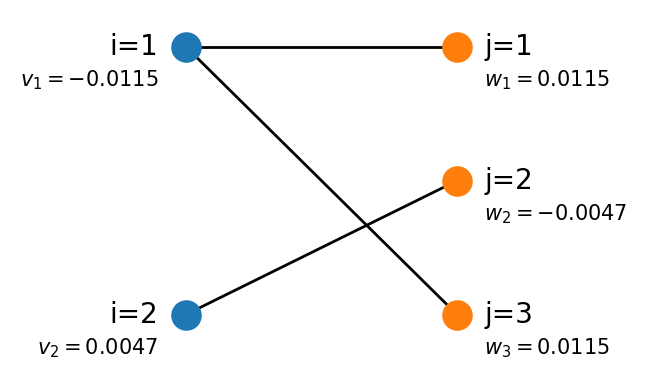}
        \caption{Phase $\ell=0$, $t \in [0, 1.4)$.}
        \label{fig:CS:CS_phase1}
    \end{subfigure}
    \hfill
    \begin{subfigure}[t]{0.32\textwidth}
        \centering
        \includegraphics[width=\linewidth]{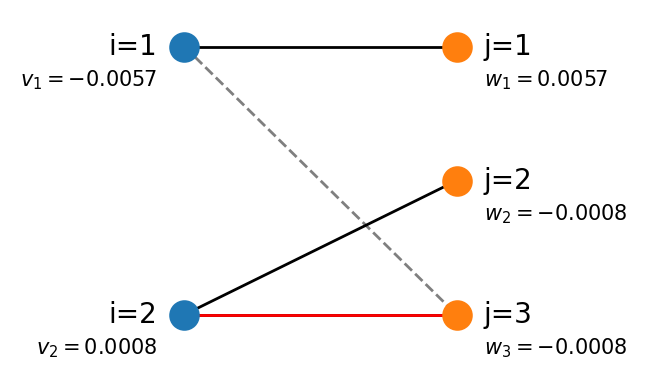}
        \caption{Phase $\ell=1$, $t \in [1.4, 2)$.}
        \label{fig:CS:CS_phase2}
    \end{subfigure}
    \hfill
    \begin{subfigure}[t]{0.32\textwidth}
        \centering
        \includegraphics[width=\linewidth]{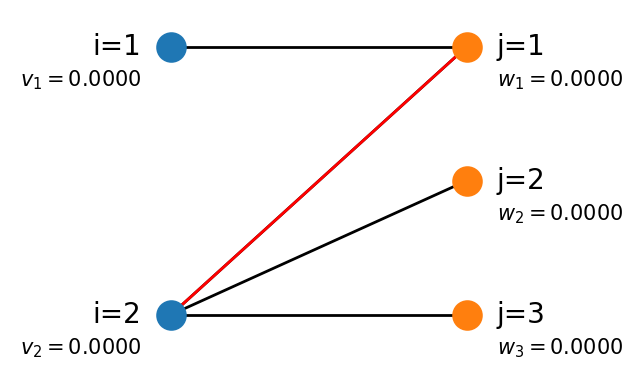}
        \caption{Phase $\ell=L=2$, $t \in [2, \infty)$.}
        \label{fig:CS:CS_phase3}
    \end{subfigure}
    \caption{Behavior of the cold Sinkhorn dynamics on \eqref{eq:intro:OT} for $m=2, n=3$, non-uniform target marginals, and a generic cost matrix.
    Bottom row: the bipartite graph with edge set $\EEE^\ell$ for $0 \leq \ell \leq L=2$. New edges $(i,j) \in \EEE^\ell \setminus \EEE^{\ell-1}$ are represented in red.
    Edges $(i,j) \in \EEE^\ell \setminus \ol\SSS^\ell$, where $\ol\SSS^\ell = \left\{(i,j) \in \EEE^\ell;~ \ol v^\ell_i + \ol w^\ell_j = 0 \right\}$, are represented as dashed lines.}
    \label{fig:CS:CS}
\end{figure}

\subsection{Properties of the dynamics} \label{subsec:CS:ppties}

In this section, we analyze the cold Sinkhorn dynamics.
In all of the following lemmas, we consider $(f(t), g(t))$, $(t_\ell, \ol v^\ell, \ol w^\ell, \EEE^\ell)_{\ell \leq L}$, and $t_{L+1}$ defined as in \autoref{def:CS:formaldef}.

\begin{lemma} \label{lm:CS:ftgt_partialFF}
    For any $t$, $(f(t), g(t)) \in \partial \FF$.
\end{lemma}

\begin{proof}
    Let us show by induction that for any $\ell \leq L$, we have $\forall t \in [t_\ell, t_{\ell+1}), (f(t), g(t)) \in \partial \FF$.
    
    For $\ell=0$, by definition $(f(0), g(0)) = (\ol f^0, \ol g^0) \in \partial \FF$.
    By definition of $t_1$, for all $t \in [t_0, t_1)$, $f_i(t) + g_j(t) - C_{ij} \leq 0$, i.e., $(f(t), g(t)) \in \FF$.
    To prove that it lies on the Pareto frontier, denote
    $\SSS(t) = \left\{ (i,j);~ f_i(t) + g_j(t) - C_{ij} = 0 \right\}$
    and
    $\ol \SSS^0 = \left\{ (i,j) \in \EEE^0;~ \ol v^0_i + \ol w^0_j = 0 \right\}$.
    By \autoref{thm:prelim:baradat_ventre}, $(\{1 \dots m\} \sqcup \{1 \dots n\}, \ol \SSS^0)$ has no isolated vertex.
    Now for any $t \in [t_0, t_1)$, by definition,
    $\SSS(t) \supset \ol \SSS^0$. 
    Thus $(\{1 \dots m\} \sqcup \{1 \dots n\}, \SSS(t))$ has no isolated vertex, and so by \autoref{lm:prelim:partialFF}, $(f(t), g(t)) \in \partial \FF$.

    For any $\ell \geq 1$, by induction hypothesis and continuity, 
    $(f(t_\ell), g(t_\ell)) \in \partial \FF$.
    A similar reasoning as in the case $\ell=0$ shows that at subsequent times $t \in [t_\ell, t_{\ell+1})$, we also have $(f(t), g(t)) \in \partial \FF$.
\end{proof}

\begin{lemma} \label{lm:CS:ddtPsi0}
    The dual OT objective $t \mapsto \Psi_0(f(t), g(t))$ is non-increasing and
    \begin{align}
        \forall \ell \leq L,~
        \forall t \in [t_\ell, t_{\ell+1}),~~
        \frac{d}{dt} \Psi_0(f(t), g(t))
        &= -\Hdiv{\mu}{\mu^*(\EEE^\ell)}
        - \Hdiv{\nu}{\nu^*(\EEE^\ell)} \\
        &= -\Hdiv{\mu^*(\EEE^\ell)}{\mu}
        - \Hdiv{\nu^*(\EEE^\ell)}{\nu}.
    \end{align}
\end{lemma}

\begin{proof}
    Let $\ell \leq L$ and $t \in [t_\ell, t_{\ell+1})$. By definition of $\Psi_0(f, g) = -f^\top \mu - g^\top \nu$ and of $\ol v^\ell_i = v^*_i(\EEE^\ell) = \log \left( \mu_i / \mu^*_i(\EEE^\ell) \right)$ and $\ol w^\ell_j$,
    \begin{multline}
        \frac{d}{dt} \Psi_0(f(t), g(t))
        = -\sum_{i=1}^m \mu_i \, \frac{d}{dt} f_i(t)
        - \sum_{j=1}^n \nu_j \, \frac{d}{dt} g_j(t)
        = -\sum_{i=1}^m \mu_i \, v^*_i(\EEE^\ell)
        - \sum_{j=1}^n \nu_j \, w^*_j(\EEE^\ell) \\
        = -\sum_{i=1}^m \mu_i \log \left( \mu_i / \mu^*_i(\EEE^\ell) \right)
        - \sum_{j=1}^n \nu_j \log \left( \nu_j / \nu^*_j(\EEE^\ell) \right)
        = -\Hdiv{\mu}{\mu^*(\EEE^\ell)}
        - \Hdiv{\nu}{\nu^*(\EEE^\ell)}.
    \end{multline}
    This proves the first equality of the lemma, and the second equality follows from the fact that
    $\Hdiv{\mu}{\mu^*(\EEE^\ell)} = \Hdiv{\nu^*(\EEE^\ell)}{\nu}$ and vice-versa,
    as proved in \cite[Remark~3.4]{baradat2024convergence}.
\end{proof}

\begin{lemma} \label{lm:CS:HdivHdiv_decr}
    The sequences $\left(
        \Hdiv{\mu^*(\EEE^\ell)}{\mu}
    \right)_{\ell \leq L}$
    and
    $\left(
        \Hdiv{\nu^*(\EEE^\ell)}{\nu}
    \right)_{\ell \leq L}$
    are non-increasing,
    and
    $\left(
        \Hdiv{\mu^*(\EEE^\ell)}{\mu} + \Hdiv{\nu^*(\EEE^\ell)}{\nu}
    \right)_{\ell \leq L}$
    is strictly decreasing.
\end{lemma}

\begin{proof}
    Fix $\ell < L$. Let us first show that $\Hdiv{\mu^*(\EEE^{\ell+1})}{\mu} \leq \Hdiv{\mu^*(\EEE^\ell)}{\mu}$, and the corresponding statement for the $\nu$'s will follow similarly.
    Let $\ol\SSS^\ell = \left\{ (i,j) \in \EEE^\ell;~ \ol v^\ell_i + \ol w^\ell_j = 0 \right\}$.
    By definition of the algorithm, ${\EEE^{\ell+1} \supset \ol\SSS^\ell}$.
    So by definition of
    $\mu^*(\EEE) = \argmin_{\ol\mu} \Hdiv{\ol\mu}{\mu}$
    subject to
    $\exists Q \in \Delta_\EEE;~ X_\sharp Q = \ol\mu$ and $Y_\sharp Q = \nu$,
    we have
    $\Hdiv{\mu^*(\EEE^{\ell+1})}{\mu} \leq \Hdiv{\mu^*(\ol\SSS^\ell)}{\mu}$.
    Now we showed in \autoref{lm:prelim:equal_mu*} that $\mu^*(\ol\SSS^\ell) = \mu^*(\EEE^\ell)$, hence the announced inequality.
    
    This shows that 
    $\Hdiv{\mu^*(\EEE^{\ell+1})}{\mu} \leq \Hdiv{\mu^*(\EEE^\ell)}{\mu}$
    and 
    $\Hdiv{\nu^*(\EEE^{\ell+1})}{\nu} \leq \Hdiv{\nu^*(\EEE^\ell)}{\nu}$.
    Now let us show that at least one of these inequalities is strict.
    Suppose by contradiction that they are both equalities.
    Since $\ol\SSS^\ell \subset \EEE^{\ell+1}$, then $\mu^*(\EEE^\ell) = \mu^*(\ol\SSS^\ell)$ is feasible for the optimization problem defining $\mu^*(\EEE^{\ell+1})$, and it is also optimal by the contradiction hypothesis, so $\mu^*(\EEE^\ell) = \mu^*(\EEE^{\ell+1})$.
    Likewise, $\nu^*(\EEE^\ell) = \nu^*(\EEE^{\ell+1})$.
    Consequently, $\ol v^\ell = \ol v^{\ell+1}$ and $\ol w^\ell = \ol w^{\ell+1}$.
    But on the other hand, for $(i_0, j_0) \in \EEE^{\ell+1} \setminus \ol\SSS^\ell$---which exists by definition of the algorithm since $\ell < L$---we necessarily have $\ol v^\ell_{i_0} + \ol w^\ell_{j_0} > 0$ and $\ol v^{\ell+1}_{i_0} + \ol w^{\ell+1}_{j_0} \leq 0$. 
    Thus,
    $\Hdiv{\mu^*(\EEE^{\ell+1})}{\mu} < \Hdiv{\mu^*(\EEE^\ell)}{\mu}$
    or
    $\Hdiv{\nu^*(\EEE^{\ell+1})}{\nu} < \Hdiv{\nu^*(\EEE^\ell)}{\nu}$,
    and so the sequence 
    $\left(
        \Hdiv{\mu^*(\EEE^\ell)}{\mu} + \Hdiv{\nu^*(\EEE^\ell)}{\nu}
    \right)_{\ell \leq L}$
    is decreasing.
\end{proof}

\begin{remark}
    \autoref{lm:CS:ddtPsi0} and \autoref{lm:CS:HdivHdiv_decr} can be viewed as the $\tau \to 0$ limits of \cite[Lemma~2]{altschuler2017near} and \cite[Proposition~6.10]{nutz2021introduction}, respectively.
\end{remark}

\begin{lemma} \label{lm:CS:L_finite}
    The dynamics has a finite number of phases: $L+1 \leq 2^{mn}$.
\end{lemma}

\begin{proof}
    By the previous lemma, the sets $(\EEE^\ell)_{0 \leq \ell \leq L}$ are pair-wise distinct. Now $\EEE^\ell \subset \{1 \dots m \} \times \{1 \dots n\}$ for each $\ell$, so there can be at most $2^{mn}$ such sets. So $L+1 \leq 2^{mn} < \infty$.
\end{proof}

\begin{lemma} \label{lm:CS:finalphase_optimal}
    At the final phase, $(\ol v^L, \ol w^L) = (0,0)$, and the final point $(f(t_L), g(t_L))$ is an optimal solution of the dual OT problem.
\end{lemma}

\begin{proof}
    Denote by $\Psi_0^* \in \RR$ the optimal value of the dual OT problem:
    $\Psi_0^* = \min_{(f, g) \in \FF} \Psi_0(f, g)$.
    Since $\Psi_0(f(t), g(t))$ is piecewise-linear, non-increasing, and lower-bounded by $\Psi_0^*$, its time-derivative during the final phase $t \in [t_L, \infty)$ must be zero. So by \autoref{lm:CS:ddtPsi0},
    $\Hdiv{\mu^*(\EEE^L)}{\mu} = \Hdiv{\nu^*(\EEE^L)}{\nu} = 0$,
    or equivalently,
    $v^*(\EEE^L) = \ol v^L = 0$ and $w^*(\EEE^L) = \ol w^L = 0$.

    Denote for concision $f^\infty = f(t_L)$ and $g^\infty = g(t_L)$ and let us show that $(f^\infty, g^\infty)$ is an optimal solution of the dual OT problem. 
    By duality, it suffices to show that there exists $\pi \in \Delta_{m \times n}$ which is feasible for the primal OT problem and satisfies the complementary slackness condition
    $\forall i,j,~ \pi_{ij} (f^\infty_i + g^\infty_j - C_{ij}) = 0$.
    Now by definition of $\mu^*(\cdot)$, the fact that $\mu^*(\EEE^L) = \mu$ implies the existence of some $Q \in \Delta_{\EEE^L}$ such that 
    $X_\sharp Q = \mu$ and $Y_\sharp Q = \nu$.
    By setting $\pi_{ij} = Q_{ij}$ if $(i,j) \in \EEE^L$ and $0$ otherwise, we indeed have that $\pi$ is feasible for the primal OT problem and that for any $i, j$, either $(i,j) \in \EEE^L$ and then $f^\infty_i + g^\infty_j - C_{ij} = 0$, or $(i,j) \not\in \EEE^L$ and then $\pi_{ij} = 0$.
\end{proof}

\begin{lemma} \label{lm:CS:t_L=O(1)}
    There exists a constant $T$ dependent only on $\mu, \nu$, and $C$ such that $t_L \leq T$.
\end{lemma}

\begin{proof}
    Let us first show that $\sup_{(f,g) \in \partial \FF} \Psi_0(f, g) < \infty$. Indeed, for any $(f, g) \in \partial \FF$,
    \begin{align}
        \forall j, j',~
        g_j - g_{j'}
        = g_0[f]_j - g_0[f]_{j'}
        &= \left( \min_i C_{ij} - f_i \right) - \left( \min_{i'} C_{i'j'} - f_{i'} \right) \\
        &= \min_i \max_{i'}
        C_{ij} - f_i - C_{i'j'} + f_{i'} 
        \geq \min_i\, C_{ij} - C_{ij'}
    \end{align}
    and so
    \begin{equation}
        \forall i,j,~~
        f_i + g_j
        = f_0[g]_i + g_j
        = \min_{j'} C_{ij'} - g_{j'} + g_j
        \geq \min_{i', j'}\, C_{ij'} + C_{i'j} - C_{i'j'}
    \end{equation}
    and so
    $\Psi_0(f, g)
    = -f^\top \mu - g^\top \nu
    = -\sum_{i,j} (f_i + g_j) \mu_i \nu_j
    \leq -\left( \min_{i,j,i',j'}\, C_{ij'} + C_{i'j} - C_{i'j'} \right)$.

    By \autoref{lm:CS:ddtPsi0}, for any $t < t_L$,
    $
        \frac{d}{dt} \Psi_0(f(t), g(t))
        \leq -\min \big[ \Hdiv{\mu}{\mu^*(\EEE')}
        + \Hdiv{\nu}{\nu^*(\EEE')} \big] < 0
    $
    where~$\EEE'$ ranges over all subsets of $\{1 \dots m\} \times \{1 \dots n\}$ such that $(\mu^*(\EEE), \nu^*(\EEE)) \neq (\mu, \nu)$.
    On the other hand,
    $\Psi_0(f(0), g(0)) \leq \sup_{(f,g) \in \partial \FF} \Psi_0(f,g) < \infty$
    and $\Psi_0(f(t_L), g(t_L)) = \min_{\FF} \Psi_0 > -\infty$.
    So
    \begin{align}
        t_L \cdot
        \min \big[ \Hdiv{\mu}{\mu^*(\EEE')} + \Hdiv{\nu}{\nu^*(\EEE')} \big]
        &\leq -\int_0^{t_L} \frac{d}{dt} \Psi_0(f(t), g(t)) \,\d t \\
        &= \Psi_0(f(0), g(0)) - \Psi_0(f(t_L), g(t_L)) \\
        &\leq \sup_{(f,g) \in \partial \FF} \Psi_0(f,g) - \min_{\FF} \Psi_0
    \end{align}
    and so $t_L \leq \frac{
        \sup_{(f,g) \in \partial \FF} \Psi_0(f,g) - \min_{\FF} \Psi_0
    }{
        \min \big[ \Hdiv{\mu}{\mu^*(\EEE')} + \Hdiv{\nu}{\nu^*(\EEE')} \big]
    }$,
    which is indeed only dependent on $\mu, \nu$, and $C$.
\end{proof}

Together, these lemmas show that 
\emph{the cold Sinkhorn dynamics is a continuous-time optimization algorithm for the dual unregularized OT problem \eqref{eq:prelim:dualOT}, that walks piecewise-linearly along the Pareto frontier $\partial \FF$ of the feasibility polytope and converges after a finite number of phases $L$ and a finite time $t_L$.}
As such, it can be thought of as a simplex-type algorithm for dual OT problems.

\begin{algorithm}[t]
\caption{Direct implementation of the cold Sinkhorn dynamics}
\label{alg:CS:CS}
    \KwIn{$\mu \in \Delta_m, \nu \in \Delta_n, C \in \RR^{m \times n}, f^0 \in \RR^m, g^0 \in \RR^n$}
    $t_0 = 0$ \;
    $f(0) = f_0[g^0], ~ g(0) = g_0[f(0)]$ \;
    \For{$\ell=0, ..., 2^{mn}$}{
        $\ol U_{ij} = f_i(t_\ell) + g_j(t_\ell) - C_{ij}$ ~~ ($\leq 0$ for all $i,j$ by construction) \;
        $\EEE^\ell = \left\{ (i,j);~ \ol U_{ij} = 0 \right\}$ \;
        $\ol v^\ell = v^*(\EEE^\ell), ~ \ol w^\ell = w^*(\EEE^\ell)$ \;
        \If{$\forall i,j,~ \ol v^\ell_i + \ol w^\ell_j \leq 0$}{
            $L = \ell$ \;
            \Return{$f(t_L), g(t_L)$, an optimal solution of the dual OT problem \eqref{eq:prelim:dualOT}}
        }
        $\forall (i,j) \not\in \EEE^\ell,~ \Delta_{ij} = -\ol U_{ij} / (\ol v^\ell_i + \ol w^\ell_j)$ if $\ol v^\ell_i + \ol w^\ell_j > 0$ and $+\infty$ otherwise \;
        $t_{\ell+1} = t_\ell + \min_{ij} \Delta_{ij}$ \;
        $f(t_{\ell+1}) = f(t_\ell) + (t_{\ell+1}-t_\ell) \ol v^\ell,~ g(t_{\ell+1}) = g(t_\ell) + (t_{\ell+1}-t_\ell) \ol w^\ell$
    }
\end{algorithm}

\vspace{0.8em}
\begin{remark}
    In principle, the cold Sinkhorn dynamics \autoref{def:CS:formaldef} can be implemented directly, instead of viewing it as a limit of Sinkhorn, and this provides an algorithm for exact unregularized OT computation (\autoref{alg:CS:CS}).
    The only costly steps are to compute the drifts $v^*(\EEE^\ell), w^*(\EEE^\ell)$ at each phase,
    but these computations can be amortized: for a fixed pair of marginals $\mu, \nu$, one can pre-compute all of the $v^*(\EEE), w^*(\EEE)$ for $\EEE \subset \{1 \dots m\} \times \{1 \dots n\}$, and simply look up the values of $v^*(\EEE^\ell), w^*(\EEE^\ell)$ when actually running the dynamics.
    Because the $v^*(\EEE), w^*(\EEE)$ do not depend on the cost matrix $C$, these pre-computed values can be reused for an arbitrary number of OT problems, provided that they share the same target marginals $\mu, \nu$.

    However, even with the pre-computed values of $v^*(\EEE), w^*(\EEE)$, the iteration complexity of one run of the cold Sinkhorn dynamics that we can guarantee is $O(L) \leq O(2^{mn})$, with a per-iteration cost of $O(mn)$;
    moreover, the number of values to pre-compute is $O(2^{mn})$.
    So such an approach would be rather impractical, and much slower than classical exact linear solvers a priori.
\end{remark}

\section{Convergence of Sinkhorn to cold Sinkhorn} \label{sec:cv}

\subsection{Main result} \label{subsec:cv:mainres}

In this section, we state and prove the main result of this paper: the Sinkhorn iterates converge uniformly to the cold Sinkhorn dynamics as $\tau \to 0$.
We do not track the constants appearing in the convergence bound explicitly here, so we do not quantify how small $\tau$ must be for the bound to be meaningful; this aspect is left for future work.
Our main result is as follows.

\begin{theorem} \label{thm:cv:cv}
    Let $(f^k, g^k)_k$ denote the iterates of the Sinkhorn algorithm applied to \eqref{eq:intro:EOT} initialized at some $(f^0, g^0)$.
    Let $(f(t), g(t))$ denote the cold Sinkhorn dynamics applied to \eqref{eq:intro:OT} initialized at $(f(0), g(0)) = (f_0[g^0], g_0[f_0[g^0]]) \in \partial \FF$.
    Then for any $\tau \leq \tau_0$,
    \begin{equation}
        \sup_{0 \leq t \leq t_L}
        \norm{f^{\floor{2t/\tau}+2} - f(t)}_\infty
        + \norm{g^{\floor{2t/\tau}+2} - g(t)}_\infty
        \leq B \, \tau (\log 1/\tau)^{2L+1},
    \end{equation}
    for some constants $\tau_0, B$ dependent only on $\mu, \nu$, and $C$.
    Here we recall that $L \leq 2^{mn}-1$ is the number of phases of the cold Sinkhorn dynamics (excluding the stationary final phase $[t_L, \infty)$).
\end{theorem}

Moreover, uniform convergence also holds over an arbitrarily large portion of the final phase, albeit with constants that may depend on the initialization.

\begin{theorem} \label{thm:cv:cv_finalphase}
    In the setting of the previous theorem, for any $\tau \leq \tau_0'$,
    \begin{equation}
        \sup_{0 \leq t \leq t_L + \tau^{-1}}
        \norm{f^{\floor{2t/\tau}+2} - f(t)}_\infty 
        + \norm{g^{\floor{2t/\tau}+2} - g(t)}_\infty
        \leq B' \, \tau (\log 1/\tau)^{2L+2}
    \end{equation}
    for some constants $\tau_0', B'$ dependent on $\mu, \nu, C$, and the initialization $(f^0, g^0)$.
\end{theorem}

\begin{remark}[Uniform convergence of the derivatives within phases]
    The theorems presented above show that the Sinkhorn iterates $(f^k, g^k)$ converge uniformly to the cold Sinkhorn dynamics.
    One may also ask for uniform convergence of the derivatives, i.e., one may ask whether the rescaled one-iteration increments $v^k, w^k$ defined in \eqref{eq:prelim:def_vk_wk} converge to the $\ol v^\ell, \ol w^\ell$ appearing in \autoref{def:CS:formaldef}.
    This is indeed the case uniformly away from the phase transitions: by inspecting step~\eqref{eq:cv:pf_cv_vkwk} of the proof of \autoref{thm:cv:cv}, one can show that
    for any $\tau \leq \tau_0$,
    \begin{equation}
        \forall \ell \leq L-1,~
        \sup_{t_\ell+\tau \log(\sfrac1\tau)^{2L+1} \leq t \leq t_{\ell+1}-\tau \log(\sfrac1\tau)^{2L+1}} 
        \norm{v^{\floor{2t/\tau}} - \ol v^\ell} 
        + \norm{w^{\floor{2t/\tau}} - \ol w^\ell} 
        \leq B\, \tau
    \end{equation}
    for some constants $\tau_0, B$ dependent only on $\mu, \nu$, and $C$.
\end{remark}

\begin{remark}[Impossibility of uniform convergence over all time] \label{rk:cv:centroid}
    This remark shows that the conclusion of \autoref{thm:cv:cv_finalphase} cannot hold if the $\sup$ is taken over all $t \geq 0$, in general.
    
    Note that
    $\Psi_\tau$ and $\Psi_0$ only depend on their arguments $f, g$ via the matrix $\ol U = (f_i+g_j-C_{ij})_{ij}$.
    Specifically,
    $\Psi_\tau(f,g) = \sum_{ij} \left[ \tau (e^{\ol U_{ij}/\tau} - 1) -\ol U_{ij} - C_{ij} \right] \mu_i \nu_j$,
    and likewise for $\Psi_0$ without the term in $\tau$.
    Note that $\Psi_\tau$ is strictly convex in $\ol U$ for any $\tau>0$, so that
    $
        \big\{ (f^*_i+g^*_j-C_{ij})_{ij},~ (f^*, g^*) \in \argmin \Psi_\tau \big\} 
        = \{ \ol U^*_\tau \}
    $
    is a singleton.
    For $\Psi_0$, the corresponding set rewrites
    \begin{align*}
        \UUU
        &\coloneqq \left\{ (f^*_i+g^*_j-C_{ij})_{ij},~ (f^*, g^*) \in \argmin\nolimits_{\FF} \Psi_0 \right\} \\
        &~= \argmax_{\ol U}\,
        \sum_{ij} \ol U_{ij} \mu_i \nu_j
        ~~~~\text{subject to}~~~~
        \forall i,j,~ \ol U_{ij} \leq 0
        \quad\text{and}\quad
        \exists f,g;~ \ol U+C = f \cdot \bmone_n^\top + \bmone_m \cdot g^\top,
    \end{align*}
    a polytope which is not a singleton in general, if $C$ is not in general position w.r.t\ $\mu$ and $\nu$.
    When this polytope $\UUU$ is not a singleton, it was shown by \cite[Section~3]{cominetti1994asymptotic} that it has a distinguished element $\ol U^*$ such that
    $\ol U^*_\tau \to \ol U^*$ as $\tau \to 0$.

    Now consider $\mu, \nu, C$ such that $\UUU$ is not a singleton, and consider any $(f^0, g^0)$ such that $\ol U^0 = (f^0_i + g^0_j - C_{ij})_{ij} \in \UUU \setminus \{\ol U^*\}$.
    Then the cold Sinkhorn dynamics initialized at $(f^0, g^0)$ is constant since $\ol U^0$ is already optimal, while the Sinkhorn algorithm initialized at $(f^0, g^0)$ for any $\tau>0$ converges to $\ol U^*_\tau$.
    Hence 
    \begin{equation}
        \lim_{\tau \to 0}~
        \lim_{t \to \infty}~
        \max_{ij} \abs{f^{\floor{2t/\tau}+2}_i - f_i(t)
        + g^{\floor{2t/\tau}+2}_j - g_j(t)}
        = \lim_{\tau \to 0} \norm{\ol U^*_\tau - \ol U^0}_\infty
        = \norm{\ol U^* - \ol U^0}_\infty > 0,
    \end{equation}
    implying that $\sup_{t \geq 0} \norm{f^{\floor{2t/\tau}+2} - f(t)}_\infty + \norm{g^{\floor{2t/\tau}+2} - g(t)}_\infty$ is not $o_\tau(1)$ in general.
\end{remark}

\vspace{1em}
The remainder of this section is dedicated to the proofs of \autoref{thm:cv:cv}, \autoref{thm:cv:cv_finalphase}.
We start by a simple lemma quantifying the rate of convergence of $f_\tau[\cdot]$ to $f_0[\cdot]$ as $\tau \to 0$.

\begin{lemma} \label{lm:cv:logsumexp_rate}
    For any $\varphi \in \RR^N$ and $p \in \Delta_N$,
    \begin{equation}
        \log p_{\min} \leq
        \log \sum\nolimits_I e^{\varphi_I} p_I
        - \max\nolimits_I \, \varphi_I
        \leq 0.
    \end{equation}
    In particular, the mappings $f_\tau[\cdot]$ and $g_\tau[\cdot]$ defined in \autoref{subsec:prelim:bg} satisfy
    \begin{equation}
        \forall g \in \RR^n,~
        \norm{f_\tau[g]-f_0[g]}_\infty \leq -\tau \log \nu_{\min},
        \qquad~~
        \forall f \in \RR^m,~
        \norm{g_\tau[f]-g_0[f]}_\infty \leq -\tau \log \mu_{\min}.
    \end{equation}
\end{lemma}

\begin{proof}
    Let $\varphi \in \RR^N, p \in \Delta_N$ and denote
    $m = \max_I \varphi_I$,
    $S = \argmax_I \varphi_I$,
    $p(S) = \sum_{I \in S} p_I$.
    We have
    \begin{align}
        \log \sum\nolimits_I e^{\varphi_I} p_I - m
        &= \log \sum\nolimits_I e^{\varphi_I-m} p_I
        = \log \bigg( \sum_{I \in S} p_I + \sum_{I \not\in S} e^{\varphi_I-m} p_I \bigg) \\
        &= \log p(S)
        + \log \Bigg( 1 + \frac{\sum_{I \not\in S} e^{\varphi_I-m} p_I}{p(S)} \Bigg)
        \geq \log p(S)
        \geq\log p_{\min}
    \end{align}
    and in the other direction,
    $\log \sum_I e^{\varphi_I} p_I - m 
    = \log \sum\nolimits_I e^{\varphi_I-m} p_I
    \leq 0$.
    The uniform bound on $f_\tau[\cdot]-f_0[\cdot]$ follows by definition of $f_\tau[g]_i = -\tau \log \sum_j e^{[-C_{ij}+g_j]/\tau} \nu_j$,
    and likewise for $g_\tau[\cdot]$. 
\end{proof}

We can now proceed to the proof of \autoref{thm:cv:cv}.

\begin{proof}[Proof of \autoref{thm:cv:cv}]
    Throughout this proof, we use $O(\cdot), \Omega(\cdot), \Theta(\cdot)$ to hide constants that depend only on $\mu, \nu$, and $C$.
    Note that $t_1 \leq ... \leq t_L = O(1)$ by \autoref{lm:CS:t_L=O(1)}.
    Moreover, we will say ``for $\tau$ small enough'' to mean that a statement holds provided $\tau \leq \tau_0$, for some constant $\tau_0 = \Theta(1)$.
    
    Let us first check that for $t=t_0=0$, we have the pointwise estimate $(f^2, g^2) = (f(0), g(0)) + O(\tau)$.
    Since $f^2=f^1=f_\tau[g^0]$ and $g^2 = g_\tau[f^1]$ by definition, this indeed follows from \autoref{lm:cv:logsumexp_rate} and from the $1$-Lipschitzness of $g_0[\cdot]$.
    For ease of notation, let us re-index the Sinkhorn iterates by shifting back the index $k$ by $2$.
    That is, we consider henceforth the sequence $(f^k, g^k)_{k \geq -2}$ initialized at some arbitrary $(f^{-2}, g^{-2})$, and we have $(f^0, g^0) \to (f(0), g(0)) = (f_0[g^{-2}], g_0[f_0[g^{-2}]])$ as $\tau \to 0$.

    To prove the theorem, it suffices to show that for any $0 \leq \ell \leq L-1$,
    for $\tau$ small enough,
    \begin{equation}
        \sup_{t \in [t_\ell, t_{\ell+1}]}
        \abs{f^{\floor{2 t/\tau}} - f(t)} 
        + \abs{g^{\floor{2 t/\tau}} - g(t)}
        \leq O\left( \tau (\log 1/\tau)^{2\ell+3} \right).
    \end{equation}
    Since $t \mapsto (f(t), g(t))$ is $O(1)$-Lipschitz by definition and $t \mapsto (f^{\floor{2 t/\tau}}, g^{\floor{2 t/\tau}})$ is $O(1)$-Lipschitz by \autoref{lm:prelim:unifbound_vk_wk},
    then equivalently it suffices to show, for any $0 \leq \ell \leq L-1$ and $\tau$ small enough,
    \begin{equation} \label{eq:cv:induction_claim}
        \sup_{\floor{2t_\ell/\tau} \leq k \leq \floor{2t_{\ell+1}/\tau}}
        \abs{f^k - f(\tau k/2)}
        + \abs{g^k - g(\tau k/2)}
        \leq O\left( \tau (\log 1/\tau)^{2\ell+3} \right).
    \end{equation}
    We show this by induction.

    \vspace{0.5em}
    \underline{For $\ell=0$:}~
    Denote by $(\tf^k, \tg^k)_k$ the iterates of the Sinkhorn algorithm applied to the EOT problem with the same target marginals $\mu, \nu$ but with the cost matrix
    $\tC_{ij} = \begin{cases}
        C_{ij} ~~\text{if}~ (i,j) \in \EEE^0 \\
        \infty ~~\text{otherwise}
    \end{cases}$,
    initialized at $(\tf^0, \tg^0) = (f^0, g^0)$.
    Also let $\tv^k, \tw^k$ and $\tU^k$ denote the associated variables as defined in \eqref{eq:prelim:def_vk_wk}, \eqref{eq:prelim:def_Uk}.
    Since $(\tf^0, \tg^0) = (f(0), g(0)) + O(\tau)$ and $(f(0), g(0)) \in \partial \FF$, then
    \begin{equation}
        \forall (i,j) \in \EEE^0,~
        \tau \tU^0_{ij} = \tf^0_j + \tg^0_j - C_{ij} = O(\tau).
    \end{equation}
    So we can apply \autoref{lm:prelim:unifbound_vk_wk} and \autoref{prop:prelim:qtve_BV24} with $\delta = O(1)$, and we get the following---where we still keep track of the dependency on $\delta$ explicitly for ease of presentation later on, and where we assume $\delta \geq \Omega(1)$ without loss of generality:
    \begin{align*}
        \forall k \geq K_0 = \Theta(\delta),~~
        & \norm{(\tv^k, \tw^k) - (\olv^0, \ol w^0)}_\infty
        = \norm{(\tv^k, \tw^k) - (v^*(\EEE^0), w^*(\EEE^0))}_\infty
        \leq O(1) \frac{\delta + \log k}{k} \\
        \text{and}~~~~
        \forall 1 \leq k \leq 3,~~
        & \norm{\tv^k}_\infty, \norm{\tw^k}_\infty \leq O(\delta) \\
        \forall k \geq 4,~~
        & \norm{\tv^k}_\infty, \norm{\tw^k}_\infty \leq O(1).
    \end{align*}
    As a consequence,
    since $\tf^k = \tf^{k-2} + \tau \tv^k$
    and $f(\tau k/2) = f(\tau (k-2)/2) + \tau \ol v^0$
    for all $k$,
    \begin{itemize}
        \item For all $k \leq K_0 = \Theta(\delta)$,
        \begin{equation}
            \norm{\tf^k - f(\tau k/2)}_\infty
            \leq \norm{\tf^k - \tf^4}_\infty
            + \norm{\tf^4 - \tf^0}_\infty
            + \norm{f(\tau k/2) - f(0)}_\infty
            \leq O(\tau k + \delta)
            = O(\tau \delta).
        \end{equation}
        \item For all $K_0 \leq k \leq \floor{2 t_1/\tau}$,
        \begin{equation}
            \norm{\tf^k - f(\tau k/2)}_\infty
            = \norm{
                \tf^{K_0} 
                - f(\tau K_0/2)
                + \sum_{\substack{l=K_0 \\ l ~\text{even}}}^k
                \left( 
                    \tau \tv^l
                    - \tau \olv^0
                \right)
            }_\infty 
            \! \leq O(\tau \delta) 
            + \sum_{\substack{l=K_0 \\ l ~\text{even}}}^k
            \tau\, O(1) \frac{\delta + \log l}{l}
        \end{equation}
        and so since $\sum_{l=1}^k l^{-1} \log l \asymp \frac12 (\log k)^2$,
        \begin{equation}
            \norm{\tf^k - f(\tau k/2)}_\infty
            \leq O(\tau (\log k) (\delta + \log k))
            \leq O\big( \tau (\log 1/\tau) (\delta + \log 1/\tau) \big).
        \end{equation}
    \end{itemize}
    We have likewise the analogous bound for the $\tg^k$.
    
    Next, let us show that $(f^k, g^k)$ remains close to $(\tf^k, \tg^k)$ throughout the phase $[t_0, t_1)$, or to be exact, throughout some $[t_0, t_1-o_\tau(1))$.
    Since $\forall (i,j) \in \EEE^0, \tU^0_{ij} = O(1)$ as noted above, we can apply \autoref{lm:prelim:control_finite_infinite_costs} with $\EEE = \EEE^0$ and $\delta = O(1)$, yielding---still keeping track of the dependency on $\delta$ explicitly for ease of later presentation:
    \begin{equation}
        \forall k \geq 0,~
        \norm{f^k-\tf^k}_\infty\!,~
        \norm{g^k-\tg^k}_\infty
        \leq \tau\, O(1)\, e^{2\delta}
        \sum_{l=0}^{k-1}
        e^{-M^l}
        ~\quad \text{where} \quad~
        -M^k = \max_{(i,j) \not\in \EEE^0} [f^k_i+g^k_j-C_{ij}]/\tau.
    \end{equation}
    Denote likewise
    \begin{equation}
        -\tM^k = \max_{(i,j) \not\in \EEE^0} [\tf^k_i+\tg^k_j-C_{ij}]/\tau
        \qquad\text{and}\qquad
        -\ol M(t) = \max_{(i,j) \not\in \EEE^0} f_i(t)+g_j(t)-C_{ij}.
    \end{equation}
    Fix $K_0' = \floor{2 t_1/\tau} - (\log 1/\tau)^N$ for some large integer $N$ to be chosen later.
    Then for any $k \leq K_0'$,
    \begin{align*}
        \abs{M^k - \tau^{-1} \ol M(\tau k/2)}
        &\leq \abs{M^k - \tM^k} + \abs{\tM^k - \tau^{-1} \ol M(\tau k/2)} \\
        &\leq \tau^{-1} \norm{(f^k, g^k)-(\tf^k, \tg^k)}_\infty
        + \tau^{-1} \norm{(\tf^k, \tg^k) - (f(\tau k/2), g(\tau k/2))}_\infty \\
        &\leq O(1)\, e^{2\delta} \sum_{l=0}^{k-1} e^{-M^l}
        + O((\log 1/\tau) (\delta + \log 1/\tau)) \\
        &\leq O(1)\, e^{2\delta} \sum_{l=0}^{k-1} e^{-\tau^{-1} \ol M(\tau l/2)}~
        e^{\abs{M^l - \tau^{-1} \ol M(\tau l/2)}}
        + O((\log 1/\tau) (\delta + \log 1/\tau)).
    \end{align*}
    That is, denoting $u_k = \abs{M^k - \tau^{-1} \ol M(\tau k/2)}$, $b_k = \Theta(1)\, e^{2\delta} e^{-\tau^{-1} \ol M(\tau k/2)}$, and $c = \Theta((\log 1/\tau) (\delta + \log 1/\tau))$,
    \vspace{-0.5em}
    \begin{equation}
        \forall 0 \leq k \leq K_0',~ 
        u_k \leq c + \sum_{l=0}^{k-1} b_l \, e^{u_l}.
    \end{equation}
    So by the discrete Bihari-LaSalle inequality \cite[Theorem~2.3.1]{pachpatte2001inequalities},
    \begin{align}
        u_k
        &\leq -\log \bigg( e^{-c} 
        - \sum_{l=0}^{k-1} b_l \bigg) \\
        e^{u_k} =
        e^{\abs{M^k - \tau^{-1} \ol M(\tau k/2)}}
        &\leq \left( 
            e^{-\Theta\left( (\log 1/\tau) (\delta+\log 1/\tau) \right)} 
            - \Theta(1)\, e^{2\delta} \sum_{l=0}^{k-1} e^{-\tau^{-1} \ol M(\tau l/2)} 
        \right)^{-1}.
    \end{align}
    Moreover, by definition of $K_0' = \floor{2 t_1/\tau} - (\log 1/\tau)^N$ and of $t_1$, for all $k \leq K_0'$,
    \begin{equation} \label{eq:CS:olM_negative}
        -\ol M(\tau k/2) \leq 
        ~\underbrace{-\ol M(t_1)}_{0}~
        - \Theta\left( t_1-\tau k/2 \right)
        \leq - \Theta\left( \tau (\log 1/\tau)^N \right)
    \end{equation}
    where the first $\Theta(\cdot)$ hides the constant
    $\min \left\{ \ol v^0_i + \ol w^0_j;~ (i,j) \not\in \EEE^0 ~\text{and}~ \ol v^0_i + \ol w^0_j > 0 \right\}$,
    which is indeed only dependent on $\mu$ and $\nu$ upon taking an infimum over all possible sets $\EEE^0$, since by definition $\olv^0=v^*(\EEE^0), \ol w^0=w^*(\EEE^0)$.
    Hence,
    \begin{align}
        e^{\abs{M^k - \tau^{-1} \ol M(\tau k/2)}}
        &\leq \left( 
            e^{-\Theta\left( (\log 1/\tau) (\delta+\log 1/\tau) \right)} 
            - \Theta(1)\, e^{2\delta} \cdot \floor{2 t_1/\tau} e^{-\Theta((\log 1/\tau)^N)} 
        \right)^{-1} \\
        &= \left( 
            e^{-\Theta\left( (\log 1/\tau) (\delta+\log 1/\tau) \right)} 
            -\Theta(1)\, e^{2\delta + \log(1/\tau) - \Theta((\log 1/\tau)^N)}
        \right)^{-1} \\
        &= O \left(
            e^{\Theta\left( (\log 1/\tau) (\delta+\log 1/\tau) \right)} 
        \right)
    \end{align}
    for $\tau$ small enough,
    provided that $\delta + \log 1/\tau = o((\log 1/\tau)^{N-1})$.
    Thus, for all $k \leq K_0'$,
    \begin{align}
        \norm{f^k-\tf^k}_\infty
        &\leq \tau\, O(1)\, e^{2\delta}
        \sum_{l=0}^{k-1}
        e^{-M^l}
        \leq \tau\, O(1)\, e^{2 \delta} \cdot \floor{2 t_1/\tau} \max_{l \leq k-1} e^{-M^l} \\
        &\leq O(1)\, e^{2 \delta}\,
        \max_{l \leq k-1} 
        e^{-\tau^{-1} \ol M(\tau l/2)}~
        e^{\abs{M^l - \tau^{-1} \ol M(\tau l/2)}} \\
        &\leq O(1)\, e^{2 \delta}~
        e^{-\Theta((\log 1/\tau)^N)}~
        e^{\Theta\left( (\log 1/\tau) (\delta+\log 1/\tau) \right)}
        = O(1) e^{-\Theta((\log 1/\tau)^N)}
        \leq O(\tau), \qquad
    \label{eq:cv:pf_cv_vkwk}
    \end{align}
    again for $\tau$ small enough and provided that $\delta + \log 1/\tau = o((\log 1/\tau)^{N-1})$.
    In summary, we have 
    \begin{align}
        \forall 0 \leq k \leq K_0',~
        \norm{f^k-f(\tau k/2)}_\infty
        &\leq \norm{f^k-\tf^k}_\infty + \norm{\tf^k - f(\tau k/2)} \\
        &\leq O(\tau) + O\big( \tau (\log 1/\tau) (\delta + \log 1/\tau) \big)
    \end{align}
    for $\tau$ small enough and provided that $\delta + \log 1/\tau = o((\log 1/\tau)^{N-1})$,
    and likewise for the $g^k$.

    \pagebreak

    It only remains to treat the $K_0' \leq k \leq \floor{2 t_1/\tau}$ where $K_0' = \floor{2 t_1/\tau} - (\log 1/\tau)^N$. In this case,
    \begin{equation}
        \norm{f^k - f(\tau k/2)}_\infty
        \leq \underbrace{
            \norm{f^{K_0'} - f(\tau K_0'/2)}_\infty
        }_{\leq~ O( \tau (\log 1/\tau) (\delta + \log 1/\tau) )}
        +~ \underbrace{
            \norm{f^k - f^{K_0'}}_\infty
            + \norm{f(\tau k/2) - f(\tau K_0'/2)}_\infty
        }_{\leq~ O(\tau) (k-K_0') ~=~ O(\tau (\log 1/\tau)^N)}
    \end{equation}
    since
    $\norm{v^l}_\infty = O(1)$ for all $l \geq 3$
    by \autoref{lm:prelim:unifbound_vk_wk},
    and likewise for the $g^k$.
    
    Since $\delta = O(1)$, then by choosing $N=3$, we indeed have $\delta + \log 1/\tau = o((\log 1/\tau)^{N-1}))$, and so we obtain the claimed inequality \eqref{eq:cv:induction_claim} at rank $\ell=0$.
    
    \vspace{0.5em}
    \underline{For $1 \leq \ell \leq L-1$:}~
    Suppose the induction hypothesis \eqref{eq:cv:induction_claim} holds at rank $\ell-1$.
    In particular, 
    $(f^{\floor{2t_\ell/\tau}}, g^{\floor{2t_\ell/\tau}})
    = (f(t_\ell), g(t_\ell)) + O\left( \tau (\log 1/\tau)^{2\ell+1} \right)$
    and $(f(t_\ell), g(t_\ell)) \in \partial \FF$.
    So we can apply the same reasoning as described in detail for $\ell=0$ above, 
    but this time with $\delta = (\log 1/\tau)^{2\ell+1}$,
    and choosing $N = (2\ell+1) + 2$
    so that $\delta + \log 1/\tau = o((\log 1/\tau)^{N-1})$ is satisfied.
\end{proof}

Next, we prove \autoref{thm:cv:cv_finalphase}, showing uniform convergence over an arbitrarily large portion of the final phase, with constants that may additionally depend on the initialization.

\begin{proof}[Proof of \autoref{thm:cv:cv_finalphase}]
    The uniform bound over $0 \leq t \leq t_L$ was shown in \autoref{thm:cv:cv}, so it suffices to show the bound over $t_L \leq t \leq t_L + \tau^{-1}$.
    For clarity, to more easily keep track of which coefficients come from what, we will show a slightly stronger version of the statement: we show a uniform bound over all $t_L \leq t \leq t_L + \tau^{-\alpha}$ for an arbitrary fixed $\alpha \geq 1$.

    Similar to the previous proof, throughout this proof we use $O(\cdot), \Omega(\cdot), \Theta(\cdot)$ to hide constants that depend only on $\mu, \nu, C, (f^0, g^0)$, and $\alpha$,
    and we will say ``for $\tau$ small enough'' to mean that a statement holds provided $\tau \leq \tau_0'$ for some constant $\tau_0' = \Theta(1)$.
    We also re-index the Sinkhorn iterates by shifting back the index $k$ by $2$.
    Recall that the cold Sinkhorn dynamics is stationary:
    $\forall t \geq t_L, (f(t), g(t)) = (f(t_L), g(t_L))$,
    as shown in \autoref{lm:CS:finalphase_optimal}.
    
    Introduce $(\tf^k, \tg^k)_k$ the iterates of the Sinkhorn algorithm for the EOT problem with target marginals $\mu, \nu$ and cost matrix
    $\tC_{ij} = \begin{cases}
        C_{ij} ~~\text{if}~ (i,j) \in \EEE^L \\
        \infty ~~\text{otherwise}
    \end{cases}$,
    initialized at $(\tf^{\floor{2 t_L/\tau}}, \tg^{\floor{2 t_L/\tau}}) = (f^{\floor{2 t_L/\tau}}, g^{\floor{2 t_L/\tau}})$.
    By \autoref{thm:cv:cv}, we have
    $(f^{\floor{2t_L/\tau}}, g^{\floor{2t_L/\tau}})
    = (f(t_L), g(t_L)) + O\left( \tau (\log 1/\tau)^{2L+1} \right)$,
    and by definition $(f(t_L), g(t_L)) \in \partial \FF$.

    Denote for concision
    \begin{equation*}
        k_L = \floor{2 t_L/\tau},
        \qquad
        k_{L+1} = \floor{2 (t_L + \tau^{-\alpha})/\tau} = k_L + \Theta(\tau^{-\alpha-1}).
    \end{equation*}
    By the same reasoning as in the proof of \autoref{thm:cv:cv} up until \eqref{eq:CS:olM_negative} excluded, we find that
    \begin{align}
        \forall k_L + \Theta(\delta) \leq k \leq k_{L+1},~~
        \norm{\tf^k - f(\tau k/2)}_\infty
        &\leq O\big( \tau \log (k-k_L) (\delta + \log (k-k_L)) \big) \\
        &\leq O\big( \tau (\log 1/\tau) (\delta + \log 1/\tau) \big)
    \end{align}
    where $\delta = O\left( (\log 1/\tau)^{2L+1} \right)$,
    for $\tau$ small enough, and likewise for the $\tg^k$.
    Still by the same reasoning,
    \begin{equation*}
        \forall k_L \leq k \leq k_{L+1},~~
        \norm{f^k-\tf^k}_\infty
        \leq \tau\, O(1)\, e^{2\delta} 
        \cdot O(\tau^{-\alpha-1})
        \max_{k_L \leq l \leq k-1} e^{-\tau^{-1} \ol M(\tau l/2)} ~
        e^{\abs{M^l - \tau^{-1} \ol M(\tau l/2)}}
    \end{equation*}
    and likewise for the $g^k$, 
    and
    \begin{gather*}
        \forall k_L \leq k \leq k_{L+1},~~
        e^{\abs{M^k - \tau^{-1} \ol M(\tau k/2)}}
        \leq \left( 
            e^{-\Theta\left( (\log 1/\tau) (\delta+\log 1/\tau) \right)} 
            - \Theta(1)\, e^{2\delta}
            \sum_{l=k_L}^{k-1}
            e^{-\tau^{-1} \ol M(\tau l/2)} 
        \right)^{-1} \\
        \text{where}\quad
        -M^k = \max_{(i,j) \not\in \EEE^L} [f^k_i+g^k_j-C_{ij}]/\tau
        \quad\text{and}\quad
        -\ol M(t) = \max_{(i,j) \not\in \EEE^L} f_i(t)+g_j(t)-C_{ij}.
    \end{gather*}
    However, we do not have an analog of the estimate on $-\ol M(\tau k/2)$ from \eqref{eq:CS:olM_negative} anymore. 
    Instead,
    \begin{equation}
        \forall t \geq t_L,~
        -\ol M(t) = -\ol M(t_L) = -\Theta(1)
    \end{equation}
    by stationarity of the cold Sinkhorn dynamics.
    Here the right-hand side is strictly negative and independent of $t$ and $\tau$, but it may depend on the initialization $(f(0), g(0))$ of the dynamics---and so on the initialization $(f^0, g^0)$ of the Sinkhorn algorithm---a priori.

    \pagebreak

    By continuing to follow the same reasoning as in the proof of \autoref{thm:cv:cv}, albeit with constants that may depend on $(f^0, g^0)$ instead of only on $\mu, \nu, C$,
    we similarly obtain
    \begin{align}
        \forall k_L \leq k \leq k_{L+1},~~
        e^{\abs{M^k - \tau^{-1} \ol M(\tau k/2)}}
        &\leq \left( 
            e^{-\Theta\left( (\log 1/\tau) (\delta+\log 1/\tau) \right)} 
            - \Theta(1)\, e^{2\delta} \cdot O(\tau^{-\alpha-1}) e^{-\Theta(\tau^{-1})} 
        \right)^{-1} 
    \nonumber \\
        &= \left( 
            e^{-\Theta\left( (\log 1/\tau) (\delta+\log 1/\tau) \right)} 
            -\Theta(1)\, e^{2\delta + O(\log(1/\tau)) - \Theta(\tau^{-1})}
        \right)^{-1} 
    \nonumber \\
        &\leq O\left( e^{\Theta\left( (\log 1/\tau) (\delta+\log 1/\tau) \right)} \right) 
    \nonumber \\
        \text{and}\qquad\qquad\qquad~
        \norm{f^k-\tf^k}_\infty
        &\leq O(\tau^{-\alpha}) \, e^{2\delta}~ e^{-\Theta(\tau^{-1})}~ O\left( e^{\Theta\left( (\log 1/\tau) (\delta+\log 1/\tau) \right)} \right)
        \leq O(\tau) 
    \qquad
    \label{eq:cv:pf_cv_vkwk_finalphase} \\
        \text{and}\qquad\qquad
        \norm{f^k-f(\tau k/2)}_\infty
        &\leq \norm{f^k-\tf^k}_\infty + \norm{\tf^k - f(\tau k/2)} 
    \nonumber \\
        &\leq O(\tau) + O\big( \tau (\log 1/\tau) (\delta + \log 1/\tau) \big)
    \nonumber
    \end{align}
    for $\tau$ small enough, and likewise for the $g^k$,
    where we recall that $\delta = O\left( (\log 1/\tau)^{2L+1} \right)$.
\end{proof}

\subsection{An improved convergence guarantee for the Sinkhorn algorithm}

In this section, we formalize the following reasoning.
We have shown in \autoref{subsec:CS:ppties} that the cold Sinkhorn dynamics $(f(t), g(t))$ converges in finite time $t_L$, and in \autoref{subsec:cv:mainres} that the Sinkhorn iterates $(f^k, g^k)$ stay $\tO(\tau)$-close to $(f(\tau k/2), g(\tau k/2))$ until $t_L$.
Since moreover the dual objectives $\Psi_0$ and $\Psi_\tau$ differ essentially by $O(\tau)$ and are Lipschitz-continuous, this implies that at iteration $k = \floor{2 t_L/\tau} = \Theta(1/\tau)$, the Sinkhorn algorithm reaches dual suboptimality $\Psi_\tau(f^k, g^k) - \min \Psi_\tau \leq \tO(\tau)$.

This convergence guarantee is new: previous works analyzing Sinkhorn show either exponential convergence but with a poor dependency on $\tau$, or polynomial convergence bounds in $\Psi_\tau(f^k, g^k) - \min \Psi_\tau \leq O(1/(\tau k))$ \cite{dvurechensky2018computational}.
So according to previous works, dual suboptimality $\leq \tO(\tau)$ can be ensured by using $k = O(1/\tau^2)$ iterations, while our result shows that actually $k = \Theta(1/\tau)$ is sufficient.
Furthermore, our argument based on the cold Sinkhorn dynamics also implies that $k=\Theta(1/\tau)$ is sharp.

\begin{theorem} \label{thm:cv:improved_guarantee}
    Denote by $(f^k, g^k)_k$ the iterates of the Sinkhorn algorithm applied to \eqref{eq:intro:EOT}.
    Then there exist constants $T, B$ dependent only on $\mu, \nu$, and $C$ 
    such that
    \begin{equation}
        \forall k \geq T/\tau,~~
        \Psi_\tau(f^k, g^k) - \min \Psi_\tau
        \leq B\, \tau (\log 1/\tau)^{2L+1}.
    \end{equation}
    Moreover, there exist constants $T', B', B''>0$ dependent only on $\mu, \nu, C$, and $(f^0, g^0)$ 
    such that either $(f_0[g^0], g_0[f_0[g^0]]) \in \argmin_\FF \Psi_0$ or
    \begin{equation}
        \qquad
        \forall k \leq T'/\tau,~~
        \Psi_\tau(f^k, g^k) - \min \Psi_\tau
        \geq B' - B'' \,\tau (\log 1/\tau)^{2L+1}.
    \end{equation}
\end{theorem}

\begin{proof}
	Denote by $(f(t), g(t))$ the cold Sinkhorn dynamics applied to \eqref{eq:intro:OT} initialized at $(f(0), g(0)) = (f_0[g^0], g_0[f_0[g^0]])$.
	By \autoref{thm:cv:cv}, for $k_L = \floor{2t_L/\tau}+2$, we have
	\begin{equation}
		\norm{f^{k_L} - f(t_L)}_\infty
		+ \norm{g^{k_L} - g(t_L)}_\infty
		\leq B \, \tau (\log 1/\tau)^{2L+1}
	\end{equation}
	for some $B$ dependent only on $\mu, \nu$, and $C$.
	Recall that for any $k \geq 1$, $Z_\tau(f^k, g^k)=1$ and so $\Psi_\tau(f^k, g^k) = -\mu^\top f^k - \nu^\top g^k = \Psi_0(f^k, g^k)$.
	Hence, by $1$-Lipschitz-continuity of $\Psi_0$,
	\begin{align}
		\Psi_\tau(f^{k_L}, g^{k_L})
		= \Psi_0(f^{k_L}, g^{k_L})
		&\leq \Psi_0(f(t_L), g(t_L)) 
		+ \norm{f^{k_L} - f(t_L)}_\infty
		+ \norm{g^{k_L} - g(t_L)}_\infty \\
		&\leq \min \Psi_0
		+ B \, \tau (\log 1/\tau)^{2L+1}
	\end{align}
	since $\Psi_0(f(t_L), g(t_L)) = \min \Psi_0$ by \autoref{lm:CS:finalphase_optimal}.
	Now $\abs{\min \Psi_0 - \min \Psi_\tau}$ is equal to the difference between the optimal values of the primal OT/EOT problems \eqref{eq:intro:OT} and \eqref{eq:intro:EOT},
	so $\abs{\min \Psi_0 - \min \Psi_\tau} \leq \tau \max_{\pi \in \Delta_{m \times n}} \Hdiv{\pi}{\mu \otimes \nu} = -\tau \log(\mu_{\min} \nu_{\min})$.
	Thus,
	\begin{align}
		\Psi_\tau(f^{k_L}, g^{k_L}) - \min \Psi_\tau
		\leq -\tau \log(\mu_{\min} \nu_{\min})
		+ B \, \tau (\log 1/\tau)^{2L+1},
	\end{align}
	and the same holds for all $k \geq k_L$ since $(\Psi_\tau(f^k, g^k))_k$ is non-increasing along the Sinkhorn algorithm by definition.
	In particular, recall from \autoref{lm:CS:t_L=O(1)} that $t_L$ is upper-bounded by a constant $T$ dependent only on $\mu, \nu$, and $C$ (and independent of $(f(0), g(0))$);
	then the above bound holds for all $k \geq \floor{2 T/\tau}+2$.
	This proves the first part of the lemma.
	
	For the second part, i.e., the lower bound, suppose that $(f_0[g^0], g_0[f_0[g^0]]) \not\in \argmin_\FF \Psi_0$. 
	Then note that $t_L>0$ and
	$\Psi_0(f(t_L/2), g(t_L/2))-\min \Psi_0 > 0$, and that these two quantities are dependent only on $\mu, \nu, C$, and $(f^0, g^0)$. 
	So for $k_L' = \floor{t_L/\tau} + 2$,
	\begin{align*}
		\MoveEqLeft
		\Psi_\tau(f^{k_L'}, g^{k_L'}) - \min \Psi_\tau
		\geq \Psi_0(f^{k_L'}, g^{k_L'}) 
		- \min \Psi_0
		+ \tau \log(\mu_{\min} \nu_{\min}) \\
		&\geq \Psi_0(f(t_L/2), g(t_L/2)) - \min \Psi_0
		- \norm{f^{k_L'} - f(t_L/2)}_\infty
		- \norm{g^{k_L'} - g(t_L/2)}_\infty
		+ \tau \log(\mu_{\min} \nu_{\min}) \\
		&\geq \Psi_0(f(t_L/2), g(t_L/2)) - \min \Psi_0
		- B\, \tau (\log 1/\tau)^{2L+1}
		+ \tau \log(\mu_{\min} \nu_{\min}).
	\end{align*}
	Finally, the same lower bound holds for all $k \leq k_L'$ since $(\Psi_\tau(f^k, g^k))_k$ is non-increasing along the Sinkhorn algorithm.
\end{proof}

We can also deduce a corresponding convergence bound in terms of the alternative suboptimality metric
$
    E_k = \norm{X_\sharp \pi^k-\mu}_1 + \norm{Y_\sharp \pi^k-\nu}_1
$,
where $\pi^k = \pi_\tau[f^k, g^k]$ are the primal Sinkhorn iterates.
Our new result guarantees that $E_k \leq \tO(\sqrt{\tau})$ as soon as $k = \Theta(1/\tau)$, instead of $k = O(\tau^{-3/2})$ as previous works would suggest since the best applicable bound is $E_k \leq O(1/(\tau k))$ \cite{dvurechensky2018computational}.

\begin{corollary} \label{coroll:cv:improved_guarantee_Ek}
    Denote by $(\pi^k)_k$ the primal iterates of the Sinkhorn algorithm applied to \eqref{eq:intro:EOT}.
    Then there exist constants $T, B$ dependent only on $\mu, \nu$, and $C$ 
    such that
    \begin{equation}
        \forall k \geq T/\tau,~~
        \norm{X_\sharp \pi^k - \mu}_1 + \norm{Y_\sharp \pi^k - \nu}_1
        \leq B\, \sqrt{\tau}\, (\log 1/\tau)^{L+\sfrac12}.
    \end{equation}
\end{corollary}

\begin{proof}
	By \cite[proof of Lemma~2.5]{wang2026almost}, for any $k \geq 2$,
	\begin{equation}
		E_k^2 \leq \frac{8}{\tau k} \left( \Psi_\tau(f^{\ceil{k/2}}, g^{\ceil{k/2}}) - \min \Psi_\tau \right).
	\end{equation}
	So by \autoref{thm:cv:improved_guarantee}, there exist $T, B$ dependent only on $\mu, \nu$, and $C$ such that
	for all $k \geq 2 T/\tau$,
	\begin{equation}
		E_k^2 \leq \frac{8}{\tau k} \cdot B\, \tau (\log 1/\tau)^{2L+1}
		\leq \frac{8}{2T} \cdot B\, \tau (\log 1/\tau)^{2L+1}
		= B'\, \tau (\log 1/\tau)^{2L+1}
	\end{equation}
	with $B'$ dependent only on $\mu, \nu$, and $C$, as announced.
\end{proof}

\subsection{The case with exactly scalable sub-problems} \label{subsec:cv:scalcase}

The proof of the convergence of Sinkhorn to cold Sinkhorn presented in \autoref{sec:cv} is fully general, but it is relatively intricate.
In this section, we show an alternative proof under the following additional assumption.
In words, in view of \autoref{lm:prelim:equiv_exact_scalability}, our assumption is that the sub-problems encountered by the dynamics are exactly scalable.

\begin{assumption} \label{assump:cv:scalcase}
    Consider a run of the cold Sinkhorn dynamics as defined in \autoref{def:CS:formaldef}. 
    For each $\ell \leq L$, consider the sets $\SSS^\ell, \ol\SSS^\ell$ associated to $\mu, \nu$, and $\EEE^\ell$ as defined by \autoref{thm:prelim:baradat_ventre_sets}.
    We assume that
    \begin{equation}
        \forall 0 \leq \ell \leq L,~
        \SSS^\ell = \ol\SSS^\ell.
    \end{equation}
\end{assumption}

In our numerical experiments, this assumption appeared to hold generically, so the analysis presented in this section may be closer to typical practical behavior.



Under this assumption, the convergence of $(v^k, w^k)$ to $(\ol v^\ell, \ol w^\ell)$ at each ``phase transition'' is exponentially fast, thanks to \autoref{prop:prelim:qtve_BV24_exp}.
This leads to the following variant of \autoref{thm:cv:cv} without log factors in the upper bound.
Similar to \autoref{thm:cv:cv_finalphase}, one could also show that uniform convergence holds over an arbitrarily large portion of the final phase without log factors, but we do not develop this here.

\begin{theorem} \label{thm:cv:scalcase}
    Let $(f^k, g^k)_k$ denote the iterates of the Sinkhorn algorithm applied to \eqref{eq:intro:EOT} initialized at some $(f^0, g^0)$.
    Let $(f(t), g(t))$ denote the cold Sinkhorn dynamics applied to \eqref{eq:intro:OT} initialized at $(f(0), g(0)) = (f_0[g^0], g_0[f_0[g^0]]) \in \partial \FF$.
    Suppose that \autoref{assump:cv:scalcase} holds.
    Then for any $\tau \leq \tau_0$,
    \begin{equation}
        \sup_{0 \leq t \leq t_L}
        \norm{f^{\floor{2t/\tau}+2} - f(t)}_\infty
        + \norm{g^{\floor{2t/\tau}+2} - g(t)}_\infty
        \leq B \, \tau,
    \end{equation}
    for some constants $\tau_0, B$ dependent only on $\mu, \nu$, and $C$.
\end{theorem}

\begin{proof}
    We use the same conventions for $O(\cdot), \Omega(\cdot), \Theta(\cdot)$ as in the proof of \autoref{thm:cv:cv}, that is, they hide constants that depend only on $\mu, \nu$, and $C$.
    We also perform the same re-indexing of the Sinkhorn iterates by shifting back the index $k$ by $2$,
    so that $(f^0, g^0) = (f(0), g(0)) + O(\tau)$ by \autoref{lm:cv:logsumexp_rate}.
    It suffices to show by induction that for any $0 \leq \ell \leq L-1$,
    for $\tau$ small enough,
    \begin{equation} \label{eq:cv:scalcase_induction_claim}
        \sup_{\floor{2t_\ell/\tau} \leq k \leq \floor{2t_{\ell+1}/\tau}}
        \abs{f^k - f(\tau k/2)}
        + \abs{g^k - g(\tau k/2)}
        \leq O(\tau).
    \end{equation}
    
    \vspace{0.5em}
    \underline{For $\ell=0$:}~
    As in the proof of \autoref{thm:cv:cv},
    denote by $(\tf^k, \tg^k)_k$ the iterates of the Sinkhorn algorithm applied to the EOT problem with target marginals $\mu, \nu$ and cost matrix
    $\tC_{ij} = \begin{cases}
        C_{ij} ~~\text{if}~ (i,j) \in \EEE^0 \\
        \infty ~~\text{otherwise}
    \end{cases}$,
    initialized at $(\tf^0, \tg^0) = (f^0, g^0)$.
    Denote by $\tv^k, \tw^k$ and $\tU^k$ the associated variables as in \eqref{eq:prelim:def_vk_wk}, \eqref{eq:prelim:def_Uk}.
    Since $(\tf^0, \tg^0) = (f(0), g(0)) + O(\tau)$ by \autoref{lm:cv:logsumexp_rate} and $(f(0), g(0)) \in \partial \FF$, then
    \begin{equation}
        \forall (i,j) \in \EEE^0,~
        \tau \tU^0_{ij} = \tf^0_j + \tg^0_j - C_{ij} = O(\tau).
    \end{equation}
    So by \autoref{lm:prelim:unifbound_vk_wk} and \autoref{prop:prelim:qtve_BV24_exp} with $\delta = \Theta(1)$, 
    \begin{align*}
        \forall k \geq K_0 = \Theta(1),~~
        & \norm{(\tv^k, \tw^k) - (\olv^0, \ol w^0)}_\infty
        = \norm{(\tv^k, \tw^k) - (v^*(\EEE^0), w^*(\EEE^0))}_\infty
        \leq O(1) \left(1  - e^{-\Theta(1)} \right)^k \\
        \text{and}~~~~
        \forall k \geq 1,~~
        & \norm{\tv^k}_\infty, \norm{\tw^k}_\infty \leq O(1).
    \end{align*}
    As a consequence,
    since $\tf^k = \tf^{k-2} + \tau \tv^k$
    and $f(\tau k/2) = f(\tau (k-2)/2) + \tau \ol v^0$
    for all $k$,
    \begin{itemize}
        \item For all $k \leq K_0 = \Theta(1)$,
        $
            \norm{\tf^k - f(\tau k/2)}_\infty
            \leq O(\tau k) = O(\tau)
        $,
        \item For all $K_0 \leq k \leq \floor{2 t_1/\tau}$,
        \begin{equation}
            \norm{\tf^k - f(\tau k/2)}_\infty
            \leq O(\tau) 
            + \sum_{\substack{l=K_0 \\ l ~\text{even}}}^k
            \tau\, O(1) \left(1  - e^{-\Theta(1)} \right)^l
            \leq O(\tau) \sum_{l=0}^\infty \left(1  - e^{-\Theta(1)} \right)^l
            = O(\tau),
        \end{equation}
    \end{itemize}
    and likewise for the $\tg^k$.
    
    Next, still following the proof of \autoref{thm:cv:cv}, let us show that $(f^k, g^k)$ remains close to $(\tf^k, \tg^k)$ essentially throughout the phase.
    By \autoref{lm:prelim:control_finite_infinite_costs} with $\EEE = \EEE^0$ and $\delta = O(1)$, we have
    \begin{equation}
        \forall k \geq 0,~
        \norm{f^k-\tf^k}_\infty\!,~
        \norm{g^k-\tg^k}_\infty
        \leq \tau\, O(1)
        \sum_{l=0}^{k-1}
        e^{-M^l}
        ~\quad \text{where} \quad~
        -M^k = \max_{(i,j) \not\in \EEE^0} [f^k_i+g^k_j-C_{ij}]/\tau.
    \end{equation}
    Denote likewise
    \begin{equation}
        -\tM^k = \max_{(i,j) \not\in \EEE^0} [\tf^k_i+\tg^k_j-C_{ij}]/\tau
        \qquad\text{and}\qquad
        -\ol M(t) = \max_{(i,j) \not\in \EEE^0} f_i(t)+g_j(t)-C_{ij}.
    \end{equation}
    Fix $K_0' = \floor{2 t_1/\tau} - \Delta$ for some large $\Delta>0$ to be chosen later.
    Then for any $k \leq K_0'$,
    \begin{align*}
        \abs{M^k - \tau^{-1} \ol M(\tau k/2)}
        &\leq \abs{M^k - \tM^k} + \abs{\tM^k - \tau^{-1} \ol M(\tau k/2)} \\
        &\leq \tau^{-1} \norm{(f^k, g^k)-(\tf^k, \tg^k)}_\infty
        + \tau^{-1} \norm{(\tf^k, \tg^k) - (f(\tau k/2), g(\tau k/2))}_\infty \\
        &\leq O(1)\, \sum_{l=0}^{k-1} e^{-\tau^{-1} \ol M(\tau l/2)}~
        e^{\abs{M^l - \tau^{-1} \ol M(\tau l/2)}}
        + O(1).
    \end{align*}
    So by the discrete Bihari-LaSalle inequality \cite[Theorem~2.3.1]{pachpatte2001inequalities},
    \vspace{-0.5em}
    \begin{equation*}
        e^{\abs{M^k - \tau^{-1} \ol M(\tau k/2)}}
        \leq \left( 
            e^{-\Theta(1)} 
            - \Theta(1)\, \sum_{l=0}^{k-1} e^{-\tau^{-1} \ol M(\tau l/2)} 
        \right)^{-1}.
    \end{equation*}
    Moreover, by definition of $K_0' = \floor{2 t_1/\tau} - \Delta$ and of $t_1$, for all $k \leq K_0'$,
    \begin{equation}
        -\ol M(\tau k/2) \leq 
        ~\underbrace{-\ol M(t_1)}_{0}
        - \xi \left( t_1-\tau k/2 \right)
    \end{equation}
    where
    $\xi = \min \left\{ \ol v^0_i + \ol w^0_j;~ (i,j) \not\in \EEE^0 ~\text{and}~ \ol v^0_i + \ol w^0_j > 0 \right\} = \Theta(1)$.
    Hence, for all $k \leq K_0'$,
    \begin{equation}
        \sum_{l=0}^{k-1} e^{-\tau^{-1} \ol M(\tau l/2)}
        \leq \sum_{l=0}^{K_0'} e^{-\tau^{-1} \xi (t_1 - \tau l/2)}
        = \sum_{l=0}^{\floor{2t_1/\tau} - \Delta} e^{-(\xi/2) (2t_1/\tau - l)}
        \leq \sum_{l=\Delta}^\infty e^{-(\xi/2) l}
        \leq O(e^{-\xi \Delta/2})
    \end{equation}
    and so,
    \begin{equation}
        e^{\abs{M^k - \tau^{-1} \ol M(\tau k/2)}}
        \leq \left( 
            e^{-\Theta(1)} 
            - O(e^{-\xi \Delta/2})
        \right)^{-1}
        \leq \left( \frac12\, e^{-\Theta(1)} \right)^{-1} = O(1)
    \end{equation}
    for $\Delta = \Theta(1)$ large enough.
    Thus, for all $k \leq K_0'$,\,
    $e^{-M^k} \leq e^{-\tau^{-1} \ol M(\tau k/2)} \cdot O(1)$ and
    \begin{align}
        \norm{f^k-\tf^k}_\infty,~
        \norm{g^k-\tg^k}_\infty
        &\leq \tau\, O(1)\,
        \sum_{l=0}^{k-1}
        e^{-M^l}
        \leq \tau\, O(1)\,
        \sum_{l=0}^{k-1}
        e^{-\tau^{-1} \ol M(\tau l/2)}
        \leq \tau\, O(1).
    \end{align}
    
    It only remains to treat the iterates $K_0' = \floor{2 t_1/\tau} - \Delta \leq k \leq \floor{2 t_1/\tau}$. In this case,
    \begin{equation}
        \norm{f^k - f(\tau k/2)}_\infty
        \leq \underbrace{
            \norm{f^{K_0'} - f(\tau K_0'/2)}_\infty
        }_{\leq O(\tau)}
        +~ \underbrace{
            \norm{f^k - f^{K_0'}}_\infty
            + \norm{f(\tau k/2) - f(\tau K_0'/2)}_\infty
        }_{\leq O(\tau \Delta) = O(\tau)}
    \end{equation}
    since
    $\norm{v^l}_\infty = O(1)$ for all $l \geq 3$
    by \autoref{lm:prelim:unifbound_vk_wk},
    and likewise for the $g^k$.
    
    \vspace{0.5em}
    \underline{For $1 \leq \ell \leq L-1$:}~
    Suppose the induction hypothesis \eqref{eq:cv:scalcase_induction_claim} holds at rank $\ell-1$.
    In particular, 
    $(f^{\floor{2t_\ell/\tau}}, g^{\floor{2t_\ell/\tau}})
    = (f(t_\ell), g(t_\ell)) + O(\tau)$
    and $(f(t_\ell), g(t_\ell)) \in \partial \FF$.
    So we can apply the exact same reasoning as described in detail for $\ell=0$ above.
\end{proof}

\begin{remark}
    Neither the proof of \autoref{thm:cv:cv} nor that of \autoref{thm:cv:scalcase} really follow the intuitive picture given in \autoref{subsec:CS:informal}: there is no notion of ``phase transition period'' appearing in the proofs.
    In the case with exactly scalable sub-problems, it would indeed be possible to formalize the intuitive picture, with the caveat that the phase transitions last for $K' = \Theta(\log 1/\tau)$ iterations (instead of $\Theta(1)$ as stated in the Ansatz from \autoref{subsec:CS:informal}).
    But this approach would yield a bound with logarithmic factors, which \autoref{thm:cv:scalcase} shows are avoidable in this case.
\end{remark}


\section{Saddle-to-saddle behavior of the primal variable} \label{sec:pik}

The previous sections provided a complete description of the small-$\tau$ behavior of the Sinkhorn algorithm in terms of the dual variables $(f^k, g^k)$.
In this section, we show that an explicit description is also available for the primal variables $\pi^k$, which we recall are defined by
$\pi^k_{ij} = e^{[-C_{ij} + f^k_i + g^k_j]/\tau} \mu_i \nu_j$ for all $k \geq 1$.

To state our result more easily, let us first introduce some notation.

\pagebreak

\begin{lemma} \label{lm:pik:def_Q*P*}
    For any $\mu \in \Delta_m, \nu \in \Delta_n, C \in \RR^{m \times n}$, 
    for any $\EEE \subset \{1 \dots m \} \times \{1 \dots n\}$ such that the bipartite graph with edge set $\EEE$ has no isolated vertex, 
    for any $\tau>0$,
    denote
    \begin{align}
        Q^*_\tau(\mu, \nu, \EEE, C)
        &= \argmin_{Q \in \Delta_\EEE} \sum_{(i,j) \in \EEE} C_{ij} Q_{ij} + \tau \Hdiv{Q}{\mu \otimes \nu}
        ~~~~\text{subject to}~~~~
        X_\sharp Q = \mu^*(\EEE),~~
        Y_\sharp Q = \nu, \\
        P^*_\tau(\mu, \nu, \EEE, C)
        &= \argmin_{P \in \Delta_\EEE} \sum_{(i,j) \in \EEE} C_{ij} P_{ij} + \tau \Hdiv{P}{\mu \otimes \nu}
        ~~~~\text{subject to}~~~~
        X_\sharp P = \mu,~~
        Y_\sharp P = \nu^*(\EEE).
    \end{align}
    Then,
    \begin{itemize}
        \item For $\pi^k$ the primal iterates of the Sinkhorn algorithm applied to the EOT problem with target marginals $\mu, \nu$ and cost matrix 
        $\tC_{ij} = \begin{cases}
            C_{ij} ~\text{if}~ (i,j) \in \EEE \\
            \infty ~~\text{otherwise}
        \end{cases}$,
        we have
        $\pi^{2k} \to Q^*_\tau(\mu, \nu, \EEE, C)$
        and
        $\pi^{2k+1} \to P^*_\tau(\mu, \nu, \EEE, C)$ as $k \to \infty$.
        \item The following limits exist:
        \begin{equation}
            Q^*_0(\mu, \nu, \EEE, C)
            \coloneqq \lim_{\tau \to 0}
            Q^*_\tau(\mu, \nu, \EEE, C),
            \qquad
            P^*_0(\mu, \nu, \EEE, C)
            \coloneqq \lim_{\tau \to 0}
            P^*_\tau(\mu, \nu, \EEE, C).
        \end{equation}
    \end{itemize}
\end{lemma}

\begin{proof}
    The first item is precisely the content of \cite[Theorem~3.2]{baradat2024convergence}.
    The second item is a consequence of \cite[Proposition~4.1]{cominetti1994asymptotic}.
\end{proof}

Our result is as follows. Contrary to \autoref{sec:cv}, here we only show a qualitative pointwise convergence without a rate, for simplicity.

\begin{theorem} \label{thm:pik:even_odd}
    Let $(\pi^k)_k$ denote the primal iterates of the Sinkhorn algorithm applied to \eqref{eq:intro:EOT} initialized at some $(f^0, g^0)$.
    Consider the phase transition times $t_0 = 0 < t_1 < ... < t_L < t_{L+1} = \infty$ and the sets $(\EEE^\ell)_{0 \leq \ell \leq L}$ appearing in the cold Sinkhorn dynamics initialized at $(f_0[g^0], g_0[f_0[g^0]])$, as defined in \autoref{def:CS:formaldef}.
    Then for any $0 \leq \ell \leq L$,
    \begin{equation}
        \forall t_\ell < t < t_{\ell+1},~~
        \left( \pi^{2 \floor{t/\tau}+2}, \pi^{2 \floor{t/\tau}+3} \right)
        \to \left( \ol\pi^\ell_{\mathrm{even}}, \ol\pi^\ell_{\mathrm{odd}} \right)
    \end{equation}
    where for each $\ell$,
    \begin{equation}
        \ol\pi^\ell_{\mathrm{even}}
        = Q^*_0(\mu, \nu, \EEE^\ell, C),
        \qquad
        \ol\pi^\ell_{\mathrm{odd}}
        = P^*_0(\mu, \nu, \EEE^\ell, C).
    \end{equation}
    %
\end{theorem}

\begin{remark} \label{rk:pik:noclosedform}
    This description of the cold Sinkhorn dynamics is not self-contained, as it does not specify how the sets $\EEE^\ell$ are defined.
    One could hope to formulate a recursive relation that jointly determines the sequence $\left( \ol\pi^\ell_{\mathrm{even}}, \ol\pi^\ell_{\mathrm{odd}}, \EEE^\ell \right)_\ell$,
    but we expect that such a formulation is actually impossible.
    Indeed, at any phase $\ell \leq L-1$, the next set $\EEE^{\ell+1}$ depends on the value of $f_i(t_\ell) + g_j(t_\ell) - C_{ij}$ for some index $(i,j) \not\in \EEE^\ell$, 
    while the knowledge of $\ol\pi^\ell_{\mathrm{even}}, \ol\pi^\ell_{\mathrm{odd}}$, and $\EEE^\ell$ does not contain any information on the $C_{ij}$ for $(i,j) \not\in \EEE^\ell$
    since $\support(\ol\pi^\ell_{\mathrm{even}}) = \support(\ol\pi^\ell_{\mathrm{odd}}) \subset \EEE^\ell$ by definition.
    Instead, it could be interesting to try and characterize $\EEE^\ell$ via some auxiliary parametric minimization problem, similar to what was achieved in \cite{berthier2023incremental} in a related context.

    A fortiori, interpreting the cold Sinkhorn dynamics as a mirror flow by adapting the framework of \cite{leger2021gradient} is impossible,
    contrary to the parabolic Monge-Amp\`ere equation \cite{deb2023wasserstein}.
    Indeed, this framework prescribes to track the variable $\mu(t) = \lim_{\tau \to 0} X_\sharp \pi^{2\floor{t/\tau}}$ along the limiting dynamics.
    But the theorem above shows that
    $\mu(t)$ is piecewise constant with
    $\forall t_\ell < t < t_{\ell+1}, \mu(t) = \mu^*(\EEE^\ell)$, so its evolution over $t \in \RR_+$ cannot be described by an autonomous ODE.
    
    In a sense, the behavior of the Sinkhorn algorithm with small $\tau$ is really driven by the evolution of the logits $[-C_{ij} + f^k_i + g^k_j]/\tau$ at all indices $(i,j)$, even the ones for which $\pi^k_{ij} = e^{-\Theta(1/\tau)} \ll 1$---which a perspective centered only on the primal variables $\pi^k$ is likely to miss.
\end{remark}

\vspace{0.05em}
\begin{proof}
    Denote by $(f^k, g^k)$ the dual Sinkhorn iterates, so that $\pi^k_{ij} = e^{[-C_{ij}+f_i+g_j]/\tau} \mu_i \nu_j$ for all $k \geq 1$.
    For convenience, let us re-index all the iterates by shifting back the index $k$ by $2$.
    Throughout this proof, we use $O(\cdot), \Omega(\cdot), \Theta(\cdot)$ to hide constants dependent on $\mu, \nu, C$, and $(f^0, g^0)$, and we use $\tO(\cdot)$ to additionally hide poly-logarithmic factors in $1/\tau$.

    \pagebreak

    Fix $0 \leq \ell \leq L$, let $k_\ell = \floor{2t_\ell/\tau}$, and denote by $(\tf^k, \tg^k)_{k \geq k_\ell}$ the iterates of the Sinkhorn algorithm applied to the EOT problem with target marginals $\mu, \nu$ and cost matrix $\tC_{ij} = \begin{cases}
        C_{ij} ~\text{if}~ (i,j) \in \EEE^\ell \\
        \infty ~~\text{otherwise}
    \end{cases}$,
    initialized at $(\tf^{k_\ell}, \tg^{k_\ell}) = (f^{k_\ell}, g^{k_\ell})$.
    Further denote $\tpi^k = \Big( e^{[-\tC_{ij} + \tf^k_i + \tg^k_j]/\tau} \mu_i \nu_j \Big)_{ij} \in \Delta_{\EEE^\ell}$ for all $k \geq k_\ell+1$.

    Note that for all $k$,
    \begin{equation}
        \Hdiv{\tpi^k}{\pi^k}
        = \sum_{(i,j) \in \EEE^\ell} \tpi^k_{ij} \log \frac{\tpi^k_{ij}}{\pi^k_{ij}}
        = \sum_{(i,j) \in \EEE^\ell} \tpi^k_{ij} \,
        \frac{\tf^k_i - f^k_i + \tg^k_j - g^k_j}{\tau}
        \leq \frac1\tau \left( \norm{\tf^k - f^k}_\infty + \norm{\tg^k - g^k}_\infty \right).
    \end{equation}
    Now as we showed at step \eqref{eq:cv:pf_cv_vkwk} of the proof of \autoref{thm:cv:cv}, we have for all $0 \leq \ell \leq L-1$
    \begin{align}
        \forall t_\ell + \tO(\tau) \leq t \leq t_{\ell+1} - \tO(\tau),~~
        \norm{f^{\floor{2t/\tau}} - \tf^{\floor{2t/\tau}}}_\infty
        &\leq e^{-\Theta((\log 1/\tau)^{2L+1})}
        \leq O(\tau^2), \\
        \text{and so}\qquad
        \Hdiv{\tpi^{\floor{2t/\tau}}}{\pi^{\floor{2t/\tau}}}
        &\leq O(\tau) = o_\tau(1).
    \end{align}
    Likewise, for $\ell = L$, as we showed at step~\eqref{eq:cv:pf_cv_vkwk_finalphase} of the proof of \autoref{thm:cv:cv_finalphase},
    \begin{align}
        \forall t_L + \tO(\tau) \leq t \leq t_L + \Theta(1/\tau),~~
        \norm{f^{\floor{2t/\tau}} - \tf^{\floor{2t/\tau}}}_\infty
        &\leq e^{-\Theta(\tau^{-1})}
        \leq O(\tau^2), \\
        \text{and so}\qquad
        \Hdiv{\tpi^{\floor{2t/\tau}}}{\pi^{\floor{2t/\tau}}}
        &\leq O(\tau) = o_\tau(1).
    \end{align}

    Next, let $\tpi^\infty_{\mathrm{even}} = Q^*_\tau(\mu, \nu, \EEE^\ell, C)$ and
    $\tpi^\infty_{\mathrm{odd}} = P^*_\tau(\mu, \nu, \EEE^\ell, C)$. Then by the first item of \autoref{lm:pik:def_Q*P*},
    $\tpi^{2k} \to \tpi^\infty_{\mathrm{even}}$ and $\tpi^{2k+1} \to \tpi^\infty_{\mathrm{odd}}$ as $k \to \infty$.
    Moreover, one can show by adapting the proof of \autoref{prop:prelim:qtve_BV24} that the convergence occurs at a time-scale governed by
    $\delta = \max_{(i,j) \in \EEE^\ell} \abs{C_{ij} - \tf^{k_\ell}_i - \tg^{k_\ell}_j}$, and more precisely,
    \begin{equation}
        \forall k \geq k_\ell/2 + \Theta(1 \vee \delta), \quad
        \Hdiv{\tpi^\infty_{\mathrm{even}}}{\tpi^{2k}},~
        \Hdiv{\tpi^\infty_{\mathrm{odd}}}{\tpi^{2k+1}} \leq O\left( \frac{(\delta + \log k)^2}{k} \right).
    \end{equation}
    Now as remarked in the proof of \autoref{thm:cv:cv}, $\delta = O((\log 1/\tau)^N)$ where $N = 2L+1$, so 
    \begin{equation}
        \forall k \geq k_\ell/2 + \Theta((\log 1/\tau)^{2N+1}), \quad
        \Hdiv{\tpi^\infty_{\mathrm{even}}}{\tpi^{2k}},~
        \Hdiv{\tpi^\infty_{\mathrm{odd}}}{\tpi^{2k+1}} \leq O\left(\frac{1}{\log (1/\tau)} \right) = o_\tau(1).
    \end{equation}
    
    The theorem statement now follows by triangle inequality and Pinsker's inequality. Indeed for all 
    $t \in \bigcup_{0 \leq \ell \leq L-1} [t_\ell + \tO(\tau), t_{\ell+1}-\tO(\tau)] \cup [t_L + \tO(\tau), t_L+\Theta(1/\tau)]$, we get
    \begin{align*}
        \MoveEqLeft
        \norm{\pi^{2\floor{t/\tau}} - \ol\pi^\ell_{\mathrm{even}}}_1
        \leq \norm{\pi^{2\floor{t/\tau}} - \tpi^{2\floor{t/\tau}}}_1
        + \norm{\tpi^{2\floor{t/\tau}} - \tpi^{\infty}_{\mathrm{even}}}_1
        + \norm{\tpi^{\infty}_{\mathrm{even}} - \ol\pi^\ell_{\mathrm{even}}}_1 \\
        &\leq \sqrt{2 \Hdiv{\pi^{2\floor{t/\tau}}}{\tpi^{2\floor{t/\tau}}}}
        + \sqrt{2 \Hdiv{\tpi^{\infty}_{\mathrm{even}}}{\tpi^{2\floor{t/\tau}}}}
        + \norm{Q^*_\tau(\mu, \nu, \EEE^\ell, C) - Q^*_0(\mu, \nu, \EEE^\ell, C)}_1 \\
        &= o_\tau(1),
    \end{align*}
    and likewise for the odd iterations.
\end{proof}

\section{Conclusion} \label{sec:ccl}

In this paper, we analyzed the behavior of the Sinkhorn algorithm for discrete EOT computation in the regime of low regularization parameter $\tau$.
We showed that in this regime, Sinkhorn effectively implements a certain simplex-type algorithm for unregularized OT computation, termed the \emph{cold Sinkhorn} dynamics.
\pagebreak
More precisely, the sequence of dual Sinkhorn iterates $(f^k, g^k)_k$ converges to a continuous curve $(f(t), g(t))$ which moves piecewise-linearly along the boundary of the dual OT problem's feasibility polytope, and converges to an optimal dual solution after a finite number $L$ of phases.
Leveraging this fact, we also deduced a novel convergence guarantee for the Sinkhorn algorithm itself.

\vspace{4pt}
\textbf{From a technical perspective,} our work leaves open a number of natural questions.
Firstly, we were rather loose with the constants appearing in our analysis, so it is unclear how small $\tau$ must be for our bounds to be meaningful.
In particular, numerical experiments suggest that even for large $m, n$, the cold Sinkhorn dynamics regime can kick in at relatively large values of $\tau$, and it would be interesting to determine whether $\tau = o(\frac{1}{\log(mn)})$---the regime advocated by \cite{altschuler2017near}---suffices.

Secondly, the dual iterates $(f^k, g^k)$ appear to trace out a smooth curve even when $\tau$ is small but not infinitesimal (cf \autoref{fig:intro:illustrative}), which would be interesting to characterize.
In other words, the question is to describe more finely the behavior of the Sinkhorn iterates at the phase transitions.

Thirdly, we observe numerically that the phase transition times tend to be more and more spread out towards the end of the cold Sinkhorn dynamics (cf \autoref{fig:apx_addtl_exps:1}, \autoref{fig:apx_addtl_exps:2}). It could be interesting to explain and to try to leverage this phenomenon algorithmically.

\vspace{4pt}
\textbf{From a broader perspective,}
our work uncovers intriguing directions for future research in several areas.
From the point of view of optimal transport, our work deepens our understanding of the relation between discrete EOT and OT in a previously unexplored direction.
Indeed the convergence of EOT to OT at the level of variational problems, including convergence of the optimal solutions, has been studied extensively \cite{cominetti1994asymptotic,weed2018explicit}, and
the convergence of the gradient flow on the semi-dual of EOT to that of OT is also well understood \cite{genevay2016stochastic,cuturi2018semidual}.
Our work reveals that a similar connection exists between the Sinkhorn algorithm---the standard method for EOT computation---and a newly discovered simplex-type algorithm for OT computation, the cold Sinkhorn dynamics.

From the point of view of matrix scaling, our finding
is perhaps especially surprising: it shows that in the limit $\tau \to 0$, the canonical method for scaling a matrix $A = (e^{-C_{ij}/\tau})_{ij}$ to prescribed row- and column-sums reduces to an algorithm for linear programming.
It would be interesting to generalize our analysis to other related settings, such as the multi-marginal setting, the unbalanced setting, or matrix balancing \cite{idel2016review,cai2026near}.


\section*{Acknowledgments}
I would like to thank Christopher Criscitiello for insightful discussions around matrix scaling and for suggesting what became \autoref{thm:pik:even_odd}, 
as well as Jonathan Niles-Weed for discussions that inspired \autoref{rk:cv:centroid}.

\printbibliography
\addcontentsline{toc}{section}{\refname} 

\ifextended%
\newpage
\appendix
\phantomsection
\addcontentsline{toc}{section}{APPENDIX}


\section{Additional illustrative experiments} \label{sec:apx_addtl_exps}

In this appendix, we present two additional numerical experiments to qualitatively illustrate the phenomena studied in this paper, as a complement to \autoref{fig:intro:illustrative}.

The setup is identical for both experiments:
we drew $\mu, \nu$ from the uniform distribution on the simplices $\Delta_m, \Delta_n$ respectively,
we drew the entries of $C$ i.i.d.\ from the standard normal distribution,
and we set $\tau = 0.001$.
The only difference is in the choice of $m$ and $n$:
we took $m=n=50$ for \autoref{fig:apx_addtl_exps:1} and $m=n=400$ for \autoref{fig:apx_addtl_exps:2}.

We display the evolution of the dual iterates $(f^k, g^k)_{k \geq 2}$ of the Sinkhorn algorithm initialized at $(f^0, g^0) = (0,0)$,
as well as the suboptimality measured by relative entropy of the marginals: $\Hdiv{\mu}{X_\sharp \pi^k}, \Hdiv{\nu}{Y_\sharp \pi^k}$.
For readability, for each experiment, we display separately the evolution at a short horizon (until $k = \floor{2/\tau}$, top subfigures)
and at a long horizon (until approximate convergence, bottom subfigures).

\begin{figure}[h]
    \centering
    \begin{subfigure}{0.95\textwidth}
        \centering
        \includegraphics[width=\linewidth]{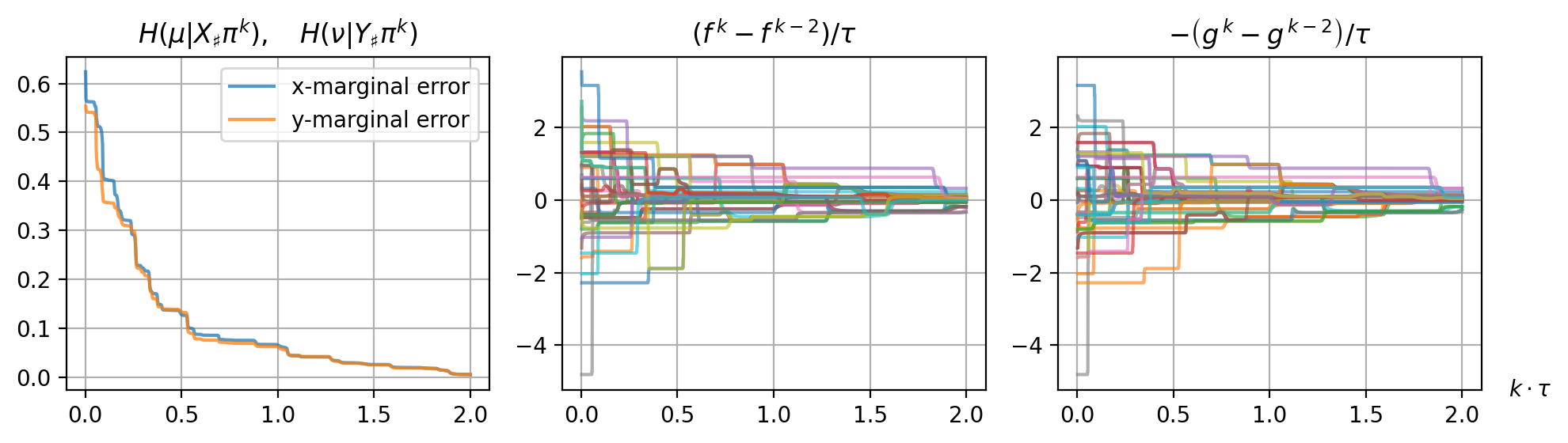} 
    \end{subfigure}

    \medskip
    \begin{subfigure}{0.95\textwidth}
        \centering
        \includegraphics[trim={0 0 0 0.8cm}, clip, width=\linewidth]{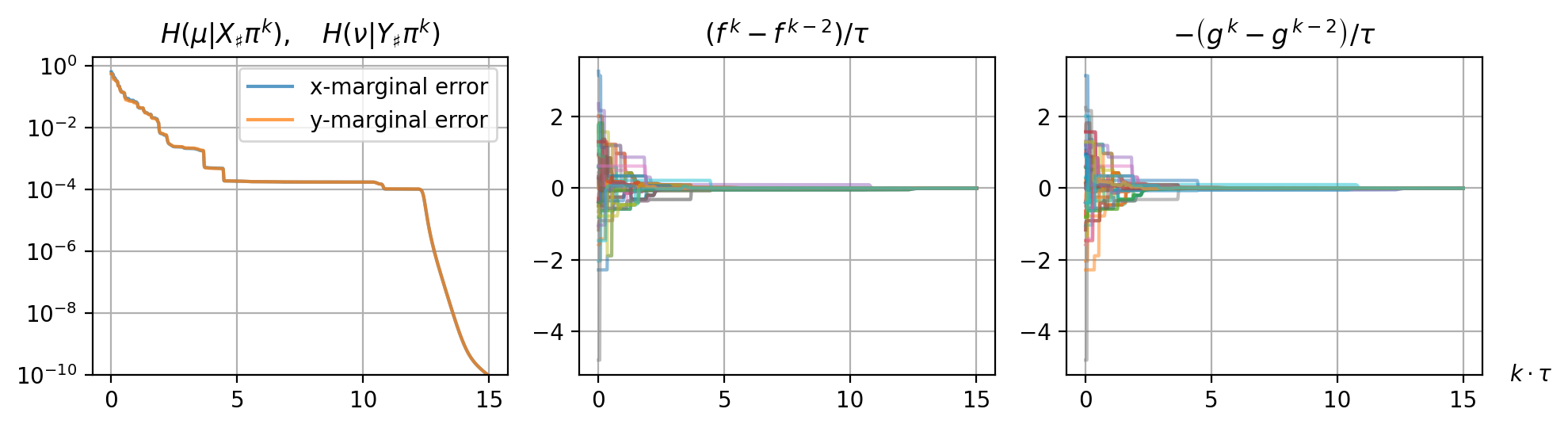}
    \end{subfigure}
    \vspace{-0.5em}
    \caption{One run of the Sinkhorn algorithm with $m=n=50$
    and $\tau = 0.001$
    }
    \label{fig:apx_addtl_exps:1}

    \vspace{1.2em}

    \begin{subfigure}{0.95\textwidth}
        \centering
        \includegraphics[width=\linewidth]{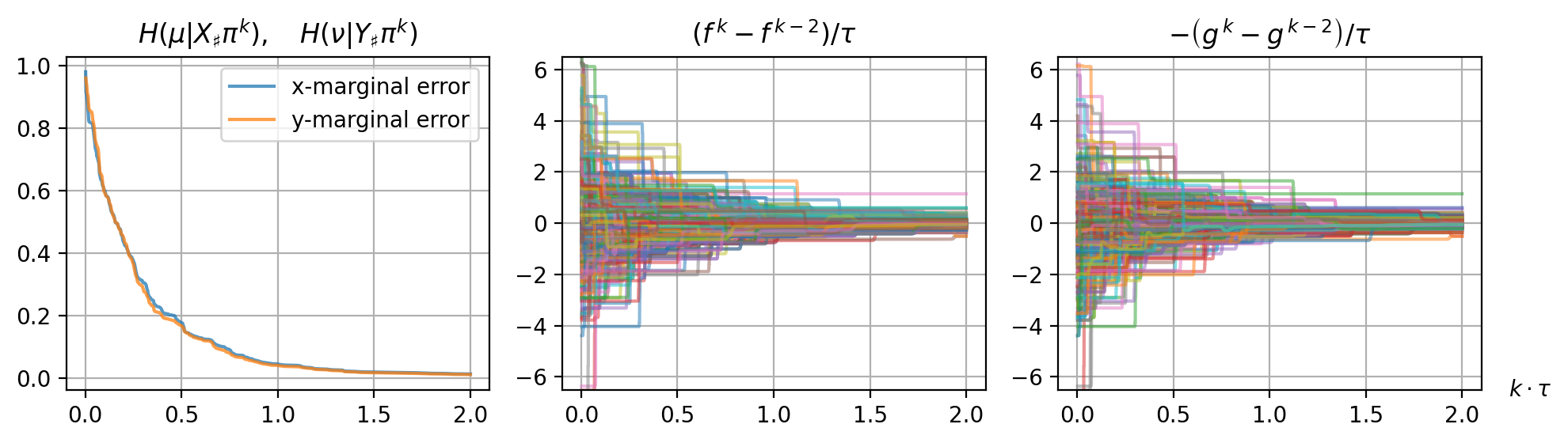} 
    \end{subfigure}

    \medskip
    \begin{subfigure}{0.95\textwidth}
        \centering
        \includegraphics[trim={0 0 0 0.8cm}, clip, width=\linewidth]{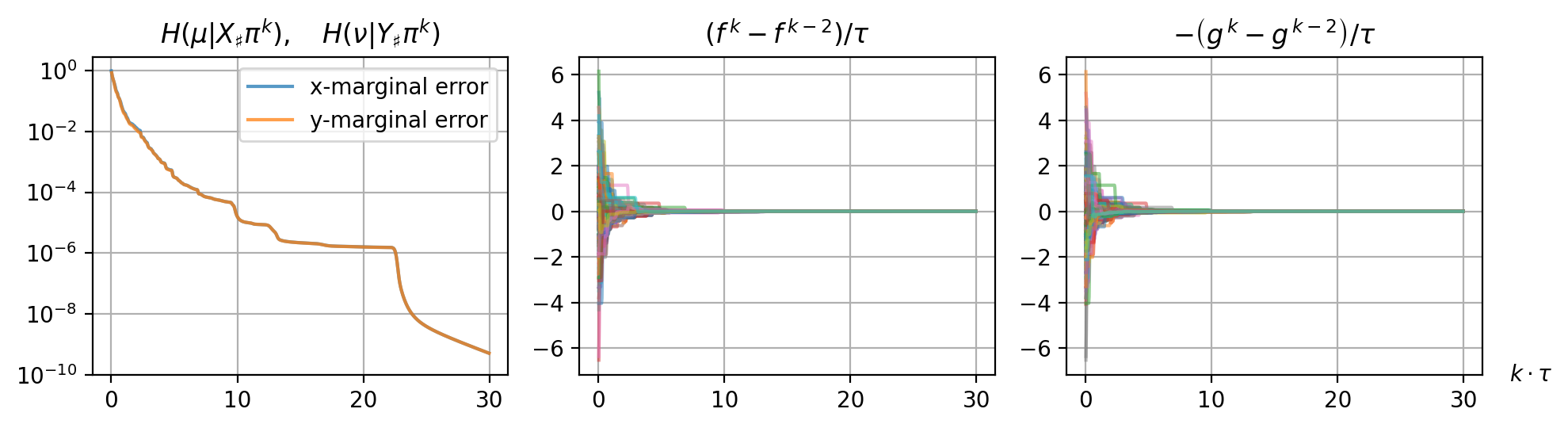} 
    \end{subfigure}
    \vspace{-0.5em}
    \caption{One run of the Sinkhorn algorithm with $m=n=400$
    and $\tau = 0.001$
    }
    \label{fig:apx_addtl_exps:2}
    \vspace{-0.7em}
\end{figure}

\section{Proofs for \autoref{subsec:prelim:qtve_BV24}}
\label{sec:apx_pf_qtve_BV24}

In this appendix, we present the proofs of \autoref{prop:prelim:qtve_BV24} and \autoref{prop:prelim:qtve_BV24_exp}, restated below.

\begin{proposition*}[\autoref{prop:prelim:qtve_BV24}, restated]
    Let $\mu, \nu, C, \EEE$ be as in \autoref{thm:prelim:baradat_ventre}.
    Let $\delta \geq 0$ and consider any initialization 
    $U^0 = \left( (f^0_i + g^0_j - C_{ij})/\tau \right)_{ij} \in (\RR \cup \{-\infty\})^{m \times n}$
    of the Sinkhorn algorithm such that
    \begin{equation}
        \forall (i,j) \in \EEE,~
        -\delta \leq U^0_{ij} \leq \delta
        \qquad \text{and} \qquad
        \forall (i,j) \not\in \EEE,~
        U^0_{ij} = -\infty.
    \end{equation}
    Then the rescaled one-iteration increments $v^k, w^k$ satisfy
    \begin{equation}
        \forall k \geq K_0 (1+\delta),~~
        \norm{v^k - v^*(\EEE)}_\infty,~
        \norm{w^k - w^*(\EEE)}_\infty
        \leq B\, \frac{1 + \delta + \log k}{k}
    \end{equation}
    for some constants $K_0, B$ dependent only on $\mu, \nu$, and $\EEE$.
\end{proposition*}

\begin{proposition*}[\autoref{prop:prelim:qtve_BV24_exp}, restated]
    In the same setting as \autoref{prop:prelim:qtve_BV24}, additionally suppose $\SSS = \ol\SSS$, where $\SSS, \ol\SSS$ are the sets defined in \autoref{thm:prelim:baradat_ventre_sets}.
    Then
    \begin{equation}
        \forall k \geq K_0 (1+\delta),~~
        \norm{v^k - v^*(\EEE)}_\infty,~
        \norm{w^k - w^*(\EEE)}_\infty
        \leq B (1+\delta) \left( 1-e^{-R(1+\delta)} \right)^k
    \end{equation}
    for some constants $K_0, B, R$ dependent only on $\mu, \nu$, and $\EEE$.
\end{proposition*}

Throughout this appendix, $\mu, \nu, C$ are fixed, $\SSS \subset \ol\SSS \subset \EEE$ are as defined in \autoref{thm:prelim:baradat_ventre}, \autoref{thm:prelim:baradat_ventre_sets},
and we abbreviate $v^*(\EEE), w^*(\EEE), \mu^*(\EEE), \nu^*(\EEE)$ to $v^*, w^*, \mu^*, \nu^*$ respectively.
Moreover, throughout, $v^k, w^k, U^k$ denote the iterates of the Sinkhorn algorithm in the formulation \eqref{eq:prelim:upd_vwU} and it is assumed that at initialization, $\forall (i,j) \in \EEE, -\delta \leq U^0_{ij} \leq \delta$.

\subsection{Preparatory lemmas}

We start by showing a slow $O(1/\sqrt k)$ convergence rate, as it will be needed to bootstrap the analysis.
To show the slow rate, it is sufficient to follow the same steps as \cite[proof of Theorem~3.2]{baradat2024convergence}, keeping track of constants slightly more explicitly.

\begin{lemma} \label{lm:apx_pf_qtve_BV24:slow_rate}
    We have
    \begin{equation}
        \forall k \geq 4,~~
        \norm{v^k - v^*}_\infty,~
        \norm{w^k - w^*}_\infty
        \leq \sqrt{\frac{B_1 (1+\delta)}{k}}
    \end{equation}
    for some constant $B_1$ dependent only on $\mu_{\min}$ and $\nu_{\min}$.
\end{lemma}

\begin{proof}
    Recall that the primal variables are given by
    $\pi^k_{ij} = e^{U^k_{ij}} \mu_i \nu_j / Z^k$,
    $Z^k = \sum_{i'j'} e^{U^k_{i'j'}} \mu_{i'} \nu_{j'}$
    for all $k \geq 0$
    and that $Z^k = 1$ for all $k \geq 1$.
    Also recall from \eqref{eq:prelim:rel_vk_Xpik} that
    for any $k \geq 0$ even,
    $v^{k+1}_i 
    = \log \left( \mu_i / (X_\sharp \pi^k)_i \right) - \log Z^k$,
    and likewise for the $w^k$.
    
    Following \cite[Eq.~(3.11)]{baradat2024convergence}, first note that 
    for any $Q \in \Delta_\EEE$ such that $X_\sharp Q = \mu^*$ and $Y_\sharp Q = \nu$,
    \begingroup
    \allowdisplaybreaks
    \begin{align}
        &
        \forall k \geq 0 ~\text{even,}~~
        \Hdiv{Q}{\pi^k} - \Hdiv{Q}{\pi^{k+2}}
        = \sum_{(i,j) \in \EEE} Q_{ij} \log \left( \pi^{k+2}_{ij} / \pi^k_{ij} \right) \\
        &= \sum_{(i,j) \in \EEE} Q_{ij} \left(
            U^{k+2}_{ij} - U^k_{ij}
        \right)
        - \log Z^{k+2} + \log Z^k \\
        &= \sum_{(i,j) \in \EEE} Q_{ij} \left(
            v^{k+1}_i + w^{k+2}_j
        \right) 
        - \log Z^{k+2} + \log Z^k \\
        &= \sum_i (X_\sharp Q)_i \, v^{k+1}_i
        + \sum_j (Y_\sharp Q)_j \, w^{k+2}_j 
        - \log Z^{k+2} + \log Z^k \\
        &= \sum_i \mu^*_i \, \log \left( \mu_i / (X_\sharp \pi^k)_i \right)
        - \log Z^k 
        + \sum_j \nu_j \log \left( \nu_j / (Y_\sharp \pi^{k+1})_j \right) 
        - \log Z^{k+1} 
        - \log Z^{k+2} + \log Z^k \\
        &= \Hdiv{\mu^*}{X_\sharp \pi^k}
        - \Hdiv{\mu^*}{\mu}
        + \Hdiv{\nu}{Y_\sharp \pi^{k+1}}
    \end{align}
    \endgroup
    since $\log Z^k = 0$ for all $k>0$.
    So by a telescopic sum,
    for any $K \geq 2$ even,
    \begin{equation}
        \Hdiv{Q}{\pi^0}
        \geq \Hdiv{Q}{\pi^0} - \Hdiv{Q}{\pi^K}
        = \sum_{\substack{k=0 \\ k \text{ even}}}^{K-2} \Hdiv{\mu^*}{X_\sharp \pi^k}
        + \left[
            \Hdiv{\nu}{Y_\sharp \pi^{k+1}}
            - \Hdiv{\mu^*}{\mu}
        \right].
    \end{equation}
    Now by \cite[Proposition~6.10]{nutz2021introduction}---or rather, a straightforward adaptation thereof to the case with infinite costs---%
    $\Hdiv{\nu}{Y_\sharp \pi^{k+1}}
    \geq \Hdiv{X_\sharp \pi^{k+2}}{\mu}$
    and
    the sequence 
    $\left( \Hdiv{X_\sharp \pi^k}{\mu} \right)_{k \in 2\NN}$
    is non-increasing.
    Moreover, by definition of
    \begin{equation}
        \mu^* = \argmin_{\ol\mu}~ \Hdiv{\ol\mu}{\mu}
        ~~~~\text{subject to}~~~~
        \exists Q \in \Delta_\EEE;~~
        \begin{cases}
            X_\sharp Q = \ol\mu \\
            Y_\sharp Q = \nu,
        \end{cases}
    \end{equation}
    since $X_\sharp \pi^k$ is feasible for this optimization problem by definition, then
    $\Hdiv{X_\sharp \pi^k}{\mu} - \Hdiv{\mu^*}{\mu} \geq 0$
    for all $k$ even.
    Thus
    \begin{gather*}
        \Hdiv{Q}{\pi^0}
        \geq \sum_{\substack{k=0 \\ k \text{ even}}}^{K-2}
        \underbrace{
            \Hdiv{\mu^*}{X_\sharp \pi^k}
        }_{\geq 0}
        + \left[
            \Hdiv{X_\sharp \pi^{k+2}}{\mu}
            - \Hdiv{\mu^*}{\mu}
        \right] \\
        \text{and so}~~~~
        \Hdiv{X_\sharp \pi^K}{\mu}
        - \Hdiv{\mu^*}{\mu}
        \leq \frac{1}{K/2}
        \sum_{\substack{k=0 \\ k \text{ even}}}^{K-2}
        \left[
            \Hdiv{X_\sharp \pi^{k+2}}{\mu}
            - \Hdiv{\mu^*}{\mu}
        \right]
        \leq \frac{1}{K/2}
        \Hdiv{Q}{\pi^0}.
    \end{gather*}
    Furthermore, 
    \begin{align}
        \Hdiv{Q}{\pi^0}
        &= \sum_{(i',j') \in \EEE} Q_{i'j'} 
        \left( 
            \log \frac{Q_{i'j'}}{\mu_{i'} \nu_{j'}}
            - U^0_{i'j'}
        \right) 
        + \log \sum_{i',j'} e^{U^0_{i'j'}} \mu_{i'} \nu_{j'} \\
        &\leq \max_{\pi \in \Delta_{m \times n}} \Hdiv{\pi}{\mu \otimes \nu}
        + \max_{(i',j') \in \EEE} (-U^0_{i'j'})
        + \max_{i',j'} U^0_{i'j'}
        \leq -\log (\mu_{\min} \nu_{\min})
        + 2 \delta.
    \end{align}
    
    It remains to relate 
    $\Hdiv{X_\sharp \pi^K}{\mu}
    - \Hdiv{\mu^*}{\mu}$
    to $\norm{v^{K+1}-v^*}_\infty$,
    where we recall that $v^{K+1}_i = \log \left( \mu_i / (X_\sharp \pi^K)_i \right)$
    and
    $v^*_i = \log \left( \mu_i / \mu^*_i \right)$.
    First note that, since the feasible set of the optimization problem defining $\mu^*$ is convex and $X_\sharp \pi^K - \mu^*$ belongs to its tangent cone at the minimizer $\mu^*$, then
    \begin{equation}
        (X_\sharp \pi^K - \mu^*)^\top
        \restr{\nabla_{\ol\mu} \Hdiv{\ol\mu}{\mu}}{\mu^*}
        = (X_\sharp \pi^K - \mu^*)^\top
        (\log \mu^* - \log \mu)
        \geq 0
    \end{equation}
    with $\log$ applied pointwise.
    On the other hand, by Bregman three-point identity,
    \begin{align}
        & \Hdiv{X_\sharp \pi^K}{\mu}
        = \Hdiv{X_\sharp \pi^K}{\mu^*}
        + \Hdiv{\mu^*}{\mu}
        - (X_\sharp \pi^K - \mu^*)^\top
        (\log \mu - \log \mu^*) \\
        \text{so}~~~~
        & \Hdiv{X_\sharp \pi^K}{\mu}
        - \Hdiv{\mu^*}{\mu}
        \geq \Hdiv{X_\sharp \pi^K}{\mu^*}.
    \end{align}
    Thus by Pinsker's inequality,
    \begin{align}
        \forall K \geq 2 ~\text{even},~~
        \frac12 \norm{X_\sharp \pi^k-\mu^*}_1^2
        &\leq \Hdiv{X_\sharp \pi^K}{\mu^*}
        \leq \frac{1}{K/2} \Hdiv{Q}{\pi^0}
        \leq \frac{-\log (\mu_{\min} \nu_{\min}) + 2 \delta}{K/2} \\
        \norm{X_\sharp \pi^k-\mu^*}_1
        &\leq 2 \sqrt{\frac{-\log(\mu_{\min} \nu_{\min}) + 2 \delta}{K}}.
    \end{align}
    Now by \autoref{lm:prelim:unifbound_vk_wk} and \autoref{coroll:prelim:bound_v*_w*},
    $\forall i,\, (X_\sharp \pi^K)_i = \mu_i e^{-v^{K+1}_i} \geq \mu_i \nu_{\min}$
    and
    $\mu^*_i \geq \mu_{\min} \nu_{\min}$.
    So by $a^{-1}$-Lipschitz-continuity of $\log$ over $[a, +\infty)$ applied with $a = \mu_{\min} \nu_{\min}$,
    \begin{align}
        \forall K \geq 2 ~\text{even},~~
        \forall i,~~
        \abs{v^{K+1}_i-v^*_i}
        &= \abs{\log \left( \mu^*_i / (X_\sharp \pi^K)_i \right)}
        \leq \frac{1}{\mu_{\min} \nu_{\min}} \abs{\mu^*_i - (X_\sharp \pi^k)_i} \\
        \norm{v^{K+1}-v^*}_\infty
        &\leq \frac{1}{\mu_{\min} \nu_{\min}} \norm{\mu^* - X_\sharp \pi^k}_\infty 
        \leq \frac{2 \sqrt{-\log(\mu_{\min} \nu_{\min}) + 2 \delta}}{\mu_{\min} \nu_{\min}} \cdot
        \frac{1}{\sqrt{K}}.
    \end{align}
    This shows the claimed convergence bound on the $v^k$.
    The bound for the $w^k$ follows similarly.
\end{proof}

In preparation for the second lemma, let us show the following auxiliary claims.

\begin{claim} \label{claim:apx_pf_qtve_BV24:claim1}
    For any $k \geq 3$,~
    $\forall i, \max_j e^{U^k_{ij}} \geq \nu_{\min}$
    ~and~
    $\forall j, \max_i e^{U^k_{ij}} \geq \mu_{\min}$.
\end{claim}

\begin{proof}
    As noted in the proof of \autoref{lm:prelim:unifbound_vk_wk},
    by definition of the update \eqref{eq:prelim:upd_vwU},
    \begin{align*}
        \forall k \geq 0 ~\text{even},~
        \forall i,~~
        & \sum\nolimits_j e^{U^{k+1}_{ij}} \nu_j 
        = 1 
        \qquad \text{so} \qquad
        \exists j;~ e^{U^{k+1}_{ij}} \geq 1
        ~~~\text{and}~~~
        e^{U^{k+2}_{ij}}
        = e^{U^{k+1}_{ij}} e^{w^{k+2}_j}
        \geq \nu_{\min} \\
        \forall k \geq 1 ~\text{odd},~
        \forall j,~~
        & \sum\nolimits_i e^{U^{k+1}_{ij}} \mu_i
        = 1
        \qquad \text{so} \qquad
        \exists i;~ e^{U^{k+1}_{ij}} \geq 1
        ~~~\text{and}~~~
        e^{U^{k+2}_{ij}}
        = e^{U^{k+1}_{ij}} e^{v^{k+2}_i}
        \geq \mu_{\min}
    \end{align*}
    where the last inequality on each line follows from the lower bounds of \autoref{lm:prelim:unifbound_vk_wk}.
\end{proof}

\begin{claim} \label{claim:apx_pf_qtve_BV24:claim2}
    For any $c>0$ and $U, U' \in (\RR \cup \{-\infty\})^{m \times n}$ such that 
    $\forall i, \max_j e^{U_{ij}}, \max_j e^{U'_{ij}} \geq c$
    and
    $\forall j, \max_i e^{U_{ij}}, \max_i e^{U'_{ij}} \geq c$,
    we have
    \begin{equation}
        \norm{v[U] - v[U']}_\infty,~
        \norm{w[U] - w[U']}_\infty
        \leq c^{-1} (\mu_{\min} \wedge \nu_{\min})^{-1} \,
        \max_{ij} \abs{e^{U_{ij}} - e^{U'_{ij}}}.
    \end{equation}
\end{claim}

\begin{proof}
    For any $i$, by definition,
    $\abs{v[U]_i - v[U']_i}
    = \abs{\log \sum_j e^{U_{ij}} \nu_j - \log \sum_j e^{U'_{ij}} \nu_j}$.
    Now by the assumption,
    $\sum_j e^{U_{ij}} \nu_j
    \geq c \, \nu_{\min}$
    and likewise for $U'$.
    So by $a^{-1}$-Lipschitz-continuity of $\log$ over $[a, +\infty)$,
    \begin{equation}
        \abs{v[U]_i - v[U']_i}
        \leq c^{-1} \nu_{\min}^{-1}
        \abs{\sum\nolimits_j (e^{U_{ij}} - e^{U'_{ij}}) \nu_j}
        \leq c^{-1} \nu_{\min}^{-1}\, 
        \max\nolimits_j
        \abs{e^{U_{ij}} - e^{U'_{ij}}}.
    \end{equation}
    Hence the bound on $\norm{v[U] - v[U']}_\infty$, and the bound on $\norm{w[U] - w[U']}_\infty$ follows similarly.
\end{proof}

\begin{claim} \label{claim:apx_pf_qtve_BV24:claim3}
    For any $k \geq 1$,
    $\max_{ij} e^{U^k_{ij}} \leq (\mu_{\min} \wedge \nu_{\min})^{-1}$.
\end{claim}

\begin{proof}
    As noted in the proof of \autoref{lm:prelim:unifbound_vk_wk}, 
    by definition of the update \eqref{eq:prelim:upd_vwU},
    for any $k \geq 1$ odd, for any $i, j$,
    $e^{U^k_{ij}} \nu_{\min}
    \leq \sum_{j'} e^{U^k_{ij'}} \nu_{j'} = 1$
    so
    $e^{U^k_{ij}} \leq \nu_{\min}^{-1}$.
    Likewise, for any $k \geq 1$ even, 
    $e^{U^k_{ij}} \leq \mu_{\min}^{-1}$.
\end{proof}

The second lemma quantifies the deviation between the true Sinkhorn iterates and the iterates of the algorithm artificially restricted to $\ol\SSS$ starting from some iteration $k_0$.

\begin{lemma} \label{lm:apx_pf_qtve_BV24:control_E_olS}
    For any $k_0$ even, let $\ol v^{(k_0) k}, \ol w^{(k_0) k}, \ol U^{(k_0) k}$ denote the iterates of the Sinkhorn algorithm in the formulation \eqref{eq:prelim:upd_vwU}, applied to the EOT problem with target marginals $\mu, \nu$ and cost matrix
    $
        \ol C_{ij} = \begin{cases}
            C_{ij} ~~\text{if}~ (i,j) \in \ol\SSS \\
            \infty ~~\text{otherwise}
        \end{cases} 
    $
    and initialized at
    $
        \ol U^{(k_0) k_0}_{ij} = \begin{cases}
            U^{k_0}_{ij} ~~\text{if}~ (i,j) \in \ol\SSS \\
            -\infty ~~\text{otherwise}
        \end{cases}
    $.
    There exist constants $K_2, B_2, \gamma$ dependent only on $\mu, \nu$, and $\EEE$ such that for any $\Delta>0$,
    if $k_0 = K_2 (1+\delta) + \Delta$,
    \begin{equation}
        \forall k \geq k_0,~~~
        \max_{(i,j) \in \ol\SSS} \abs{\ol U^{(k_0) k}_{ij} - U^k_{ij}},~~
        \norm{v^k - \ol v^{(k_0) k}}_\infty,~~
        \norm{w^k - \ol w^{(k_0) k}}_\infty
        \leq
        B_2 \, e^{-\Delta \gamma/4}.
    \end{equation}
\end{lemma}

\begin{proof}
    Recall that $\ol\SSS = \left\{ (i,j) \in \EEE;~ v^*_i+w^*_j = 0 \right\}$.
    If $\ol\SSS = \EEE$, there is nothing to prove, so suppose henceforth the inclusion $\ol\SSS \subset \EEE$ is strict.
    Let
    \begin{equation}
        -\gamma 
        = \max_{(i,j) \in \EEE \setminus \ol\SSS} v^*_i + w^*_j
        < 0
    \end{equation}
    (and note that $\gamma$ depends only on $\mu, \nu, \EEE$).
    Let $B_1$ be as in \autoref{lm:apx_pf_qtve_BV24:slow_rate}
    and $k_1 = 4 \vee \ceil{\frac{16 B_1}{\gamma^2} (1+\delta)}$.
    Then
    \begin{equation}
        \forall k \geq k_1,~
        \max_{(i,j) \in \EEE \setminus \ol\SSS}
        v^k_i + w^k_j
        \leq -\gamma 
        + \norm{v^k-v^*}_\infty
        + \norm{w^k-w^*}_\infty
        \leq -\gamma + 2 \sqrt{\frac{B_1 (1+\delta)}{k}}
        \leq -\frac12 \gamma.
    \end{equation}
    Note that by \autoref{lm:prelim:unifbound_vk_wk},
    \begin{align}
        \forall i,j,~~
        U^{k_1}_{ij}
        &= U^0_{ij}
        + v^1_i + w^2_j
        + \sum_{\substack{l=2\\l~\text{even}}}^{k_1-1}
        v^{l+1}_i
        + \sum_{\substack{l=3\\l~\text{odd}}}^{k_1-1}
        w^{l+1}_i \\
        &\leq \delta -\log(\mu_{\min} \nu_{\min}) + 3\delta
        + (k_1-2)\, [-\log(\mu_{\min} \wedge \nu_{\min})]
        \leq 4 \delta - k_1 \, \log(\mu_{\min} \wedge \nu_{\min}).
    \end{align}
    Consequently,
    \begin{align}
        \forall (i,j) \in \EEE \setminus \ol\SSS,~~~
        &\forall l \geq 0 ~\text{even},~~
        U^{k_1+l}_{ij}
        \leq U^{k_1}_{ij}
        - (l/2) \frac{\gamma}{2} 
        \leq 
        4 \delta - k_1 \, \log(\mu_{\min} \wedge \nu_{\min})
        - l \, \gamma/4 \\
        &\forall l \geq 0 ~\text{odd},~~~
        U^{k_1+l}_{ij}
        \leq U^{k_1+l-1}_{ij}
        - \log (\mu_{\min} \wedge \nu_{\min})
    \end{align}
    where for the odd case we used \autoref{lm:prelim:unifbound_vk_wk} again.
    Thus, recalling the definition of $k_1 = 4 \vee \ceil{\frac{16 B_1}{\gamma^2} (1+\delta)}$,
    \begin{equation}
        \forall (i,j) \not\in \ol\SSS,~~~
        \forall k \geq k_1,~~
        U^k_{ij} 
        \leq B_2' (1+\delta)
        - (k-k_1) \,\gamma/4
    \end{equation}
    for a constant $B_2'$ dependent only on $\mu, \nu$, and $\EEE$.

    Now set $k_2 = k_1 + (4/\gamma) \left[ B_2' (1+\delta) - 2 \log(\mu_{\min} \wedge \nu_{\min}) \right]$.
    Then
    \begin{align}
        \forall (i,j) \not\in \ol\SSS,~~~
        \forall k \geq k_2,~~
        U^k_{ij}
        &\leq 2 \log (\mu_{\min} \wedge \nu_{\min}) \\
        e^{U^k_{ij}} 
        &\leq (\mu_{\min} \wedge \nu_{\min})^2
        < \mu_{\min} \wedge \nu_{\min}.
    \end{align}
    So we can refine the result of \autoref{claim:apx_pf_qtve_BV24:claim1} by affirming that for any $i$, the $j$ for which $e^{U^k_{ij}} \geq \nu_{\min}$ must be such that $(i,j) \in \ol\SSS$, and likewise for the other estimate; formally,
    \begin{align}
        \forall k \geq k_2,~~~
        & \forall i,~ \exists j;~ (i,j) \in \ol\SSS ~~\text{and}~~
        e^{U^k_{ij}} \geq \nu_{\min} \\
        & \forall j,~ \exists i;~ (i,j) \in \ol\SSS ~~\text{and}~~
        e^{U^k_{ij}} \geq \mu_{\min}.
    \end{align}
    Note that we can write $k_2 = K_2 (1+\delta)$ with $K_2$ dependent only on $\mu, \nu$, and $\EEE$.
    
    Fix $\Delta>0$ and $k_0 = k_2 + \Delta$ even and consider $\ol v^{(k_0) k}, \ol w^{(k_0) k}, \ol U^{(k_0) k}$ as in the lemma statement, 
    abbreviated in the rest of this proof as $\ol v^k, \ol w^k, \ol U^k$.
    Let us bound $\max_i \ol v^{k_0+1}_i$ and $\max_j \ol w^{k_0+2}_j$.
    The previous paragraph and the fact that $\forall (i,j) \in \ol\SSS,~ \ol U^{k_0}_{ij} = U^{k_0}_{ij}$ by definition imply that the assumptions of \autoref{claim:apx_pf_qtve_BV24:claim2} are verified for $U = U^{k_0}$, $U' = \ol U^{k_0}$, and $c = \mu_{\min} \wedge \nu_{\min}$, so
    \begin{equation}
        \norm{\ol v^{k_0+1} - v^{k_0+1}}_\infty 
        \leq (\mu_{\min} \wedge \nu_{\min})^{-2} \, 
        \max_{ij} \abs{e^{U^{k_0}_{ij}} - e^{\ol U^{k_0}_{ij}}}
        = (\mu_{\min} \wedge \nu_{\min})^{-2} \, 
        \max_{(i,j) \in \EEE \setminus \ol\SSS} e^{U^{k_0}_{ij}}
        \leq 1
    \end{equation}
    where in the middle equality we used that for $(i,j) \in \ol\SSS,~ \ol U^{k_0}_{ij} = U^{k_0}_{ij}$
    and for $(i,j) \not\in \ol\SSS$, $e^{\ol U^{k_0}_{ij}} = 0$,
    and in the last inequality we used the previous paragraph again.
    In particular by \autoref{lm:prelim:unifbound_vk_wk},
    \begin{equation}
        \max_i \ol v^{k_0+1}_i
        \leq \max_i v^{k_0+1}_i + 1
        \leq -\log \nu_{\min} + 1.
    \end{equation}
    To bound $\max_j \ol w^{k_0+2}_j$, let us apply \autoref{claim:apx_pf_qtve_BV24:claim2} to $U = U^{k_0+1}$, $U' = \ol U^{k_0+1}$, and 
    $c = e^{-1} (\mu_{\min} \wedge \nu_{\min})$. 
    Indeed, the assumption on $U'$ is verified as
    \begin{align}
        \forall (i,j) \in \ol\SSS,~~
        \ol U^{k_0+1}_{ij}
        &= \ol U^{k_0}_{ij}
        + \ol v^{k_0+1}_i
        = U^{k_0+1}_{ij}
        - v^{k_0+1}_i
        + \ol v^{k_0+1}_i \\
        e^{\ol U^{k_0+1}_{ij}}
        &\geq e^{U^{k_0+1}_{ij}} e^{-1}.
    \end{align}
    Applying the claim yields
    \begin{equation}
        \max_j \ol w^{k_0+2}_j 
        \leq \max_j w^{k_0+2}_j 
        + e\, (\mu_{\min} \wedge \nu_{\min})^{-2} \,
        \max_{ij} \abs{e^{U^{k_0+1}_{ij}} - e^{\ol U^{k_0+1}_{ij}}}.
    \end{equation}
    The first term is upper-bounded by $-\log\mu_{\min}$ by \autoref{lm:prelim:unifbound_vk_wk}.
    In the second term, the $\max_{ij}$ decomposes into a max over $\ol\SSS$ and a max over $\EEE \setminus \ol\SSS$.
    The latter one is equal to
    $\max_{(i,j) \in \EEE \setminus \ol\SSS}\, e^{U^{k_0+1}_{ij}}$
    and is upper-bounded by $(\mu_{\min} \wedge \nu_{\min})^2$ by the previous paragraph.
    As for the max over $\ol\SSS$, we have
    \begin{align*}
        \forall i,j,~
        \abs{e^{U^{k_0+1}_{ij}} - e^{\ol U^{k_0+1}_{ij}}}
        &= e^{U^{k_0+1}_{ij}} 
        \abs{1 - e^{\ol U^{k_0+1}_{ij} - U^{k_0+1}_{ij}}}
        = e^{U^{k_0+1}_{ij}} 
        \abs{1 - e^{\ol v^{k_0+1}_i - v^{k_0+1}_i}}
        \leq
        e^{U^{k_0+1}_{ij}} 
        (1+e) \\
        \max_{(i,j) \in \ol\SSS} \abs{e^{U^{k_0+1}_{ij}} - e^{\ol U^{k_0+1}_{ij}}}
        &\leq (1+e) \max_{(i,j) \in \ol\SSS} e^{U^{k_0+1}_{ij}}
        \leq (1+e)
        (\mu_{\min} \wedge \nu_{\min})^{-1}
    \end{align*}
    by \autoref{claim:apx_pf_qtve_BV24:claim3}.
    In summary, we have shown that
    \begin{equation}
        \max\nolimits_i \ol v^{k_0+1}_i
        \leq -\log \nu_{\min} + 1
        \qquad \text{and} \qquad
        \max\nolimits_j \ol w^{k_0+2}_j \leq 
        -\log \mu_{\min}
        + e\, (1+e) (\mu_{\min} \wedge \nu_{\min})^{-3}.
    \end{equation}
    
    Denote by $(f^k, g^k)_{k \geq 0}$ and $(\ol f^k, \ol g^k)_{k \geq k_0}$ the dual iterates corresponding to $v^k, w^k, U^k$ resp.\ $\ol v^k, \ol w^k, \ol U^k$ in the original formulation of the Sinkhorn algorithm.
    That is, $v^k = (f^k-f^{k-2})/\tau$, $w^k = (g^k-g^{k-2})/\tau$, $U^k_{ij} = (f^k_i+g^k_j-C_{ij})/\tau$ for all $k \geq 2$, and likewise for the $\ol f^k, \ol g^k$ for $k \geq k_0$.
    In particular $(\ol f^{k_0}, \ol g^{k_0}) = (f^{k_0}, g^{k_0})$.
    Moreover, set 
    $-M^k = B_2' (1+\delta)
    - (k-k_1) \,\gamma/4$
    and recall from the first paragraph of this proof that
    $\forall (i,j) \not\in \ol\SSS,~
    \forall k \geq k_0,~
    U^k_{ij} 
    \leq -M^k$.
    Furthermore, note that
    \begin{equation}
        \forall k \geq k_0 = k_2 + \Delta,~~~
        -M^k 
        = -M^{k_2} - (k-k_2) \frac\gamma4
        = 2 \log (\mu_{\min} \wedge \nu_{\min})
        - (k-k_0 + \Delta) \frac\gamma4.
    \end{equation}
    Then by \autoref{lm:prelim:control_finite_infinite_costs},
    for all $k \geq k_0+2$,
    \begin{align*}
        \norm{f^k-\ol f^k}_\infty \!,
        \norm{g^k-\ol g^k}_\infty
        &\leq
        \tau \left[
            e^{(\max \ol v^{k_0+1}) - M^{k_0}}
            + e^{(\max \ol w^{k_0+2}) - M^{k_0+1}}
            + (\mu_{\min} \!\wedge\! \nu_{\min})^{-1}
            \! \sum_{l=k_0+3}^{k-1} e^{-M^l}
        \right] \\
        &\leq \tau 
        \left[ 
            e^{(\max \ol v^{k_0+1})}
            \vee 
            e^{(\max \ol w^{k_0+2})}
            \vee
            (\mu_{\min} \!\wedge\! \nu_{\min})^{-1}
        \right]
        \sum_{l=k_0}^\infty e^{-M^l} \\
        &\leq \tau B_2'' \, e^{-\Delta \gamma/4}
    \end{align*}
    for a constant $B_2''$ dependent only on $\mu, \nu, \EEE$.
    In particular,
    for all $k \geq k_0+2$,
    \begin{equation}
        \forall (i,j) \in \ol\SSS,~
        \abs{\ol U^k_{ij} - U^k_{ij}}
        = \abs{(\ol f^k_i+\ol g^k_j)/\tau - (f^k_i+g^k_j)/\tau}
        \leq 2 B_2'' \, e^{-\Delta \gamma/4}
    \end{equation}
    and for all $k \geq k_0+4$,
    \begin{equation}
        \norm{v^k - \ol v^k}_\infty
        = \norm{(f^k-f^{k-2})/\tau - (\ol f^k-\ol f^{k-2})/\tau}_\infty
        \leq 2 B_2'' \, e^{-\Delta \gamma/4}
    \end{equation}
    and likewise for the $w^k - \ol w^k$.
\end{proof}

The following lemma is an adaptation of \cite[proof of Proposition~5.3]{baradat2024convergence}.
It shows that to analyze the algorithm artificially restricted to $\ol\SSS$, it is equivalent to analyze the convergence of the Sinkhorn algorithm in the asymptotically scalable case.

\begin{lemma} \label{lm:apx_pf_qtve_BV24:prop5.3_BV24}
    Consider any $k_0$ even and let $\ol v^{(k_0) k}, \ol w^{(k_0) k}, \ol U^{(k_0) k}$ be as in \autoref{lm:apx_pf_qtve_BV24:control_E_olS}, so that
    $\left\{ (i,j);~ \ol U^{(k_0) k}_{ij} > -\infty \right\} = \ol\SSS$.
    Let $\tv^k, \tw^k, \tU^k$ denote the iterates of the Sinkhorn algorithm, in the formulation \eqref{eq:prelim:upd_vwU}, applied to the EOT problem with target marginals $\mu^*, \nu$ and initialized at $\tU^{k_0}_{ij} = \ol U^{(k_0) k_0}_{ij} + v^*_i$.
    Then
    \begin{equation}
        \forall k \geq k_0 ~\text{even},~~
        \forall i,j,~~
        \tU^k_{ij} = \ol U^{(k_0) k}_{ij} + v^*_i.
    \end{equation}

    Symmetrically, if $\tv^{\prime k}, \tw^{\prime k}, \tU^{\prime k}$ denote the iterates of the Sinkhorn algorithm applied to the EOT problem with target marginals $\mu, \nu^*$ and initialized at $\tU^{\prime k_0+1}_{ij} = \ol U^{(k_0) k_0+1}_{ij} + w^*_j$,
    then
    \begin{equation}
        \forall k \geq k_0 ~\text{odd},~~
        \forall i,j,~~
        \tU^{\prime k}_{ij} = \ol U^{(k_0) k}_{ij} + w^*_j.
    \end{equation}
\end{lemma}

\begin{proof}
    Abbreviate $\ol v^{(k_0) k}, \ol w^{(k_0) k}, \ol U^{(k_0) k}$ as $\ol v^k, \ol w^k, \ol U^k$.
    Let us prove the first part of the lemma, and the second part will follow by the symmetric arguments.
    First note that the desired equality is trivial for the $(i,j) \not\in \ol\SSS$ as both sides equal $-\infty$, so it suffices to show it for the $(i,j) \in \ol\SSS$.
    
    We proceed by induction.
    The claimed equality is true by definition at $k=k_0$.
    Let any $k \geq k_0$ even and suppose that the equality holds at $k$.
    To show that it holds also at $k+2$, it suffices to show that 
    $\forall (i,j) \in \ol\SSS,~ \tv^{k+1}_i + \tw^{k+2}_j = \ol v^{k+1}_i + \ol w^{k+2}_j$.
    Now by definition,
    \begin{align}
        \forall i,~
        \tv^{k+1}_i
        &= -\log \sum\nolimits_j e^{\tU^k_{ij}} \nu_j
        & \qquad
        \forall j,~
        \tw^{k+2}_j
        &= -\log \sum\nolimits_i
        e^{\tU^{k+1}_{ij}} \mu^*_i
        = -\log \sum\nolimits_i
        e^{\tU^k_{ij} + \tv^{k+1}_i} \mu^*_i \\
        \text{and} \quad~
        \forall i,~
        \ol v^{k+1}_i
        &= -\log \sum\nolimits_j e^{\ol U^k_{ij}} \nu_j
        & \qquad
        \forall j,~
        \ol w^{k+2}_j
        &= -\log \sum\nolimits_i
        e^{\ol U^{k+1}_{ij}} \mu_i
        = -\log \sum\nolimits_i
        e^{\ol U^k_{ij} + \ol v^{k+1}_i} \mu_i.
    \end{align}
    So by the induction hypothesis, for any $i$,
    \begin{equation}
        \tv^{k+1}_i 
        = -\log \sum\nolimits_j e^{\tU^k_{ij}} \nu_j
        = -\log \sum\nolimits_j e^{\ol U^k_{ij} + v^*_i} \nu_j
        = \ol v^{k+1}_i - v^*_i
    \end{equation}
    and for any $j$,
    \begin{align}
        \tw^{k+2}_j
        &= -\log \sum\nolimits_i e^{\tU^k_{ij}} \cdot e^{\tv^{k+1}_i} \cdot \mu^*_i \\
        &= -\log \sum\nolimits_i e^{\ol U^k_{ij} + v^*_i} \cdot e^{\ol v^{k+1}_i - v^*_i} \cdot e^{-v^*_i} \mu_i
        \\
        &= -\log \sum_{i: (i,j) \in \ol\SSS} e^{\ol U^k_{ij} + \ol v^{k+1}_i} \mu_i \cdot
        ~\underbrace{ e^{-v^*_i} }_{e^{w^*_j}}~
        = \ol w^{k+2}_j - w^*_j
    \end{align}
    since $v^*_i = -w^*_j$ for all $(i,j) \in \ol\SSS$.
    Thus for any $(i,j) \in \ol\SSS$,
    \begin{equation}
        \tv^{k+1}_i + \tw^{k+2}_j
        = \ol v^{k+1}_i + \ol w^{k+2}_j
        - v^*_i - w^*_j
        = \ol v^{k+1}_i + \ol w^{k+2}_j,
    \end{equation}
    which concludes the proof by induction.
\end{proof}

\subsection{Proof of \autoref{prop:prelim:qtve_BV24}}

The next lemma shows a convergence bound for the algorithm artificially restricted to $\ol\SSS$ starting from iteration $k_0$,
with special care given to the dependency of the bound on $k_0$.

\begin{lemma} \label{lm:apx_pf_qtve_BV24:asymp_scalable_logk/k}
    Consider any $k_0 \geq 4$ even and let $\ol v^{(k_0) k}, \ol w^{(k_0) k}, \ol U^{(k_0) k}$ be as in \autoref{lm:apx_pf_qtve_BV24:control_E_olS}.
    Then
    \begin{equation}
        \forall k \geq k_0 + K_4,~~
        \norm{\ol v^{(k_0) k} - v^*}_\infty,~
        \norm{\ol w^{(k_0) k} - w^*}_\infty
        \leq B_4\,
        \frac{1 + \delta + k_0 + \log (k-k_0)}{k-k_0}
    \end{equation}
    for some constants $K_4, B_4$ dependent only on $\mu, \nu$, and $\EEE$.
\end{lemma}

\begin{proof}
    Abbreviate $\ol v^{(k_0) k}, \ol w^{(k_0) k}, \ol U^{(k_0) k}$ as $\ol v^k, \ol w^k, \ol U^k$.
    Let $\tv^k, \tw^k, \tU^k$
    be defined as in \autoref{lm:apx_pf_qtve_BV24:prop5.3_BV24}
    and further pose $\tpi^k = (e^{\tU^k_{ij}} \mu^*_i \nu_j)_{ij} \in \Delta_{\ol\SSS}$ for all $k \geq 1$.
    It suffices to show that for some $K_4, B_4$ to be determined,
    \begin{equation}
        \forall k \geq k_0 + K_4 ~\text{even},~~
        \norm{\tv^{k+1}}_\infty \leq B_4 \, \frac{1 + \delta + k_0 + \log (k-k_0)}{k-k_0}.
    \end{equation}
    Indeed since $\tv^{k+1}_i = v[\tU^k]_i = v[\ol U^k]_i - v^*_i = \ol v^{k+1}_i - v^*_i$ for all $k$ even by \autoref{lm:apx_pf_qtve_BV24:prop5.3_BV24}, 
    this will show the announced bound on $\norm{\ol v^k - v^*}_\infty$, and the bound on $\norm{\ol w^k-w^*}_\infty$ will follow similarly.

    To show the above convergence bound on the $\tv^k$, 
    note that they are defined as the increments of the Sinkhorn algorithm applied to an EOT problem with marginals $\mu^*, \nu$ and finite-cost pattern~$\ol\SSS$, and this problem admits a feasible primal solution by definition of $\mu^* = \mu^*(\ol\SSS)$.
    In other words, the corresponding matrix scaling problem is asymptotically scalable.
    So by the main result of \cite{wang2026almost}, there exist $\tK_4$ and $\tB_4$ dependent only on $\mu^*, \nu$, and $\ol\SSS$---and hence only on $\mu, \nu$, and $\EEE$---such that
    \begin{equation}
        \forall k \geq k_0 + \tK_4,~~
        \norm{X_\sharp \tpi^k - \mu^*}_1
        \leq \frac{\tB_4}{k-k_0} \left( 1 - \min_{(i,j) \in \ol\SSS} \tU^{k_0}_{ij} + \max_{ij} \tU^{k_0}_{ij} + \log (k-k_0) \right).
    \end{equation}
    Moreover, note that for any $k \geq k_0 + 1$ even,
    $\tv^{k+1}_i = \log(\mu^*_i / (X_\sharp \tpi^k)_i)$ by \eqref{eq:prelim:rel_vk_Xpik}, 
    and by \autoref{lm:prelim:unifbound_vk_wk},
    $\forall i, (X_\sharp \tpi^k)_i = \mu^*_i e^{-\tv^{k+1}_i} \geq \mu^*_i \nu_{\min} \geq \mu_{\min} \nu_{\min}^2$.
    So by local Lipschitz-continuity of $\log$,
    \begin{equation}
        \norm{\tv^{k+1}}_\infty
        \leq \frac{1}{\mu_{\min} \nu_{\min}^2}
        \norm{X_\sharp \tpi^k - \mu^*}_\infty
        \leq \frac{1}{\mu_{\min} \nu_{\min}^2}
        \norm{X_\sharp \tpi^k - \mu^*}_1.
    \end{equation}
    Thus it only remains to bound 
    $\max_{ij} \tU^{k_0}_{ij} - \min_{(i,j) \in \ol\SSS} \tU^{k_0}_{ij}$.
    For this, recall that $\tU^{k_0}_{ij} = \ol U^{(k_0) k_0}_{ij} + v^*_i = U^{k_0}_{ij} + v^*_i$ for all $(i,j) \in \ol\SSS$ by definition and note that by \autoref{lm:prelim:unifbound_vk_wk},
    \begin{multline}
        \forall (i,j) \in \ol\SSS,~~
        U^{k_0}_{ij}
        = U^0_{ij}
        + v^1_i + w^2_j
        + \sum_{\substack{l=2\\l~\text{even}}}^{k_0-1}
        v^{l+1}_i
        + \sum_{\substack{l=3\\l~\text{odd}}}^{k_0-1}
        w^{l+1}_i \\
        \abs{U^{k_0}_{ij}}
        \leq \delta -2 \log(\mu_{\min} \wedge \nu_{\min}) + 3\delta
        + (k_0-2)\, [-\log(\mu_{\min} \wedge \nu_{\min})]
        \leq 4 \delta - k_0 \, \log(\mu_{\min} \wedge \nu_{\min}).
    \end{multline}
    Hence $\max_{ij} \tU^{k_0}_{ij} - \min_{(i,j) \in \ol\SSS} \tU^{k_0}_{ij} \leq 8 \delta - 2 k_0 \, \log(\mu_{\min} \wedge \nu_{\min})$.
    Substituting into the inequalities above yields the announced convergence bound on the $\tv^k$, and the lemma follows.
\end{proof}

We can now present the proof of \autoref{prop:prelim:qtve_BV24}.

\begin{proof}[Proof of \autoref{prop:prelim:qtve_BV24}]
    Let $K_2, B_2, \gamma$ be as in \autoref{lm:apx_pf_qtve_BV24:control_E_olS} and 
    let $K_4, B_4$ be as in \autoref{lm:apx_pf_qtve_BV24:asymp_scalable_logk/k}.
    Fix any $k \geq 2 \left[ K_2 (1+\delta) + K_4 \right] + K_5$, where $K_5$ is a constant dependent only on $\mu, \nu, \EEE$ which will be specified later.
    Consider any $\Delta > 0$ such that $k_0 = K_2 (1+\delta) + \Delta$ is even and $k_0 + K_4 \leq k$,
    and let $(\ol v^{(k_0)l}, \ol w^{(k_0)l}, \ol U^{(k_0)l})_{l \geq k_0}$ be defined as in \autoref{lm:apx_pf_qtve_BV24:control_E_olS}.
    Then by \autoref{lm:apx_pf_qtve_BV24:control_E_olS} for the first term and by \autoref{lm:apx_pf_qtve_BV24:asymp_scalable_logk/k} for the second term,
    \begin{align}
        \norm{v^k-v^*}_\infty
        &\leq \norm{v^k-\ol v^{(k_0) k}}_\infty
        + \norm{\ol v^{(k_0) k} - v^*}_\infty \\
        &\leq B_2 e^{-\Delta \gamma/4}
        + B_4 \frac{1+\delta + k_0 + \log(k-k_0)}{k-k_0} \\
        &= B_2 e^{-\Delta \gamma/4}
        + B_4 \frac{(1+K_2) (1+\delta) + \Delta + \log(k-k_0)}{k - K_2 (1+\delta) - \Delta}.
    \end{align}
    
    Take $\Delta = \frac4\gamma \log\left[ k - K_2 (1+\delta) - K_4 \right] - r$, where $0 \leq r \leq 2$ serves to ensure that $k_0 = K_2(1+\delta)+\Delta$ is an even integer.
    Let us check that this choice also satisfies the condition $k_0 + K_4 \leq k$, i.e., that
    \begin{gather}
        K_4 + K_2 (1+\delta) + \frac4\gamma \log\left[ k - K_2 (1+\delta) - K_4 \right] - r
        \leq k \\
        \impliedby
        \frac4\gamma \log\left[ k - K_2 (1+\delta) - K_4 \right]
        \leq k - K_2 (1+\delta) - K_4.
    \end{gather}
    Since $\log x / x \to 0$ as $x \to \infty$, this can indeed be ensured by choosing $K_5$ larger than a constant dependent only on $\gamma$, and hence only on $\mu, \nu, \EEE$.
    For concreteness and to prepare the sequel, let us take $K_5$ such that
    $\forall x \geq K_5,~ \frac4\gamma \frac{\log x}{x} \leq \frac12$, so that
    \begin{equation}
        \frac4\gamma \log\left[ k - K_2 (1+\delta) - K_4 \right]
        \leq \frac12 \big( k - K_2 (1+\delta) - K_4 \big).
    \end{equation}
    
    Plugging this choice of $\Delta$ into the bound on $\norm{v^k-v^*}_\infty$ above, we obtain
    \begin{equation}
        \norm{v^k-v^*}_\infty
        \leq \frac{B_2~ e^{r \gamma/4}}{k-K_2 (1+\delta) - K_4}
        + B_4 \frac{(1+K_2) (1+\delta) + \frac4\gamma \log\left[ k - K_2 (1+\delta) - K_4 \right] + r + \log k}{k - K_2 (1+\delta) - \frac4\gamma \log\left[ k - K_2 (1+\delta) - K_4 \right] + r}.
    \end{equation}
    Now by assumption on $k$, the denominators are lower-bounded as
    \begin{equation}
        k - K_2 (1+\delta) - K_4
        \geq k/2 + [K_2 (1+\delta) + K_4] + K_5/2 - K_2 (1+\delta) - K_4
        \geq k/2
    \end{equation}
    and
    \begin{align}
        &~~ k - K_2 (1+\delta) - \frac4\gamma \log\left[ k - K_2 (1+\delta) - K_4 \right] + r \\
        &\geq k - K_2 (1+\delta) - K_4 - \frac4\gamma \log\left[ k - K_2 (1+\delta) - K_4 \right]
        \geq \frac12 \big( k - K_2 (1+\delta) - K_4 \big)
        \geq k/2
    \end{align}
    by our choice of $K_5$.
    Hence, 
    \begin{align}
        \norm{v^k-v^*}_\infty
        \leq
        \frac{B_2~ e^{2 \gamma/4}}{k/2}
        + B_4
        \frac{(1+K_2)(1+\delta) + (\frac4\gamma + 1) \log k + 2}{k/2}
        \leq B \frac{1+\delta+\log k}{k}
    \end{align}
    for a constant $B$ dependent only on $\mu, \nu, \EEE$.
    This proves the announced bound for the $v^k$, and the bound for the $w^k$ follows similarly.
\end{proof}

\subsection{Interlude: explicit exponential rate for the exactly scalable case}

In this section, as a preparation for the proof of \autoref{prop:prelim:qtve_BV24_exp}, we recall the following result from \cite{qu2025sinkhorn}
and we deduce a corollary which is more directly applicable for our purpose.

\begin{proposition}[{\cite[Theorem~3]{qu2025sinkhorn}}] \label{prop:apx_pf_qtve_BV24:QGGU25}
    Consider $\mu \in \Delta_m, \nu \in \Delta_n$ and $A \in \RR_+^{m \times n}$.
    Let $\EEE = \{ (i,j);~ A_{ij}>0 \}$
    and let $\SSS, \ol\SSS$ be defined from $\mu, \nu, \EEE$ as in \autoref{thm:prelim:baradat_ventre_sets}.
    Suppose that $A$ is exactly $(\mu, \nu)$-scalable, or equivalently by \autoref{lm:prelim:equiv_exact_scalability}, $\SSS = \ol\SSS = \EEE$,
    and that the bipartite graph with edge set $\EEE$ is connected.
    Let $C \in (\RR \cup \{\infty\})^{m \times n}$ such that $A_{ij} = e^{-C_{ij}} \mu_i \nu_j$
    and denote by $(f^k, g^k)_{k \geq 0}$ the iterates of the Sinkhorn algorithm applied to the EOT problem with marginals $\mu, \nu$, cost matrix $C$, and temperature $\tau=1$.
    %
    Further denote
    $\Psi(f, g) = \sum_{ij} A_{ij} e^{f_i+g_j} - \mu^\top f - \nu^\top g$.
    Then 
    \begin{equation}
        \forall k \geq 0,~
        \Psi(f^{k+1}, g^{k+1}) - \min \Psi
        \leq \left( 
            1 - \frac{e^{-2B}~ \lambda_2(\LLL[A])}
            {\big( \max_i \sum_j A_{ij} \big) \wedge \big( \max_j \sum_i A_{ij} \big)} 
        \right)
        (\Psi(f^k, g^k) - \min \Psi)
    \end{equation}
    where%
    \footnote{Compared to the original statement of \cite[Theorem~3]{qu2025sinkhorn}, the constant ``$B$'' in our statement is slightly smaller. One can check by inspecting step~II of their proof that our $B$ is indeed all they need.}
    $B = \sup_{k \geq 0} \max_{ij} \abs{f^k_i + g^k_j}$
    and
    $\lambda_2(\LLL[A])$ is the second smallest eigenvalue of the matrix
    \begin{equation}
        \LLL[A] = \begin{bmatrix}
            \mathrm{diag}(A \bmone_n) & -A \\
            -A^\top & \mathrm{diag}(A^\top \bmone_m).
        \end{bmatrix}.
    \end{equation}
\end{proposition}

Let us also recall the following result, corresponding to the first half of \cite[step II of proof of Theorem~3]{qu2025sinkhorn}.
It is intuitively consistent with the interpretation of $\LLL[A]$ as the Laplacian of the bipartite graph with edge weights $A_{ij}$, as put forward by \cite{qu2025sinkhorn}.
Indeed in this perspective, $\lambda_2(\LLL[A])$ quantifies the connectedness of the bipartite graph, and increasing the weight of an edge (its transition rate in the associated Markov jump process) only increases connectedness.

\begin{lemma} \label{prop:apx_pf_qtve_BV24:spectral_gap_lb}
    For any $A, A' \in \RR_+^{m \times n}$ such that $A_{ij} \geq A'_{ij}$ for all $i,j$,
    $\lambda_2(\LLL[A]) \geq \lambda_2(\LLL[A'])$.
    Moreover for any $c \geq 0$, $\lambda_2(\LLL[c\, A]) = c\, \lambda_2(\LLL[A])$.
    In particular, $\lambda_2(\LLL[A]) \geq [\min_{(i,j) \in \EEE} A_{ij}]\, \lambda_2(\LLL[\bmone_\EEE])$, 
    where $\EEE = \{(i,j);~ A_{ij}>0 \}$ and 
    $(\bmone_\EEE)_{ij} = 1$ if $(i,j) \in \EEE$ and $0$ otherwise.
\end{lemma}

\begin{proof}
    First note that for any $A'' \in \RR_+^{m \times n}$, $\LLL[A'']$ is diagonally dominant and so positive-semi-definite, and its smallest eigenvalue is $\lambda_1(\LLL[A'']) = 0$ with eigenvector $\bmone_{m+n}$.
    Introduce the orthogonal projector $\Pi = I_{m+n} - \bmone_{m+n} \bmone_{m+n}^\top / (m+n)$.
    Let $\Delta = \LLL[A]-\LLL[A'] = \LLL[A-A']$.
    Since $A-A'$ has all non-negative entries, then $\Delta$ is positive-semi-definite and
    $\lambda_1(\Delta)=0$ with eigenvector $\bmone_{m+n}$.
    So $\Pi \Delta$ is also positive-semi-definite and
    $\Pi \LLL[A] \succeq \Pi \LLL[A']$, and in particular $\lambda_2(\LLL[A]) \geq \lambda_2(\LLL[A'])$.
    The second part of the lemma follows immediately from the linearity of $\LLL[\cdot]$ and the non-negativity of the spectrum of $\LLL[A]$.
    The third part follows from applying successively the first part with $A' = [\min_{(i,j) \in \EEE} A_{ij}] \, \bmone_\EEE$, and the second part with $c = \min_{(i,j) \in \EEE} A_{ij}$.
\end{proof}

\begin{corollary} \label{coroll:apx_pf_qtve_BV24:QGGU25}
    Consider $\mu \in \Delta_m, \nu \in \Delta_n$, and $\EEE \subset \{1 \dots m\} \times \{1 \dots n\}$ such that the bipartite graph with edge set $\EEE$ has no isolated vertex.
    Let $\SSS, \ol\SSS$ be defined from $\mu, \nu, \EEE$ as in \autoref{thm:prelim:baradat_ventre_sets},
    and suppose that $\SSS = \ol\SSS = \EEE$.
    Denote by $v^k, w^k, U^k$ the iterates of the Sinkhorn algorithm in the formulation \eqref{eq:prelim:upd_vwU}
    (and as explained in \autoref{subsec:prelim:bg}, the corresponding cost matrix and temperature are subsumed by the initialization $U^0$).
    Suppose that $\forall (i,j) \in \EEE, -\delta \leq U^0_{ij} \leq \delta$ for some $\delta \geq 0$.
    Then
    \begin{align}
        &\forall k \geq 0,~~~
        \max_{(i,j) \in \EEE} \abs{U^k_{ij}} \leq B_7\, (1+\delta) \\
        \text{and}~~~~
        &\forall k \geq 4,~~~
        \norm{v^k}_\infty,~ \norm{w^k}_\infty
        \leq B_7\, (1+\delta) \left( 1 - e^{-R_7 (1+\delta)} \right)^k
    \end{align}
    for some constants $B_7, R_7>0$ dependent only on $\mu, \nu$, and $\EEE$.
\end{corollary}

\begin{proof}
    As explained, e.g., in \cite[Remark~1.6]{wang2026almost}, we may assume without loss of generality that $\EEE$ is connected.
    Pose $A_{ij} = e^{U^0_{ij}} \mu_i \nu_j$
    and let $\Psi(f,g) = \sum_{ij} A_{ij} e^{f_i+g_j} - 1 - \mu^\top f - \nu^\top g$.
    Consider $(f^k, g^k)_{k \geq 0}$ the iterates of the Sinkhorn algorithm in its original formulation with cost matrix $C = -U^0$, temperature $\tau=1$, and initialized at $(f^0, g^0) = (0,0)$,
    so that $U^k_{ij} = (f^k_i+g^k_j-C_{ij})/\tau$ and $v^k = f^k-f^{k-2}$, $w^k = g^k-g^{k-2}$ for all $k \geq 2$.
    Denote by $\pi^k_{ij} = e^{U^k_{ij}} \mu_i \nu_j$ for $k \geq 1$ the corresponding primal variables.
    In this proof, we will use $B, B', B'' ...$ and $R, R', R'' ...$ to denote constants dependent only on $\mu, \nu$, and $\EEE$ without introducing them each time.

    Let us estimate the quantities appearing in the bound of \autoref{prop:apx_pf_qtve_BV24:QGGU25}.
    \begin{itemize}
        \item 
        Note that
        $\max_i \sum_j A_{ij}
        $$= \max_i \sum_j e^{U^0_{ij}} \mu_i \nu_j \leq e^{\delta}$,
        $\max_j \sum_i A_{ij}
        \leq e^{\delta}$, and
        $\min_{(i,j) \in \EEE} A_{ij}
        = \min_{(i,j) \in \EEE} e^{U^0_{ij}} \mu_i \nu_j 
        \geq \mu_{\min} \nu_{\min} e^{-\delta}$.
        Thus by \autoref{prop:apx_pf_qtve_BV24:spectral_gap_lb},
        \begin{equation}
            \frac{\lambda_2(\LLL[A])}
            {\left( \max_i \sum_j A_{ij} \right) \wedge \left( \max_j \sum_i A_{ij} \right)} 
            \geq \mu_{\min} \nu_{\min}\, e^{-2\delta} \lambda_2(\LLL[\bmone_\EEE])
            \geq e^{-R (1+\delta)}.
        \end{equation}
        \item By \cite[Proposition~2.7]{wang2026almost}, there exists $(f^*, g^*) \in \argmin \Psi$ such that 
        \begin{align}
            \norm{f^*}_\infty, \norm{g^*}_\infty 
            &\leq R' \left[ -\min_{(i,j) \in \EEE} U^0_{ij}
            - \log (\mu_{\min} \vee \nu_{\min})
            + \max_{ij} U^0_{ij} \right]
            \leq R'' \, (1+\delta),
        \end{align}
        and by (a slight variation on the proof of) \cite[Lemma~1.7]{wang2026almost},
        the sequence $\big( \norm{f^k-f^*}_\infty \vee \norm{g^k-g^*}_\infty \big)_{k \geq 0}$ is non-increasing.
        So, since $(f^0, g^0) = (0,0)$,
        \begin{equation}
            \sup_{k \geq 0}
            \max_{ij} \abs{f^k_i+g^k_j}
            \leq \sup_{k \geq 0}
            \norm{(f^k,g^k)}_\infty
            \leq \norm{(f^*,g^*)}_\infty
            + \norm{(f^0-f^*,g^0-g^*)}_\infty
            \leq 4 R'' \, (1+\delta).
        \end{equation}
        This already shows the first inequality of the Corollary, as for all $k \geq 0$ and $(i,j) \in \EEE$,
        \begin{equation}
            \abs{U^k_{ij}}
            = \abs{(f^k_i+g^k_j-C_{ij})/\tau}
            = \abs{f^k_i + g^k_j + U^0_{ij}}
            \leq 4 R'' (1+\delta) + \delta.
        \end{equation}
        \item Let $\Phi(f,g) = \log \left( \sum_{ij} A_{ij} e^{f_i+g_j} \right) - \mu^\top f - \nu^\top g$.
        Then $\min \Phi = \min \Psi$ and
        $\Psi(f^1, g^1) = \Phi(f^1, g^1)$ since $\sum_{ij} \pi^k_{ij} = 1$ for all $k \geq 1$.
        One can also check that $\forall g, \argmin \Psi(\cdot, g) \subset \argmin \Phi(\cdot, g)$
        and so $\Phi(f^1, g^1) \leq \Phi(f^0, g^0) = \Phi(0,0)$.
        Thus
        \begin{equation}
            \Psi(f^1,g^1)
            \leq \Phi(0,0)
            = \log \sum_{ij} A_{ij}
            = \log \sum_{ij} e^{U^0_{ij}} \mu_i \nu_j
            \leq \max_{ij} U^0_{ij}
            \leq \delta.
        \end{equation}
        \item By duality, since the primal EOT problem has a feasible solution by the asymptotic scalability assumption $\ol\SSS = \EEE$,
        \begin{align}
            -\min \Psi
            = \max (-\Psi)
            &= \min_{\pi \in \Delta_\EEE} \sum_{ij} C_{ij} \pi_{ij} + \Hdiv{\pi}{\mu \otimes \nu} 
            ~~~~\text{subject to}~~~~
            X_\sharp \pi = \mu,~
            Y_\sharp \pi = \nu\\
            &\leq \max_{(i,j) \in \EEE} C_{ij} + \max_{\pi \in \Delta_{m \times n}} \Hdiv{\pi}{\mu \otimes \nu}
            \leq \delta - \log (\mu_{\min} \nu_{\min}).
        \end{align}
    \end{itemize}
    
    Thus by \autoref{prop:apx_pf_qtve_BV24:QGGU25},
    \begin{equation}
        \Psi(f^k, g^k) - \min \Psi
        \leq \left( 1 - e^{-R (1+\delta)} e^{-8 R'' (1+\delta)} \right)^{k-1}\!
        (2 \delta - \log(\mu_{\min} \nu_{\min}))
        \leq B (1+\delta) \left( 1 - e^{-R''' (1+\delta)} \right)^k.
    \end{equation}
    So, since
    $\Psi(f^k, g^k) \!-\! \Psi(f^{k+1}, g^{k+1}) = \Hdiv{\mu}{X_\sharp \pi^k} + \Hdiv{\nu}{Y_\sharp \pi^k}$
    for all $k \geq 1$ by \cite[Lemma~2]{altschuler2017near},
    \begin{align}
        \Hdiv{\mu}{X_\sharp \pi^k}
        \leq \abs{\Psi(f^k,g^k) - \min \Psi} + \abs{\Psi(f^{k+1},g^{k+1}) - \min \Psi}
        \leq 2 B (1+\delta) \left( 1 - e^{-R''' (1+\delta)} \right)^k.
    \end{align}
    Further, for any $k \geq 1$ even, $v^{k+1}_i = \log(\mu_i / (X_\sharp \pi^k)_i)$ and $(X_\sharp \pi^k)_i = \mu_i e^{-v^{k+1}_i} \geq \mu_i \nu_{\min} \geq \mu_{\min} \nu_{\min}$ by \autoref{lm:prelim:unifbound_vk_wk}, so by local Lipschitz-continuity of $\log$ and by Pinsker's inequality,
    \begin{align}
        \norm{v^{k+1}}_\infty
        \leq \frac{1}{\mu_{\min} \nu_{\min}} \norm{\mu-X_\sharp \pi^k}_1 
        &\leq \frac{1}{\mu_{\min} \nu_{\min}}
        \sqrt{4 B (1+\delta) \left( 1 - e^{-R''' (1+\delta)} \right)^k} \\
        &\leq B' \, (1+\delta) \left( 1 - e^{-R'''' (1+\delta)} \right)^k.
    \end{align}
    Hence the bound on the $v^k$, and the bound on the $w^k$ follows similarly.
\end{proof}

\subsection{Proof of \autoref{prop:prelim:qtve_BV24_exp}}

\begin{lemma} \label{lm:apx_pf_qtve_BV24:olvw_exp_rate}
    Suppose that $\SSS = \ol\SSS$.
    For any $k_0 \geq K_2 (1+\delta)$ even,
    let $\ol v^{(k_0) k}, \ol w^{(k_0) k}, \ol U^{(k_0) k}$ be as in \autoref{lm:apx_pf_qtve_BV24:control_E_olS}
    (where $K_2$ is defined as in \autoref{lm:apx_pf_qtve_BV24:control_E_olS}).
    Then
    \begin{equation}
        \forall k \geq k_0+4,~~
        \norm{\ol v^{(k_0) k} - v^*}_\infty,~
        \norm{\ol w^{(k_0) k} - w^*}_\infty
        \leq B_8 (1+\delta) \left( 1-e^{-R_8 (1+\delta)} \right)^{k-k_0}
    \end{equation}
    for some constants $B_8, R_8>0$ dependent only on $\mu, \nu$, and $\EEE$.
\end{lemma}

\begin{proof}
    In this proof, we write $B_2, K_2$ for the constants introduced in \autoref{lm:apx_pf_qtve_BV24:control_E_olS}
    and $B_7, R_7$ for the ones introduced in \autoref{coroll:apx_pf_qtve_BV24:QGGU25},
    and we use $B, B', B'', ...$ to denote constants dependent only on $\mu, \nu, \EEE$ without introducing them each time.

    Let us show that
    $\max_{(i,j) \in \ol\SSS} \abs{U^{k_0}_{ij}}$ 
    is bounded by a constant independent of $k_0$.
    For this, let $k_2 = \ceil{K_2 (1+\delta)} \leq k_0$ and consider the iterates $\ol v^{(k_2) k}, \ol w^{(k_2) k}, \ol U^{(k_2) k}$ defined as in \autoref{lm:apx_pf_qtve_BV24:control_E_olS}.
    Then 
    \begin{equation}
        \max_{(i,j) \in \ol\SSS} \abs{\ol U^{(k_2)k_0}_{ij} - U^{k_0}_{ij}} \leq B_2
    \end{equation}
    by \autoref{lm:apx_pf_qtve_BV24:control_E_olS}.
    Now we claim that $\max_{(i,j) \in \ol\SSS} \abs{\ol U^{(k_2)k}_{ij}}$ is bounded uniformly over all $k \geq k_2$.
    Indeed, at its initialization,
    \begin{align}
        \forall (i,j) \in \ol\SSS,~
        \abs{\ol U^{(k_2) k_2)}_{ij}}
        = \abs{U^{k_2}_{ij}}
        &= \abs{ 
            U^0_{ij}
            + v^1_i + w^2_j
            + \sum_{\substack{l=2\\l~\text{even}}}^{k_2-1}
            v^{l+1}_i
            + \sum_{\substack{l=3\\l~\text{odd}}}^{k_2-1}
            w^{l+1}_i 
        } \\ 
        &\leq 4 \delta + k_2 \, [-\log(\mu_{\min} \wedge \nu_{\min})]
        \leq B (1+\delta),
    \end{align}
    and so by the first part of \autoref{coroll:apx_pf_qtve_BV24:QGGU25} applied with $\mu^*, \nu, \ol\SSS, \ol U^{(k_2)k_2}$ in place of ``$\mu, \nu, \EEE, U^0$'',
    we have $\forall k \geq k_2,~ \max_{(i,j) \in \ol\SSS} \abs{\ol U^{(k_2) k}_{ij}} \leq B_7\, (1 + B(1+\delta))$.
    This shows that, as we claimed,
    \begin{equation}
        \forall k \geq k_2,~~
        \max_{(i,j) \in \ol\SSS}
        \abs{\ol U^{(k_2) k}_{ij}}
        \leq B' (1+\delta),
    \end{equation}
    and in particular this estimate holds for $k=k_0$.
    Stringing together the above inequalities, we get as announced
    \begin{equation}
        \max_{(i,j) \in \ol\SSS} \abs{U^{k_0}_{ij}}
        \leq B_2 + \max_{(i,j) \in \ol\SSS} \abs{\ol U^{(k_2) k_0}_{ij}}
        \leq B_2 + B' (1+\delta)
        = B'' (1+\delta)
    \end{equation}
    uniformly in $k_0$.

    Thus we have
    \begin{equation}
        \forall (i,j) \in \ol\SSS,~
        -B'' (1+\delta) 
        \leq \ol U^{(k_0)k_0}_{ij} = U^{k_0}_{ij} 
        \leq B'' (1+\delta).
    \end{equation}
    The announced bound then follows by reasoning as in the proof of \autoref{lm:apx_pf_qtve_BV24:asymp_scalable_logk/k}, applying \autoref{coroll:apx_pf_qtve_BV24:QGGU25} with $\mu^*, \nu, \ol\SSS, \tU^{k_0}$ in place of ``$\mu, \nu, \EEE, U^0$'' in the corollary statement to quantify the convergence of the~$\tv^k$
    (for $\tv^k, \tw^k, \tU^k$ being defined as in \autoref{lm:apx_pf_qtve_BV24:prop5.3_BV24}).
\end{proof}

Finally, we can now present the proof of \autoref{prop:prelim:qtve_BV24_exp}.

\begin{proof}[Proof of \autoref{prop:prelim:qtve_BV24_exp}]
    Let $K_2, B_2, \gamma$ be as in \autoref{lm:apx_pf_qtve_BV24:control_E_olS} and 
    let $B_8, R_8$ be as in \autoref{lm:apx_pf_qtve_BV24:asymp_scalable_logk/k}.
    Fix any $k \geq 2 K_2 (1+\delta) + 8$.
    Consider any $\Delta > 0$ such that $k_0 = K_2 (1+\delta) + \Delta$ is even and $k_0 + 4 \leq k$,
    and let $(\ol v^{(k_0)l}, \ol w^{(k_0)l}, \ol U^{(k_0)l})_{l \geq k_0}$ be defined as in \autoref{lm:apx_pf_qtve_BV24:control_E_olS}.
    Then by \autoref{lm:apx_pf_qtve_BV24:control_E_olS} for the first term and by \autoref{lm:apx_pf_qtve_BV24:olvw_exp_rate} for the second term,
    \begin{equation}
        \norm{v^k-v^*}_\infty
        \leq \norm{v^k-\ol v^{(k_0) k}}_\infty
        + \norm{\ol v^{(k_0) k} - v^*}_\infty
        \leq B_2 e^{-\Delta \gamma/4}
        + B_8 (1+\delta) \left( 1 - e^{-R_8 (1+\delta)} \right)^{k-k_0}.
    \end{equation}
    
    Take $\Delta = k/2 - r$, where $0 \leq r \leq 2$ serves to ensure that $k_0 = K_2(1+\delta)+\Delta$ is an even integer.
    Let us check that this choice also satisfies the condition $k_0+4 \leq k$, i.e.,
    \begin{equation}
        K_2 (1+\delta) + \Delta + 4 
        = K_2 (1+\delta) + k/2 - r + 4 
        \leq k.
    \end{equation}
    This is indeed true by our assumption that $k \geq 2 K_2 (1+\delta)+8$.

    Plugging this choice of $\Delta$ into the bound on $\norm{v^k-v^*}_\infty$ above, we obtain
    \begin{align}
        \norm{v^k-v^*}_\infty
        &\leq B_2 e^{-k \gamma/8} e^{r \gamma/4}
        + B_8 (1+\delta) 
        \left( 1 - e^{-R_8 (1+\delta)} \right)^{k/2+r}
        \left( 1 - e^{-R_8 (1+\delta)} \right)^{-K_2 (1+\delta)}
        \\
        &\leq B (1+\delta) \left( 1-e^{-R(1+\delta)} \right)^k
    \end{align}
    for some constants $B, R$ dependent only on $\mu, \nu$, and $\EEE$, since
    $\sup_{\delta > 0} \left( 1 - e^{-R_8 (1+\delta)} \right)^{-K_2 (1+\delta)}$ is finite and dependent only on $R_8, K_2$.
    This proves the announced bound for the $v^k$, and the bound for the $w^k$ follows similarly.
\end{proof}

\fi

\end{document}